\documentclass{daj}

\usepackage{siunitx}
\usepackage{amssymb}
\usepackage{amsmath}
\usepackage{amsthm}
\usepackage{hyperref}
\usepackage{mathtools}

\newcommand{\E}{\mathbb{E}}

\newcommand{\X}{\mathbf{X}} 
\newcommand{\Y}{\mathbf{Y}}

\newcommand\st{{\operatorname{st}}}

\theoremstyle{plain}

\newtheorem{theorem}{Theorem}[section]
\newtheorem{proposition}[theorem]{Proposition}

\newtheorem{lemma}[theorem]{Lemma}

\newtheorem{corollary}[theorem]{Corollary}

\theoremstyle{definition}

\newtheorem{definition}[theorem]{Definition}
\newtheorem{remark}[theorem]{Remark}

\newtheorem{example}[theorem]{Example}
\newtheorem{examples}[theorem]{Examples}

\newcommand\F{\mathbb{F}}
\newcommand\R{\mathbb{R}}
\newcommand\Z{\mathbb{Z}}
\newcommand\N{\mathbb{N}}
\newcommand\C{\mathbb{C}}

\newcommand\eps{\varepsilon}
\newcommand\ZZ{\mathbf{Z}}

\dajAUTHORdetails{%
  title = {Concatenation Theorems for Anti-Gowers-Uniform Functions and Host-Kra Characteristic Factors}, 
  author = {Terence Tao and Tamar Ziegler},
  plaintextauthor = {Terence Tao, Tamar Ziegler},
  plaintexttitle = {Concatenation Theorems for Anti-Gowers-Uniform Functions and Host-Kra Characteristic Factors}, 
  runningtitle = {Concatenation theorems}, 
  keywords = {concatenation theorems, Gowers-Host-Kra seminorms, dual functions},
}   

\dajEDITORdetails{%
   year={2016},
   number={13},
   received={28 March 2016},   
   published={30 July 2016},  
   doi={10.19086/da.850},       
}   

\begin{document}

\begin{frontmatter}[classification=text]

\title{Concatenation Theorems for Anti-Gowers-Uniform Functions and Host-Kra Characteristic Factors} 

\author[tt]{Terence Tao\thanks{TT is supported by NSF grant DMS-0649473 and by a Simons Investigator Award.}}
\author[tz]{Tamar Ziegler\thanks{TZ is supported by ISF grant 407/12.}}

\begin{abstract}
We establish a number of ``concatenation theorems'' that assert, roughly speaking, that if a function exhibits ``polynomial'' (or ``Gowers anti-uniform'', ``uniformly almost periodic'', or ``nilsequence'') behaviour in two different directions separately, then it also exhibits the same behavior (but at higher degree) in both directions jointly.  Among other things, this allows one to control averaged local Gowers uniformity norms by global Gowers uniformity norms.  In a sequel to this paper, we will apply such control to obtain asymptotics for ``polynomial progressions'' $n+P_1(r),\dots,n+P_k(r)$ in various sets of integers, such as the prime numbers.
\end{abstract}
\end{frontmatter}

\section{Introduction}

\subsection{Concatenation of polynomiality}

Suppose $P\colon \Z^2 \to \R$ is a function with the property that $n \mapsto P(n,m)$ is an (affine-)linear function of $n$ for each $m$, and $m \mapsto P(n,m)$ is an (affine-)linear function of $m$ for each $n$.  Then it is easy to see that $P$ is of the form
\begin{equation}\label{pnm}
 P(n,m) = \alpha nm + \beta n + \gamma m + \delta
\end{equation}
for some coefficients $\alpha,\beta,\gamma,\delta \in \R$.  In particular, $(n,m) \mapsto P(n,m)$ is a polynomial of degree at most $2$.

The above phenomenon generalises to higher degree polynomials.  Let us make the following definition:

\begin{definition}[Polynomials]\label{poly-def}  Let $P\colon G \to K$ be a function from one additive group $G = (G,+)$ to another $K = (K,+)$.  For any subgroup $H$ of $G$ and for any integer $d$, we say that $P$ is a \emph{polynomial of degree $<d$ along $H$} according to the following recursive definition\footnote{Our conventions here are somewhat nonstandard, but are chosen so that the proofs of our main theorems, such as Theorem \ref{concat-uap}, resemble the proofs of warmup results such as Proposition \ref{concat}.}:
\begin{itemize}
\item[(i)]  If $d \leq 0$, we say that $P$ is of degree $<d$ along $H$ if and only if it is identically zero.
\item[(ii)]  If $d \geq 1$, we say that $P$ is of degree $<d$ along $H$ if and only if, for each $h \in H$, there exists a polynomial $P_h: G \to K$ of degree $<d-1$ along $H$ such that 
\begin{equation}\label{pxh}
P(x+h) = P(x) + P_h(x)
\end{equation}
for all $x \in G$.
\end{itemize}
\end{definition}

We then have

\begin{proposition}[Concatenation theorem for polynomials]\label{concat}  Let $P\colon G \to K$ be a function from one additive group $G$ to another $K$, let $H_1,H_2$ be subgroups of $G$, and let $d_1,d_2$ be integers.  Suppose that
\begin{itemize}
\item[(i)] $P$ is a polynomial of degree ${<}d_1$ along $H_1$.
\item[(ii)] $P$ is a polynomial of degree ${<}d_2$ along $H_2$.
\end{itemize}
Then $P$ is a polynomial of degree ${<}d_1+d_2-1$ along $H_1+H_2$.
\end{proposition}

The degree bound here is sharp, as can for instance be seen using the example $P\colon \Z \times \Z \to \R$ of the monomial $P(n,m) \coloneqq  n^{d_1-1} m^{d_2-1}$, with $H_1 \coloneqq  \Z \times \{0\}$ and $H_2 \coloneqq  \{0\} \times \Z$.

\begin{proof}  The claim is trivial if $d_1 \leq 0$ or $d_2 \leq 0$, so we suppose inductively that $d_1,d_2 \geq 1$ and that the claim has already been proven for smaller values of $d_1+d_2$.

Let $h_1 \in H_1$ and $h_2 \in H_2$.  By (i), there is a polynomial $P_{h_1}\colon G \to K$ of degree ${<}d_1-1$ along $H_1$ such that
\begin{equation}\label{px}
 P(x+h_1) = P(x) + P_{h_1}(x)
\end{equation}
for all $x \in G$.  Similarly, by (ii), there is a polynomial $P_{h_2}\colon G \to K$ of degree ${<}d_2-1$ along $H_2$ such that
\begin{equation}\label{py}
 P(x+h_2) = P(x) + P_{h_2}(x)
\end{equation}
for all $x \in G$.  Replacing $x$ by $x+h_1$ and combining with \eqref{px}, we see that
\begin{equation}\label{pz}
 P(x+h_1+h_2) = P(x) + P_{h_1,h_2}(x)
\end{equation}
for all $x \in G$, where
\begin{equation}\label{phh}
 P_{h_1,h_2}(x) \coloneqq  P_{h_1}(x) + P_{h_2}(x+h_1).
\end{equation}
Since $P$ is of degree ${<}d_2$ along $H_2$, and $P_{h_1}$ is the difference of two translates of $P$, we see that $P_{h_1}$ is of degree ${<}d_2$ along $H_2$, in addition to being of degree ${<}d_1-1$ along $H_1$.  By induction hypothesis, $P_{h_1}$ is thus of degree ${<}d_1+d_2-2$ along $H_1+H_2$.  A similar argument shows that $P_{h_2}$ is of degree ${<}d_1+d_2-2$ along $H_1+H_2$, which implies from \eqref{phh} that $P_{h_1,h_2}$ is also of degree ${<}d_1+d_2-2$ along $H_1+H_2$.  This implies that $P$ is of degree ${<}d_1+d_2-1$ along $H_1+H_2$ as required.
\end{proof}

Note how this proposition ``concatenates'' the polynomial structure along the direction $H_1$ with the polynomial structure along the direction $H_2$ to obtain a (higher degree) polynomial structure along the combined direction $H_1+H_2$ variable.  In this paper we develop a number of further concatenation theorems in this spirit, with an eye towards applications in ergodic theory, additive combinatorics, and analytic number theory.  The proofs of these theorems will broadly follow the same induction strategy that appeared in the above proof of Proposition \ref{concat}, but there will be significant additional technical complications.

\subsection{Concatenation of low rank}

We begin with a variant of Proposition \ref{concat} in which the concept of a polynomial is replaced by that of a \emph{low rank function}. 

\begin{definition}[Rank]\label{rank-def} Given a function $P: G \to K$ between additive groups $G,K$, as well as subgroups $H_1,\dots,H_d$ of $G$ for some non-negative integer $d$, we say that $P$ has \emph{rank ${<}(H_1,\dots,H_d)$} or \emph{rank ${<}(H_i)_{1 \leq i \leq d}$} according to the following recursive definition:
\begin{itemize}
\item[(i)]  If $d = 0$, we say that $P$ is of rank ${<}()$ if and only if it is identically zero.
\item[(ii)]  If $d \geq 1$, we say that $P$ is of rank ${<}(H_1,\dots,H_d)$ if and only if, for every $1 \leq i \leq d$ and $h \in H_i$, there exists a polynomial $P_h\colon G \to K$ of rank ${<}(H_1,\dots,H_{i-1},H_{i+1},\dots,H_d)$ such that $P(x+h) = P(x) + P_h(x)$ for all $x \in G$.
\end{itemize}
\end{definition}

\begin{examples} The property of being a polynomial of degree ${<}d$ along $H$ is the same as having rank ${<}(H)_{1 \leq i \leq d}$.  As another example, a function $P: \Z^2 \to \R$ is of rank ${<}(\R \times \{0\}, \{0\} \times \R)$ if and only if it takes the form
$$ P(n,m) = f(n) + g(m)$$
for some functions $f,g: \Z \to \R$; a function $P: \Z^3 \to \R$ is of rank ${<}(\R \times \{0\} \times \{0\}, \{0\} \times \R \times \{0\}, \{0\} \times \{0\} \times \R)$ if and only if it takes the form
$$ P(n,m,k) = f(n,m) + g(n,k) + h(m,k)$$
for some functions $f,g,h: \Z^2 \to \R$; and so forth.
\end{examples}

\begin{proposition}[Concatenation theorem for low rank functions]\label{concat-lowrank}  Let $P\colon G \to K$ be a function between additive groups, let $d_1$, $d_2$ be non-negative integers, and let $H_{1,1}$, $\dots$, $H_{1,d_1}$, $H_{2,1}$, $\dots$, $H_{2,d_2}$ be subgroups of $G$.  Suppose that
\begin{itemize}
\item[(i)] $P$ has rank ${<}(H_{1,i})_{1 \leq i \leq d_1}$.
\item[(ii)] $P$ has rank ${<}(H_{2,j})_{1 \leq j \leq d_2}$.
\end{itemize}
Then $P$ has rank ${<}(H_{1,i}+H_{2,j})_{1 \leq i \leq d_1; 1 \leq j \leq d_2}$.
\end{proposition}

In this case, the concatenated rank $d_1 d_2$ is the \emph{product} of the individual ranks $d_1,d_2$, in contrast to the concatenation in Proposition \ref{concat} which \emph{adds} the degrees together.  Thus this proposition is inferior to Proposition \ref{concat} in the case when $d_1,d_2 > 1$, and $H_{1,i} = H_1$ and $H_{2,j} = H_2$ for all $i=1,\dots,d_1$ and $j=1,\dots,d_2$. However, in other cases the dependence on parameters can be optimal.  For instance, if $G = \Z^{d_1} \times \Z^{d_2}$, and we take $H_{1,1}$, $\dots$, $H_{1,d_1}$, $H_{2,1}$, $\dots$, $H_{2,d_2}$ to be the cyclic groups generated by the standard basis elements $e_{1,1}$, $\dots$, $e_{1,d_1}$, $e_{2,1}$, $\dots$, $e_{2,d_2}$ of $G$ respectively, and take $P$ to be a function of the form
\begin{align*}
&P(n_{1,1},\dots,n_{1,d_1},n_{2,1},\dots,n_{2,d_2}) \\
&\quad \coloneqq  \sum_{i=1}^{d_1} \sum_{j=1}^{d_2}
f_{i,j}( n_{1,1},\dots, n_{1,i-1}, n_{1,i+1}, \dots, n_{1,d_1} ) g_{i,j}( n_{2,1},\dots, n_{2,j-1}, n_{2,j+1}, \dots, n_{2,d_2} ) 
\end{align*}
for arbitrary functions $f_{i,j}\colon \Z^{d_1-1} \to \R$, $g_{i,j}\colon \Z^{d_2-1} \to \R$, then one verifies the hypotheses (i) and (ii) of Proposition \ref{concat-lowrank}, but in general one does not expect any better low rank properties for $P$ with regards to the groups $H_{1,i}+H_{2,j}$ than the one provided by that proposition.

\begin{proof}  The claim is again trivial when $d_1=0$ or $d_2=0$, so we may assume inductively that $d_1,d_2 \geq 1$ and that the claim has already been proven for smaller values of $d_1+d_2$.

Let $1 \leq i_0 \leq d_1, 1 \leq j_0 \leq d_2$, $h_1 \in H_{1,i_0}$ and $h_2 \in H_{2,j_0}$ be arbitrary.  As before we have the identities \eqref{px}, \eqref{py}, \eqref{pz} for all $x \in G$ and some functions $P_{h_1}\colon G \to K$, $P_{h_2}\colon G \to K$, with $P_{h_1,h_2}$ defined by \eqref{phh}.  By (i), the function $P_{h_1}$ has rank 
${<}(H_{1,i})_{1 \leq i \leq d_1; i \neq i_0}$; from (ii) it also has rank ${<}(H_{2,j})_{1 \leq j \leq d_2}$.  By induction hypothesis, $P_{h_1}$ thus has rank ${<}(H_{1,i}+H_{2,j})_{1 \leq i \leq d_1; 1 \leq j \leq d_2; i \neq i_0}$, which in particular implies that $P_{h_1}$ has rank ${<}(H_{1,i}+H_{2,j})_{1 \leq i \leq d_1; 1 \leq j \leq d_2; (i,j) \neq (i_0,j_0)}$.  Similarly for $P_{h_2}$, which by \eqref{phh} implies that $P_{h_1,h_2}$ also has rank
${<}(H_{1,i}+H_{2,j})_{1 \leq i \leq d_1; 1 \leq j \leq d_2; (i,j) \neq (i_0,j_0)}$.  This implies that $P$ has rank ${<}(H_{1,i}+H_{2,j})_{1 \leq i \leq d_1; 1 \leq j \leq d_2}$ as required.
\end{proof}

\subsection{Concatenation of anti-uniformity}

Now we turn to a more non-trivial variant of Proposition \ref{concat}, in which the property of being polynomial in various directions is replaced by that of being \emph{anti-uniform} in the sense of being almost orthogonal to Gowers uniform functions.  To make this concept precise, and to work at a suitable level of generality, we need some notation.  Recall that a \emph{finite multiset} is the same concept as a finite set, but in which elements are allowed to appear with multiplicity.

\begin{definition}[Sumset and averaging]  If $A \coloneqq  \{a_1,\dots,a_n\}$ and $B \coloneqq  \{b_1,\dots,b_m\}$ are finite non-empty multisets in an additive group $G = (G,+)$, we define the sumset $A+B \coloneqq  \{ a_i + b_j: 1 \leq i \leq n, 1 \leq j \leq m \}$ and difference set $A-B \coloneqq  \{a_i-b_j: 1 \leq i \leq n, 1 \leq j \leq m \}$, where multiplicity is counted; thus for instance we always have $|A+B|=|A-B| = |A| |B|$, where $|A|$ denotes the cardinality of $A$, counting multiplicity.  Note that $A+B$ or $A-B$ may contain multiplicity even if $A$ and $B$ does not, for instance $\{1,2\}+\{1,2\}=\{2,3,3,4\}$.  If $f\colon G \to \C$ is a function, we use the averaging notation
$$ \E_{a \in A} f(a) \coloneqq  \frac{1}{|A|} \sum_{a \in A} f(a)$$
where the sum $\sum_{a \in A}$ is also counting multiplicity.  For instance, $\E_{a \in \{1,2,2\}} a = \frac{5}{3}$.
\end{definition}

\begin{definition}[$G$-system]\label{gsys-def}  Let $G = (G,+)$ be an at most countable additive group.  A \emph{$G$-system} $(\X, T)$ is a probability space $\X = (X, {\mathcal B}, \mu)$, together with a collection $T = (T^g)_{g \in G}$ of invertible measure-preserving maps $T^g\colon X \to X$, such that $T^0$ is the identity and $T^{g+h} = T^g T^h$ for all $g,h \in G$.  For technical reasons we will require that the probability space $\X$ is countably generated modulo null sets (or equivalently, that the Hilbert space $L^2(\X)$ is separable).  Given a measurable function $f\colon \X \to \C$ and $g \in G$, we define $T^g f \coloneqq  f \circ T^{-g}$.  We shall often abuse notation and abbreviate $(\X,T)$ as $\X$.  
\end{definition}

\begin{remark} As it turns out, a large part of our analysis would be valid even when $G$ was an uncountable additive group (in particular, no amenability hypothesis on $G$ is required); however the countable case is the one which is the most important for applications, and so we shall restrict to this case to avoid some minor technical issues involving measurability.  Once the group $G$ is restricted to be countable, the requirement that $\X$ is countably generated modulo null sets is usually harmless in applications, as one can often easily reduce to this case.  In combinatorial applications, one usually works with the case when $G$ is a finite group, and $\X$ is $G$ with the uniform probability measure and the translation action $T^g x \coloneqq x+g$, but for applications in ergodic theory, and also because we will eventually apply an ultraproduct construction to the combinatorial setting, it will be convenient to work with the more general setup in Definition \ref{gsys-def}.
\end{remark}

\begin{definition}[Gowers uniformity norm]\label{gaw}  Let $G$ be an at most countable additive group, and let $(\X, T)$ be a $G$-system.  If $f \in L^\infty(\X)$, $Q$ is a non-empty finite multiset and $d$ is a positive integer, we define the \emph{Gowers uniformity (semi-)norm} $\| f \|_{U^d_Q(\X)}$ by the formula
\begin{equation}\label{boxnorm-def}
 \|f\|_{U^d_Q(\X)}^{2^d} \coloneqq  \E_{h_1,\dots,h_d,h'_1,\dots,h'_d \in Q} \int_X \Delta_{h_1,h'_1} \dots \Delta_{h_d,h'_d} f\ d\mu
\end{equation}
where $\Delta_{h,h'}$ is the nonlinear operator
\begin{equation}\label{deltah-def}
 \Delta_{h,h'} f \coloneqq  (T^h f) \overline{(T^{h'} f)}.
\end{equation}
More generally, given $d$ non-empty finite multisets $Q_1,\dots,Q_d$, we define the \emph{Gowers box (semi-)norm}
$\| f \|_{\Box^d_{Q_1,\dots,Q_d}(\X)}$ by the formula
$$ \|f\|_{\Box^d_{Q_1,\dots,Q_d}(\X)}^{2^d} \coloneqq  \E_{h_i,h'_i \in Q_i \forall i=1,\dots,d} \int_X \Delta_{h_1,h'_1} \dots \Delta_{h_d,h'_d} f\ d\mu,$$
so in particular
$$ \| f \|_{U^d_Q(\X)} = \|f\|_{\Box^d_{Q,\dots,Q}(\X)}.$$
Note that the $\Delta_{h,h'}$ commute with each other, and so the ordering of the $Q_i$ is irrelevant.  
\end{definition}

It is well known that $\| f \|_{U^d_Q(\X)}$ and $\| f \|_{\Box^d_{Q_1,\dots,Q_d}(\X)}$ are indeed semi-norms; see \cite[Appendix B]{gt-linear}.  We will be primarily interested in the Gowers norms for a specific type of multiset $Q$, namely the \emph{coset progressions} that arise in additive combinatorics (see \cite{tao-vu}).  The original Gowers uniformity norms from  \cite{gow-szem-4, gow-szem} correspond to the case when $Q = G = \Z/N\Z$ is a finite cyclic group, and $\X = G$ is equipped with the uniform probability measure and the translation action of $G$.

\begin{definition}[Coset progression]  A (symmetric) \emph{coset progression} is a (multi-)subset\footnote{Strictly speaking, one should refer to the tuple $(Q, H, r, (v_i)_{i=1}^r, (N_i)_{i=1}^r)$ as the coset progression, rather than just the multi-set $Q$, as one cannot define key concepts such as rank or the dilates $\eps Q$ without this additional data.  However, we shall abuse notation and use the multi-set $Q$ as a metonym for the entire coset progression.} $Q$ of an additive group $G$ of the form
$$ Q = H + \{n_1 v_1 + \dots + n_r v_r\colon |n_i| \leq N_i \forall i=1,\dots,r\}$$
where $H$ is a finite subgroup of $G$, $r \geq 0$ is a non-negative integer (which we call the \emph{rank} of $Q$), $v_1,\dots,v_r$ are elements of $G$, and $N_1,\dots,N_r > 0$ are real numbers.  Here we count the sums $n_1 v_1 + \dots + n_r v_r$ with multiplicity, and we view $H$ as a multiset in which each element has an equal multiplicity (not necessarily $1$).  Given any $\eps > 0$, we define the \emph{dilate} $\eps Q$ of $Q$ by the formula
$$ \eps Q \coloneqq  H + \{n_1 v_1 + \dots + n_r v_r: |n_i| \leq \eps N_i \forall i=1,\dots,r\}.$$
\end{definition}

Note with our conventions on multiplicity, that if $Q_1, Q_2$ are coset progressions of rank $r_1,r_2$ respectively, then $Q_1+Q_2$ is a coset progression of rank $r_1+r_2$.  The regime of interest for this paper is when the coset progressions involved have bounded rank $r$, but their dimensions $N_i$ and torsion group $H$ are allowed to be arbitrarily large.

One can form dual (semi-)norms in the usual fashion, for instance setting
$$ \|f\|_{U^d_Q(\X)^*} \coloneqq  \sup \{ |\langle f, g \rangle_{L^2(\X)}|: g \in L^\infty(\X), \|g\|_{U^d_Q(\X)} \leq 1 \}$$
for all $f \in L^\infty(\X)$, where
$$ \langle f, g \rangle_{L^2(\X)} \coloneqq  \int_X f \overline{g}\ d\mu.$$
However, for technical reasons we will need to use a variant of these norms, namely the quantities
\begin{equation}\label{upqd}
 \|f\|_{U^d_Q(\X)^*,\eps} \coloneqq  \sup \{ |\langle f, g \rangle_{L^2(\X)}|: g \in L^\infty(\X), \|g\|_{L^\infty(\X)} \leq 1, \|g\|_{U^d_{\eps Q}(\X)} \leq \eps \}
\end{equation}
and more generally
$$ \|f\|_{\Box^d_{Q_1,\dots,Q_d}(\X)^*,\eps} \coloneqq  \sup \{ |\langle f, g \rangle_{L^2(\X)}|: g \in L^\infty(\X), \|g\|_{L^\infty(\X)} \leq 1, \|g\|_{\Box^d_{\eps Q_1, \dots, \eps Q_d}(\X)} \leq \eps \}.$$

Our first main theorem is analogous to Proposition \ref{concat}, and is stated as follows.

\begin{theorem}[Concatenation theorem for anti-uniformity norms]\label{concat-uap}  Let $Q_1, Q_2$ be coset progressions of ranks $r_1,r_2$ respectively in an at most countable additive group $G$, let $(\X,T)$ be a $G$-system, let $d_1,d_2$ be positive integers, and let $f$ lie in the closed unit ball of $L^\infty(\X)$.  Let $c_1,c_2\colon (0,+\infty) \to (0,+\infty)$ be functions such that $c_i(\eps) \to 0$ as $\eps \to 0$ for $i=1,2$.  We make the following hypotheses:
\begin{itemize}
\item[(i)] $\| f \|_{U^{d_1}_{Q_1}(\X)^*, \eps} \leq c_1(\eps)$ for all $\eps > 0$.
\item[(ii)] $\| f \|_{U^{d_2}_{Q_2}(\X)^*, \eps} \leq c_2(\eps)$ for all $\eps > 0$.
\end{itemize}
Then there exists a function $c\colon (0,\infty) \to (0,+\infty)$ with $c(\eps) \to 0$ as $\eps \to 0$, which depends only on $d_1,d_2,r_1,r_2,c_1,c_2$, such that
$$ \| f \|_{U^{d_1+d_2-1}_{Q_1+Q_2}(\X)^*, \eps} \leq c(\eps)$$
for all $\eps>0$.
\end{theorem}

Heuristically, hypothesis (i) (resp. (ii)) is asserting that $f$ ``behaves like'' a function of the form $x \mapsto e^{2\pi i P(x)}$ for some function $P\colon \X \to \R/\Z$ that ``behaves like'' a polynomial of degree ${<}d_1$ along $Q_1$ (resp. ${<}d_2$ along $Q_2$), thus justifying the analogy between Theorem \ref{concat-uap} and  Proposition \ref{concat}.  The various known \emph{inverse theorems for the Gowers norms} \cite{gt-u3, btz, gtz-u4, gtz-uk, szegedy} make this heuristic more precise in some cases; however, our proof of the above theorem does not require these (difficult) theorems (which are currently unavailable for the general Gowers box norms).

We prove Theorem \ref{concat-uap} in Sections \ref{dual-sec}, \ref{boxproof}, after translating these theorems to a nonstandard analysis setting in Section \ref{nonst-sec}.  The basic idea is to use the hypothesis (i) to approximate $f$ by a linear combination of ``dual functions'' along the $Q_1$ direction, and then to use (ii) to approximate the arguments of those dual functions in turn by dual functions along the $Q_2$ direction.  This gives a structure analogous to the identities \eqref{px}-\eqref{phh} obeyed by the function $P$ considered in Proposition \ref{concat}, and one then uses an induction hypothesis to conclude.  To obtain the desired approximations, one could either use structural decomposition theorems (as in \cite{gow}) or nonstandard analysis (as is used for instance in \cite{gtz-uk}).  We have elected to use the latter approach, as the former approach becomes somewhat messy due to the need to keep quantitative track of a number of functions such as $\eps \mapsto c(\eps)$, whereas these functions are concealed to the point of invisibility by the nonstandard approach.  We give a more expanded sketch of Theorem \ref{concat-uap} in Section \ref{sketch-sec} below.

\begin{remark}
It may be possible to establish a version of Theorem \ref{concat-uap} in which one does not shrink the coset progressions $Q$ by a small parameter $\eps$, so that the appearance of $\eps Q$ in \eqref{upqd} is replaced by $Q$.  This would give the theorem a more ``combinatorial'' flavor, as opposed to an ``ergodic'' one (if one views the limit $\eps \to 0$ as being somewhat analogous to the ergodic limit $n \to \infty$ of averaging along a F{\o}lner sequence $\Phi_n$).  Unfortunately, our methods rely heavily on techniques such as the van der Corput inequality, which reflects the fact that $Q$ is almost invariant with respect to translations in $\eps Q$ when $\eps$ is small.  As such, we do not know how to adapt our methods to remove this shrinkage of the coset progressions $Q$.  Similarly for Theorem \ref{concat-box} below.
\end{remark}

We also have an analogue of Proposition \ref{concat-lowrank}:

\begin{theorem}[Concatenation theorem for anti-box norms]\label{concat-box}  Let $d_1,d_2$ be positive integers.  For any $i=1,2$ and $1 \leq j \leq d_i$, let $Q_{i,j}$ be a coset progression of rank $r_{i,j}$ in an at most countable additive group $G$.  Let $(\X,T)$ be a $G$-system, let $d_1,d_2$ be positive integers, and let $f$ lie in the unit ball of $L^\infty(\X)$.  Let $c_1,c_2\colon (0,+\infty) \to (0,+\infty)$ be functions such that $c_i(\eps) \to 0$ as $\eps \to 0$ for $i=1,2$.  We make the following hypotheses:
\begin{itemize}
\item[(i)] $\| f \|_{\Box^{d_1}_{Q_{1,1},\dots,Q_{1,d_1}}(\X)^*, \eps} \leq c_1(\eps)$ for all $\eps > 0$.
\item[(ii)] $\| f \|_{\Box^{d_2}_{Q_{2,1},\dots,Q_{2,d_2}}(\X)^*, \eps} \leq c_2(\eps)$ for all $\eps > 0$.
\end{itemize}
Then there exists a function $c\colon (0,\infty) \to (0,+\infty)$ with $c(\eps) \to 0$ as $\eps \to 0$, which depends only on $d_1,d_2,c_1,c_2$ and the $r_{1,j}, r_{2,j}$, such that
$$ \| f \|_{\Box^{d_1 d_2}_{(Q_{1,i}+Q_{2,j})_{1 \leq i \leq d_1, 1 \leq j \leq d_2}}(\X)^*, \eps} \leq c(\eps)$$
for all $\eps>0$.
\end{theorem}

The proof of Theorem \ref{concat-box} is similar to that of Theorem \ref{concat-uap}, and is given at the end of Section \ref{boxproof}.

\subsection{Concatenation of characteristic factors}

Analogues of the above results can be obtained for characteristic factors of the Gowers-Host-Kra seminorms \cite{hk} in ergodic theory.  To construct these factors for arbitrary abelian group actions (including uncountable ones), it is convenient to introduce the following notation (which should be viewed as a substitute for the machinery of F{\o}lner sequences that does not require amenability).  Given an additive group $H$, we consider the set ${\mathcal F}[H]$ of non-empty finite multisets $Q$ in $H$.  We can make ${\mathcal F}[H]$ a directed set by declaring $Q_1 \leq Q_2$ if one has $Q_2 = Q_1 + R$ for some non-empty finite multiset $R$; note that any two $Q_1,Q_2$ have a common upper bound $Q_1+Q_2$. One can then define convergence along nets in the usual fashion: given a sequence of elements $x_Q$ of a Hausdorff topological space indexed by the non-empty finite multisets $Q$ in $H$, we write $\lim_{Q \to H} x_Q = x$ if for every neighbourhood $U$ of $x$, there is a finite non-empty multiset $Q_0$ in $H$ such that $x_Q \in U$ for all $Q \geq Q_0$.  Similarly one can define joint limits $\lim_{(Q_1,\dots,Q_k) \to (H_1,\dots,H_k)} x_{Q_1,\dots,Q_k}$, where each $Q_i$ ranges over finite non-empty multisets in $H_i$, using the product directed set ${\mathcal F}[H_1] \times \dots \times {\mathcal F}[H_d]$.  Thus for instance $\lim_{(Q_1,Q_2) \to (H_1,H_2)} x_{Q_1,Q_2} = x$ if, for every neighbourhood $U$ of $x$, there exist $Q_{1,0}, Q_{2,0}$ in $H_1,H_2$ respectively such that $x_{Q_1,Q_2} \in U$ whenever $Q_1 \geq Q_{1,0}$ and $Q_2 \geq Q_{2,0}$.  If the $x_Q$ or $x_{Q_1,\dots,Q_k}$ take values in $\R$, we can also define limit superior and limit inferior in the usual fashion.

\begin{remark}[Amenable case]
If $G$ is an amenable (discrete) countable additive group with a F{\o}lner sequence $\Phi_n$, and $a_g$ is a bounded sequence of complex numbers indexed by $g \in G$, then we have the relationship
\begin{equation}\label{fol}
 \lim_{n \to \infty} \E_{g \in \Phi_n} a_g = \lim_{Q \to G} \E_{g \in Q} a_g
\end{equation}
between traditional averages and the averages defined above, whenever the right-hand side limit exists.  Indeed, for any $\eps>0$, we see from the F{\o}lner property that for any given $Q \in {\mathcal F}[G]$, that $\E_{g \in \Phi_n} a_g$ and $\E_{g \in \Phi_n+Q} a_g$ differ by at most $\eps$ if $n$ is large enough; while from the convergence of the right-hand side limit we see that $\E_{g \in \Phi_n+Q} a_g$ and $\E_{g \in Q} a_g$ differ by at most $\eps$ for all $n$ if $Q$ is large enough, and the claim follows.  A similar result holds for joint limits, namely that
\begin{equation}\label{fol-joint}
 \lim_{n_1,\dots,n_d \to \infty} \E_{h_1 \in \Phi_{n_1,1}, \dots, h_d \in \Phi_{n_i,d}} a_{h_1,\dots,h_d} = \lim_{(Q_1,\dots,Q_d) \to (H_1,\dots,H_d)} \E_{h_1 \in Q_1,\dots,h_d \in Q_d} a_{h_1,\dots,h_d}
\end{equation}
whenever $\Phi_{n,i}$ is a F{\o}lner sequence for $H_i$ and the $a_{h_1,\dots,h_d}$ are a bounded sequence of complex numbers.  
\end{remark}

Given a $G$-system $(\X,T)$, a natural number $d$, a subgroup $H$ of $G$, and a function $f \in L^\infty(\X)$, we define the \emph{Gowers-Host-Kra seminorm}
\begin{equation}\label{ghk-def} 
\| f \|_{U^d_H(\X)} \coloneqq  \lim_{Q \to H} \| f \|_{U^d_Q(\X)};
\end{equation}
we will show in Theorem \ref{gawd} below that this limit exists, and agrees with the more usual definition of the Gowers-Host-Kra seminorm from \cite{hk}; in fact the definition given here even extends to the case when $G$ and $H$ are uncountable abelian groups.  More generally, given subgroups $H_1,\dots,H_d$ of $G$, we define the \emph{Gowers-Host-Kra box seminorm}
$$ 
\| f \|_{\Box^d_{H_1,\dots,H_d}(\X)} \coloneqq  \lim_{(Q_1,\dots,Q_d) \to (H_1,\dots,H_d)} \| f \|_{\Box^d_{Q_1,\dots,Q_d}(\X)}.$$
again, the existence of this limit will be shown in Theorem \ref{gawd} below.  

Define a \emph{factor} of a $G$-system $(\X,T)$ with $\X = (X, {\mathcal B}, \mu)$ to be a $G$-system $(\Y,T)$ with $\Y = (Y, {\mathcal Y}, \nu)$ together with a measurable factor map $\pi\colon X \to Y$ intertwining the group actions (thus $T^g \circ \pi = \pi \circ T^g$ for all $g \in G$) such that $\nu$ is equal to the pushforward $\pi_* \mu$ of $\mu$, thus $\mu(\pi^{-1}(E)) = \nu(E)$ for all $E \in {\mathcal Y}$.  By abuse of notation we use $T$ to denote the action on the factor $\Y$ as well as on the original space $\X$.   Note that $L^\infty(\Y)$ can be viewed as a ($G$-invariant) subalgebra of $L^\infty(\X)$, and similarly $L^2(\Y)$ is a ($G$-invariant) closed subspace of the Hilbert space $L^2(\X)$; if $f \in L^2(\X)$, we write $\E(f|\Y)$ for the orthogonal projection onto $L^2(\Y)$.    We also call $\X$ an \emph{extension} of $\Y$.  Note that any subalgebra ${\mathcal Y}$ of ${\mathcal B}$ can be viewed as a factor of $\X$ by taking $Y=X$ and $\nu = \mu\downharpoonright_{{\mathcal Y}}$.  For instance, given a subgroup $H$ of $G$, the invariant $\sigma$-algebra ${\mathcal B}^H$ consisting of sets $E \in {\mathcal B}$ such that $T^h E = E$ up to null sets for any $h \in H$ generates a factor $\X^H$ of $\X$, and so we can meaningfully define the conditional expectation $\E(f|\X^H)$ for any $f \in L^2(\X)$.

Two factors $\Y, \Y'$ of $\X$ are said to be \emph{equivalent} if the algebras $L^\infty(\Y)$ and $L^\infty(\Y')$ agree (using the usual convention of identifying functions in $L^\infty$ that agree almost everywhere), in which case we write $\Y \equiv \Y'$.  We partially order the factors of $\X$ up to equivalence by declaring $\Y \leq \Y'$ if $L^\infty(\Y) \subset L^\infty(\Y')$, thus for instance $\Y \leq \Y'$ if $\Y$ is a factor of $\Y'$.  This gives the factors of $\X$ up to equivalence the structure of a lattice: the meet $\Y \wedge \Y'$ of two factors is given (up to equivalence) by setting $L^\infty(\Y \wedge \Y') = L^\infty(\Y) \cap L^\infty(\Y')$, and the join $\Y \vee \Y'$ of two factors is given by setting $L^\infty(\Y \vee \Y')$ to be the von Neumann algebra generated by $L^\infty(\Y) \cup L^\infty(\Y')$ (i.e., the smallest von Neumann subalgebra of $L^\infty(\X)$ containing $L^\infty(\Y) \cup L^\infty(\Y')$).

We say that a $G$-system $\X$ is \emph{$H$-ergodic} for some subgroup $H$ of $G$ if the invariant factor $\X^H$ is trivial (equivalent to a point).  Note that if a system is $G$-ergodic, it need not be $H$-ergodic for subgroups $H$ of $G$.  Because of this, it will be important to not assume $H$-ergodicity for many of the results and arguments below, which will force us to supply new proofs of some existing results in the literature that were specialised to the ergodic case.

For the seminorm $U^d_H(\X)$, it is known\footnote{The factor $\ZZ^{{<}d}$ here is often denoted $\ZZ^{d-1}$ in the literature.  As such, the reader should be cautioned that some of the indices here are shifted by one from those in other papers.} (see \cite{hk} or Theorem \ref{charfac} below) that there exists a \emph{characteristic factor} $(\ZZ^{{<}d}_H(\X),T) = (Z^{{<}d}_H, {\mathcal Z}^{{<}d}_H, \nu,T)$ of $(\X,T)$, unique up to equivalence, with the property that
\begin{equation}\label{fuz}
 \|f\|_{U^d_H(\X)} = 0 \iff \E( f | \ZZ^{{<}d}_H(\X) ) = 0
\end{equation}
for all $f \in L^\infty(\X)$; for instance, $\ZZ^{{<}1}_H(\X)$ can be shown to be equivalent to the invariant factor $\X^H$, and the factors $\ZZ^{{<}d}_H(\X)$ are non-decreasing in $d$.  In the case when $H$ is isomorphic to the integers $\Z$, and assuming for simplicity that the system $\X$ is $H$-ergodic, the characteristic factor was studied by Host and Kra \cite{hk}, who obtained the important result that the characteristic factor $\ZZ^{{<}d}_H(\X)$ was isomorphic to a limit of $d-1$-step nilsystems (see also \cite{ziegler} for a related result regarding characteristic factors of multiple averages); the analogous result for actions of infinite-dimensional vector spaces $\F^\omega \coloneqq  \bigcup_n \F^n$ was obtained in \cite{btz}.  More generally, given subgroups $H_1,\dots,H_d$, there is a unique (up to equivalence) characteristic factor $\ZZ^{{<}d}_{H_1,\dots,H_d}(\X) = (Z^{{<}d}_{H_1,\dots,H_D}, {\mathcal Z}^{{<}d}_{H_1,\dots,H_d}, \nu)$ with the property that
$$ \|f\|_{\Box^d_{H_1,\dots,H_d}(\X)} = 0 \iff \E( f | \ZZ^{{<}d}_{H_1,\dots,H_d}(\X) ) = 0$$
for all $f \in L^\infty(\X)$; this was essentially first observed in \cite{host}, and we establish it in Theorem \ref{charfac}.  However, a satisfactory structural description of the factors $\ZZ^{<d}_{H_1,\dots,H_d}(\X)$ (in the spirit of \cite{hk}) is not yet available; see \cite{austin-sated} for some recent work in this direction.

We can now state the ergodic theory analogues of Theorems \ref{concat-uap} and \ref{concat-box}.  In these results $G$ is always understood to be an at most countable additive group. Because our arguments will require a F{\o}lner sequence of coset progressions of bounded rank, we will also have to temporarily make a further technical restriction on $G$, namely that $G$ be the sum of a finitely generated group and a profinite group (or equivalently, a group which becomes finitely generated after quotienting by a profinite subgroup).  This class of groups includes the important examples of lattices $\Z^d$ and vector spaces ${\mathbb F}^\omega = \bigcup_n {\mathbb F}^n$ over finite fields, but excludes the infinitely generated torsion-free group $\Z^\omega = \bigcup_n \Z^n$.  Observe that this class of groups is also closed under quotients and taking subgroups.  

\begin{theorem}[Concatenation theorem for characteristic factors]\label{concat-uap-erg}  Suppose that $G$ is the sum of a finitely generated group and a profinite group.  Let $\X$ be a $G$-system, let $H_1,H_2$ be subgroups of $G$, and let $d_1,d_2$ be positive integers.  
Then
$$ L^\infty(\ZZ^{{<}d_1}_{H_1}(\X)) \cap L^\infty(\ZZ^{{<}d_2}_{H_2}(\X)) \subset L^\infty( \ZZ^{{<}d_1+d_2-1}_{H_1+H_2}(\X) ).$$
Equivalently, using the lattice structure on factors discussed previously,
$$ \ZZ^{{<}d_1}_{H_1}(\X) \wedge \ZZ^{{<}d_2}_{H_2}(\X) \leq \ZZ^{{<}d_1+d_2-1}_{H_1+H_2}(\X).$$
\end{theorem}

\begin{theorem}[Concatenation theorem for low rank characteristic factors]\label{concat-box-erg}  Suppose that $G$ is the sum of a finitely generated group and a profinite group.  Let $\X$ be a $G$-system, let $d_1,d_2$ be positive integers, and let $H_{1,1},\dots,H_{1,d_1}$,$H_{2,1},\dots,H_{2,d_2}$ be subgroups of $G$.  
Then
$$ L^\infty(\ZZ^{{<}d_1}_{H_{1,1},\dots,H_{1,d_1}}(\X)) \cap L^\infty(\ZZ^{{<}d_2}_{H_{2,1},\dots,H_{2,d_2}}(\X)) \subset L^\infty( \ZZ^{{<}d_1 d_2}_{(H_{1,i}+H_{2,j})_{1 \leq i \leq d_1, 1 \leq j \leq d_2}}(\X) )$$
or equivalently,
$$ \ZZ^{{<}d_1}_{H_{1,1},\dots,H_{1,d_1}}(\X) \wedge \ZZ^{{<}d_2}_{H_{2,1},\dots,H_{2,d_2}}(\X) \leq \ZZ^{{<}d_1 d_2}_{(H_{1,i}+H_{2,j})_{1 \leq i \leq d_1, 1 \leq j \leq d_2}}(\X).$$
\end{theorem}

In Section \ref{erg-sec}, we deduce these results from the corresponding combinatorial results in Theorems \ref{concat-uap}, \ref{concat-box}.

\begin{example}\label{stex}  Let $\X \coloneqq  (\R/\Z)^3$ with the uniform probability measure, and define the shifts $T^{(1,0)}, T^{(0,1)}\colon \X \to \X$ by
$$ T^{(1,0)}(x,y,z) \coloneqq  (x+\alpha,y,z+y)$$
and
$$ T^{(0,1)}(x,y,z) \coloneqq  (x,y+\alpha,z+x)$$
where $\alpha \in \R/\Z$ is a fixed irrational number.  These shifts commute and generate a $\Z^2$-system
$$ T^{(n,m)}(x,y,z) \coloneqq  (x+n\alpha, y+m\alpha,z + ny + mx + nm \alpha)$$
(compare with \eqref{pnm}).
The shift $T^{(1,0)}$ does not act ergodically on $\X$, but one can perform an ergodic decomposition into ergodic components $\R/\Z \times \{y\} \times \R/\Z$ for almost every $y \in \R/\Z$, with $T^{(1,0)}$ acting as a circle shift $(x,z) \mapsto (x+\alpha,z+y)$ on each such component.  From this one can easily verify that $\ZZ^{{<}2}_{\Z \times \{0\}}(\X) = \X$. Similarly $\ZZ^{{<}2}_{\{0\} \times \Z}(\X) = \X$.  On the other hand, $\ZZ^{{<}2}_{\Z^2}(\X) \neq \X$, as there exist functions in $L^\infty(\Z)$ whose $U^2_{\Z^2}(\X)$ norm vanish (for instance the function $(x,y,z) \mapsto e^{2\pi i z}$).  Nevertheless, Corollary \ref{concat-uap-erg} concludes that $\ZZ^{{<}3}_{\Z^2}(\X) = \X$ (roughly speaking, this means that $\X$ exhibits ``quadratic'' or ``$2$-step'' behaviour as a $\Z^2$-system, despite only exhibiting ``linear'' or ``$1$-step'' behaviour as a $\Z \times \{0\}$-system or $\{0\} \times \Z$-system).
\end{example}

\begin{remark}  In the case that $H$ is an infinite cyclic group acting ergodically, Host and Kra \cite{hk}  show that the characteristic factor $\ZZ^{{<}d}_H(\X) = \ZZ^{{<}d}_{H,\dots,H}(\X)$ is an inverse limit of $d-1$-step nilsystems.  If $H$ does not act ergodically, then (assuming some mild regularity on the underlying measure space $X$) one has a similar characterization of $\ZZ^{{<}d}_{H}(\X)$ on each ergodic component.  The arguments in \cite{hk} were extended to finitely generated groups $H$ acting ergodically in \cite{griesmer}; see also \cite{btz} for an analogous result in the case of actions of infinite-dimensional vector spaces $\F^\omega$ over a finite field. Theorem \ref{concat-uap-erg} can then be interpreted as an assertion that if $\X$ acts as an inverse limit of nilsystems of step $d_1-1$ along the components of one group action $H_1$, and as an inverse limit of nilsystems of step $d_2-1$ along the components of another (commuting) group action $H_2$, then $\X$ is an inverse limit of nilsystems of step at most $d_1+d_2-2$ along the components of the joint $H_1+H_2$ action.  It seems of interest to obtain a more direct proof of this assertion.  A related question would be to establish a nilsequence version of Proposition \ref{concat}.  For instance one could conjecture that whenever a sequence $f: \Z \times \Z \to \C$ was such that $n_1 \mapsto f(n_1,n_2)$ was a Lipschitz nilsequence of step $d_1-1$ uniformly in $n_2$ (as defined for instance in \cite{gt-mobius}), and $n_2 \mapsto f(n_1,n_2)$ was a Lipschitz nilsequence of step $d_2-1$ uniformly in $n_1$, then $f$ itself would be a Lipschitz nilsequence jointly on $\Z^2$ of step $d_1+d_2-2$.  It seems that Proposition \ref{concat-uap} is at least able to show that $f$ can be locally approximated (in, say, an $L^2$ sense) by such nilsequences on arbitrarily large scales, but some additional argument is needed to obtain the conjecture as stated.
\end{remark}

We are able to remove the requirement that $G$ be the sum of a finitely generated group and a profinite group from Theorem \ref{concat-uap-erg}:

\begin{theorem}[Concatenation theorem for characteristic factors]\label{concat-uap-erg-weak}  Let $G$ be an at most countable additive group.  Let $(\X,T)$ be a $G$-system, let $H_1$, $H_2$ be subgroups of $G$, and let $d_1$, $d_2$ be positive integers.  
Then
$$ L^\infty(\ZZ^{{<}d_1}_{H_1}(\X)) \cap L^\infty(\ZZ^{{<}d_2}_{H_2}(\X)) \subset L^\infty( \ZZ^{{<}d_1+d_2-1}_{H_1+H_2}(\X) )$$
or equivalently
$$ \ZZ^{{<}d_1}_{H_1}(\X) \wedge \ZZ^{{<}d_2}_{H_2}(\X) \leq \ZZ^{{<}d_1+d_2-1}_{H_1+H_2}(\X).$$
\end{theorem}

We prove this result in Section \ref{erg2-sec} using an ergodic theory argument that relies on the machinery of cubic measures and cocycle type that was introduced by Host and Kra \cite{hk}, rather than on the combinatorial arguments used to establish Theorems \ref{concat-uap}, \ref{concat-box}.  It is at this point that we use our requirement that $G$-systems be countably generated modulo null sets, in order to apply the ergodic decomposition (after first passing to a compact metric space model), as well as the Mackey theory of isometric extensions.  It is likely that a similar strengthening of Theorem \ref{concat-box-erg} can be obtained, but this would require extending much of the Host-Kra machinery to tuples of commuting actions, which we will not do here.

\subsection{Globalizing uniformity}

We have seen that anti-uniformity can be ``concatenated'', in that functions which are approximately orthogonal to functions that are locally Gowers uniform in two different directions are necessarily also approximately orthogonal to more globally Gowers uniform functions.  By duality, one then expects to be able to decompose a globally Gowers uniform function into functions that are locally Gowers uniform in different directions.  For instance, in the ergodic setting, one has the following consequence of Theorem \ref{concat-uap-erg}:

\begin{corollary}\label{concat-uap-split}  Let $(\X,T)$ be a $G$-system, let $H_1$, $H_2$ be subgroups of $G$, and let $d_1$, $d_2$ be positive integers.   If $f \in L^\infty(\X)$ is orthogonal to $L^\infty( \ZZ^{<d_1+d_2-1}_{H_1+H_2}(\X) )$ (with respect to the $L^2$ inner product), then one can write $f=f_1+f_2$, where $f_1 \in L^\infty(\X)$ is orthogonal to
$L^\infty(\ZZ^{<d_1}_{H_1}(\X))$ and $f_2 \in L^\infty(\X)$ is orthogonal to $L^\infty(\ZZ^{<d_2}_{H_2}(\X))$; furthermore, $f_1$ and $f_2$ are orthogonal to each other.  
\end{corollary}

Of course, by \eqref{fuz}, one can replace ``orthogonal to $L^\infty( \ZZ^{{<}d_1+d_2-1}_{H_1+H_2}(\X) )$'' in the above corollary with ``having vanishing $U^{d_1+d_2-1}_{H_1+H_2}(\X)$ norm'', and similarly for being orthogonal to $L^\infty(\ZZ^{{<}d_1}_{H_1}(\X))$ or $L^\infty(\ZZ^{{<}d_2}_{H_2}(\X))$.

\begin{proof}  By Theorem \ref{concat-uap-erg-weak} and \eqref{fuz} applied to the system $\ZZ^{{<}d_1}_{H_1}(\X)$, any function in $L^\infty(\ZZ^{{<}d_1}_{H_1}(\X))$ orthogonal to $L^\infty(\ZZ^{{<}d_1+d_2-1}_{H_1+H_2}(\X))$ will have vanishing $U^{d_2}_{H_2}( \ZZ^{{<}d_1}_{H_1}(\X) )$ seminorm.  This seminorm is the same as the $U^{d_2}_{H_2}(\X)$ seminorm, and so by \eqref{fuz} again this function must be necessarily orthogonal to $L^\infty(\ZZ^{{<}d_2}_{H_2}(\X))$.  We conclude that the restrictions of the spaces $L^\infty(\ZZ^{{<}d_1}_{H_1}(\X))$ and $L^\infty( \ZZ^{{<}d_2}_{H_2}(\X) )$ to $L^\infty(\ZZ^{{<}d_1+d_2-1}_{H_1+H_2}(\X))$ are orthogonal, and the claim follows.
\end{proof}

One can similarly use Theorem \ref{concat-box-erg} to obtain 

\begin{corollary}\label{concat-box-split}  Suppose that $G$ is the sum of a finitely generated group and a profinite group. Let $(\X,T)$ be a $G$-system, let $d_1$, $d_2$ be positive integers, and let $H_{1,1}$, $\dots$, $H_{1,d_1}$, $H_{2,1}$, $\dots$, $H_{2,d_2}$ be subgroups of $G$. If $f \in L^\infty(\X)$ is orthogonal to $L^\infty( \ZZ^{<d_1 d_2}_{(H_{1,i}+H_{2,j})_{1 \leq i \leq d_1, 1 \leq j \leq d_2}}(\X) )$, then one can write $f=f_1+f_2$, where $f_1 \in L^\infty(\X)$ is orthogonal to
$L^\infty(\ZZ^{<d_1}_{H_{1,1},\dots,H_{1,d_1}}(\X))$ and $f_2 \in L^\infty(\X)$ is orthogonal to $L^\infty(\ZZ^{<d_2}_{H_{2,1},\dots,H_{2,d_2}}(\X))$; furthermore, $f_1$ and $f_2$ are orthogonal to each other.
\end{corollary}

We can use the orthogonality in Corollary \ref{concat-uap-split} to obtain a Bessel-type inequality:

\begin{corollary}[Bessel inequality]\label{bescor} Let $G$ be an at most countable additive group.  Let $(\X,T)$ be a $G$-system, let $(H_i)_{i \in I}$ be a finite family of subgroups of $G$, and let $(d_i)_{i \in I}$ be a family of positive integers.
Then for any $f \in L^\infty(\X)$, we have
\begin{equation}\label{bes}
\sum_{i \in I} \| \E( f | \ZZ^{{<}d_i}_{H_i}(\X) ) \|_{L^2(\X)}^2 \leq \|f\|_{L^2(\X)} \left( \sum_{i,j \in I} \| \E( f | \ZZ^{{<}d_i+d_j-1}_{H_i+H_j}(\X) ) \|_{L^2(\X)}^2\right)^{1/2}.
\end{equation}
\end{corollary}

\begin{proof}  Write $f_i \coloneqq  \E( f | \ZZ^{{<}d_i}_{H_i}(\X) )$.  We can write the left-hand side of \eqref{bes} as
$$ \left \langle f, \sum_{i \in I} f_i \right \rangle $$
which by the Cauchy-Schwarz inequality is  bounded by
$$ \|f\|_{L^2(\X)} \left| \sum_{i,j \in I} \langle f_i, f_j \rangle \right|^{1/2}.$$
But from Corollary \ref{concat-uap-split}, $L^\infty(\ZZ^{{<}d_i}_{H_i}(\X))$ and $L^\infty(\ZZ^{{<}d_j}_{H_j}(\X))$ are orthogonal on the orthogonal complement of 
$L^\infty( \ZZ^{{<}d_i+d_j-1}_{H_i+H_j}(\X) )$, hence
$$ \langle f_i, f_j \rangle = \left\| \E( f | \ZZ^{{<}d_i+d_j-1}_{H_i+H_j}(\X) ) \right\|_{L^2(\X)}^2$$
and the claim follows.
\end{proof}

One can use Corollary \ref{concat-box-split} to obtain an analogous Bessel-type inequality involving finitely generated subgroups $H_{i,k}$, $k=1,\dots,d_i$, which we leave to the interested reader.

Returning now to the finitary Gowers norms, one has a qualitative analogue of the Bessel inequality involving the Gowers uniformity norms:

\begin{theorem}[Qualitative Bessel inequality for uniformity norms]\label{besu}  Let $(Q_i)_{i \in I}$ be a finite non-empty family of coset progressions $Q_i$, all of rank at most $r$, in an additive group $G$. Let $(\X,T)$ be a $G$-system, and let $d$ be a positive integer.  Let $f$ lie in the unit ball of $L^\infty(\X)$, and suppose that
\begin{equation}\label{bes-1}
\E_{i,j \in I} \|f\|_{U^{2d-1}_{\eps Q_i+\eps Q_j}(\X)} \leq \eps 
\end{equation}
for some $\eps > 0$.  Then
\begin{equation}\label{bes-2}
\E_{i \in I} \|f\|_{U^{d}_{Q_i}(\X)} \leq c(\eps) 
\end{equation}
where $c\colon (0,+\infty) \to (0,+\infty)$ is a function such that $c(\eps) \to 0$ as $\eps \to 0$.  Furthermore, $c$ depends only on $r$ and $d$. (In particular, $c$ is independent of the size of $I$.)
\end{theorem}

We prove this theorem in Section \ref{bes-sec}.  We remark that the theorem is only powerful when the cardinality of the set $I$ is large compared to $\eps$, otherwise the claim would easily follow from considering the diagonal contribution $i=j$ to \eqref{bes-1}.

Theorem \ref{besu} has an analogue for the Gowers box norms:

\begin{theorem}[Qualitative Bessel inequality for box norms]\label{besu-box}  Let $d$ be a positive integer.  For each $1 \leq j \leq d$, let $(Q_{i,j})_{i \in I}$ be a finite family of coset progressions $Q_{i,j}$, all of rank at most $r$, in an additive group $G$. Let $(\X,T)$ be a $G$-system.  Let $f$ lie in the unit ball of $L^\infty(\X)$, and suppose that
$$
\E_{i,j \in I} \|f\|_{\Box^{d^2}_{(\eps Q_{i,k}+\eps Q_{j,l})_{1 \leq k \leq d, 1 \leq l \leq d}}(\X)} \leq \eps $$
for some $\eps > 0$.  Then
$$
\E_{i \in I} \|f\|_{\Box^{d}_{Q_{i,1},\dots,Q_{i,d}}(\X)} \leq c(\eps) 
$$
where $c\colon (0,+\infty) \to (0,+\infty)$ is a function such that $c(\eps) \to 0$ as $\eps \to 0$.  Furthermore, $c$ depends only on $r$ and $d$.
\end{theorem}

\begin{remark}  In view of Theorem \ref{besu}, one may ask if the smallness of $\|f\|_{U^{d_1+d_2-1}_{\eps Q_1 + \eps Q_2}(\X)}$ implies the smallness of $\min( \|f\|_{U^{d_1}_{Q_1}(\X)}, \|f\|_{U^{d_2}_{Q_2}(\X)})$ (or, taking contrapositives, the largeness of $\|f\|_{U^{d_1}_{Q_1}(\X)}$ and $\|f\|_{U^{d_2}_{Q_2}(\X)}$ implies some non-trivial lower bound for $\|f\|_{U^{d_1+d_2-1}_{\eps Q_1 + \eps Q_2}(\X)}$).  Unfortunately this is not the case; a simple counterexample is provided by a function $f$ of the form $f_1+f_2$, where $f_1$ is a bounded function with large $U^{d_1}_{Q_1}(\X)$ norm but small $U^{d_2}_{Q_2}(\X)$ and $U^{d_1+d_2-1}_{\eps Q_1 + \eps Q_2}(\X)$ norm (which is possible for instance if $Q_1,Q_2$ are large subgroups of $G$, $\X = G$ with the translation action, and $f_1$ is constant in the $Q_1$ direction but random in the $Q_2$ direction), and vice versa for $f_2$.  
\end{remark}

\subsection{A sample application}

In a sequel to this paper \cite{tz-2}, we will use the concatenation theorems proven in this paper to study polynomial patterns in sets such as the primes.  Here, we will illustrate how this is done with a simple example, namely controlling the average
$$
A_{N,M}(f_1,f_2,f_3,f_4) \coloneqq  \E_{x \in \Z/N\Z} \E_{n,m,k \in [M]} f_1(x) f_2(x+nm) f_3(x+nk) f_4(x+nm+nk)
$$
for functions $f_1,f_2,f_3,f_4\colon \Z/N\Z \to \C$ on a cyclic group $\Z/N\Z$, where $[M] \coloneqq  \{ n \in \N: n \leq M \}$.  The natural range of parameters is when $M$ is comparable to $\sqrt{N}$.  In that case we will be able to control this expression by the global Gowers $U^3$ norm:

\begin{proposition}\label{cprop}  Let $N \geq 1$ be a natural number, let $M$ be such that $C^{-1} \sqrt{N} \leq M \leq C \sqrt{N}$ for some $C>0$, and let $f_1,f_2,f_3,f_4\colon \Z/N\Z \to \C$ be functions bounded in magnitude by $1$.  Suppose that
$$
\| f_i \|_{U^3_{\Z/N\Z}(\Z/N\Z)} \leq \eps $$
for some $\eps>0$ and $i=1,\dots,4$, where we give $\Z/N\Z$ the uniform probability measure.  Then
$$ |A_{N,M}(f_1,f_2,f_3,f_4)| \leq c(\eps)$$
for some quantity $c(\eps)$ depending only on $\eps$ and $C$ that goes to zero as $\eps \to 0$.
\end{proposition}

For instance, using this proposition (embedding $[N]$ in, say, $\Z/5N\Z$) and the known uniformity properties of the M\"obius function $\mu$ (see \cite{gt-quadratic}) we can now obtain the asymptotic
$$ \E_{x \in [N]} \E_{n,m,k \in [\sqrt{N}]} \mu(x) \mu(x+nm) \mu(x+nk) \mu(x+nm+nk) = o(1)$$
as $N \to \infty$; we leave the details to the interested reader.  As far as we are able to determine, this asymptotic appears to be new.  In the sequel \cite{tz-2} to this paper we will consider similar asymptotics involving various polynomial averages such as $\E_{x \in [N]} \E_{n \in [\sqrt{N}]} f(x) f(x+n^2) f(x+2n^2)$, and arithmetic functions such as the von Mangoldt function $\Lambda$.

We prove Proposition \ref{cprop} in Section \ref{cprop-sec}.  Roughly speaking, the strategy is to observe that $A_{N,M}(f_1,f_2,f_3,f_4)$ can be controlled by averages of ``local'' Gowers norms of the form $U^2_Q$, where $Q$ is an arithmetic progression in $\Z/N\Z$ of length comparable to $M$.  Each individual such norm is not controlled directly by the $U^3_{\Z/N\Z}$ norm, due to the sparseness of $Q$; however, after invoking Theorem \ref{besu}, we can control the averages of the $U^2_Q$ norms with averages of $U^3_{Q+Q'}$ norms, where $Q,Q'$ are two arithmetic progressions of length comparable to $M$.  For typical choices of $Q,Q'$, the rank two progerssion $Q+Q'$ will be quite dense in $\Z/N\Z$, allowing one to obtain the proposition.

One would expect that the $U^3$ norm in Proposition \ref{cprop} could be replaced with a $U^2$ norm. Indeed this can be done invoking the  inverse theorem for the $U^3$ norm \cite{gt-u3} as well as equidistribution results for nilsequences \cite{gt-ratner}:

\begin{proposition}\label{cprop-2}  The $U^3_{\Z/N\Z}(\Z/N\Z)$ norm in Proposition \ref{cprop} may be replaced with $U^2_{\Z/N\Z}(\Z/N\Z)$.
\end{proposition}

This result can be proven by adapting the arguments based on the arithmetic regularity lemma in \cite[\S 7]{gt-reg}; we sketch the key new input required in Section \ref{cprop2-sec}.
In the language of Gowers and Wolf \cite{gow-wolf}, this proposition asserts that the \emph{true complexity} of the average $A_{N,M}$ is $1$, rather than $2$ (the latter being roughly analogous to the ``Cauchy-Schwarz complexity'' discussed in \cite{gow-wolf}).  This drop in complexity is consistent with similar results established in the ergodic setting in \cite{bll}, and in the setting of linear patterns in \cite{gow-wolf,gow-wolf-2,gow-wolf-3,gow-wolf-4,gt-reg}, and is proven in a broadly similar fashion to these results.  In principle, Proposition \ref{cprop-2} is purely an assertion in ``linear'' Fourier analysis, since it only involves the $U^2$ Gowers norm, but we do not know of a way to establish it without exploiting both the concatenation theorem and the inverse $U^3$ theorem.

We thank the anonymous referees for a careful reading of the paper and for many useful suggestions and corrections.

\section{The ergodic theory limit}\label{erg-sec}

In this section we show how to obtain the main ergodic theory results of this paper (namely, Theorem \ref{concat-uap-erg} and Theorem \ref{concat-box-erg}) as a limiting case of their respective finitary results, namely Theorem \ref{concat-uap} and Theorem \ref{concat-box}.  The technical hypothesis that the group $G$ be the sum of a finitely generated group and a profinite group will only be needed near the end of the section.  Readers who are only interested in the combinatorial assertions of this paper can skip ahead to Section \ref{sketch-sec}.

We first develop some fairly standard material on the convergence of the Gowers-Host-Kra norms, and on the existence of characteristic factors.  Given a $G$-system $(\X,T)$, finite non-empty multisets $Q_1,\dots,Q_d$ of $G$, and elements $f_\omega$ of $L^{2^d}(\X)$ for each $\omega \in \{0,1\}^d$, define the \emph{Gowers inner product}
\begin{equation}\label{gip}
\langle (f_\omega)_{\omega \in\{0,1\}^d} \rangle_{\Box^d_{Q_1,\dots,Q_d}(\X)} 
\coloneqq  \E_{h^0_i,h^1_i \in Q_i \forall i=1,\dots,d} \int_X \prod_{\omega \in \{0,1\}^d} {\mathcal C}^{|\omega|} T^{h^{\omega_1}_1 + \dots + h^{\omega_d}_d} f_\omega\ d\mu
\end{equation}
where $\omega = (\omega_1,\dots,\omega_d)$, $|\omega| \coloneqq  \omega_1 + \dots + \omega_d$, and ${\mathcal C}\colon f \mapsto \overline{f}$ is the complex conjugation operator; the absolute convergence of the integral is guaranteed by H\"older's inequality.  Comparing this with Definition \ref{gaw}, we see that
$$
\|f\|_{\Box^d_{Q_1,\dots,Q_d}(\X)}^{2^d} = \langle (f)_{\omega \in\{0,1\}^d} \rangle_{\Box^d_{Q_1,\dots,Q_d}(\X)}.$$ 
We also recall the \emph{Cauchy-Schwarz-Gowers inequality}
\begin{equation}\label{csg}
|\langle (f_\omega)_{\omega \in\{0,1\}^d} \rangle_{\Box^d_{Q_1,\dots,Q_d}(\X)}|
\leq \prod_{\omega \in \{0,1\}^d} \|f\|_{\Box^d_{Q_1,\dots,Q_d}(\X)}^{2^d} 
\end{equation}
(see \cite[Lemma B.2]{gt-primes}).  By setting $f_\omega$ to equal $f$ when $\omega_{d'+1}=\dots=\omega_d=0$, and equal to $1$ otherwise, we obtain as a corollary the monotonicity property
\begin{equation}\label{mono}
\| f\|_{\Box^{d'}_{Q_1,\dots,Q_{d'}}(\X)} \leq 
\| f\|_{\Box^{d}_{Q_1,\dots,Q_{d}}(\X)} 
\end{equation}
for all $1 \leq d' \leq d$ and $f \in L^{2^d}(\X)$.

We have the following convergence result:

\begin{theorem}[Existence of Gowers-Host-Kra seminorm]\label{gawd} Let $(\X,T)$ be a $G$-system, let $d$ be a natural number, and let $H_1,\dots,H_d$ be subgroups of $G$.  For each $\omega \in \{0,1\}^d$, let $f_\omega$ be an element of $L^{2^d}(\X)$.  Then the limit
$$ \langle (f_\omega)_{\omega \in\{0,1\}^d} \rangle_{\Box^d_{H_1,\dots,H_d}(\X)}  \coloneqq 
\lim_{(Q_1,\dots,Q_d) \to (H_1,\dots,H_d)} \langle (f_\omega)_{\omega \in\{0,1\}^d} \rangle_{\Box^d_{Q_1,\dots,Q_d}(\X)} $$
exists.  In particular, the limit
$$ \| f \|_{\Box^d_{H_1,\dots,H_d}(\X)} \coloneqq  \lim_{(Q_1,\dots,Q_d) \to (H_1,\dots,H_d)} \|f\|_{\Box^d_{Q_1,\dots,Q_d}(\X)}$$
exists for any $f \in L^{2^d}(\X)$.
\end{theorem}

It is likely that one can deduce this theorem from the usual ergodic theorem, by adapting the arguments in \cite{hk}, but we will give a combinatorial proof here, which applies even in cases in which $G$ or $H_1,\dots,H_d$ are uncountable.

\begin{proof}
By multilinearity we may assume that the $f_\omega$ are all real-valued, so that we may dispense with the complex conjugation operations.  We also normalise the $f_\omega$ to lie in the closed unit ball of $L^{2^d}(\X)$.

If $\vec f \coloneqq  (f_\omega)_{\omega \in \{0,1\}^d}$ is a tuple of functions $f_\omega \in L^{2^d}(\X)$ and $0 \leq d' \leq d$, we say that $\vec f$ is \emph{$d'$-symmetric} if we have $f_\omega = f_{\omega'}$ whenever $\omega = (\omega_1,\dots,\omega_d)$ and $\omega' = (\omega'_1,\dots,\omega'_d)$ agree in the first $d'$ components (that is, $\omega_i = \omega'_i$ for $i=1,\dots,d'$).  We will prove Theorem \ref{gawd} by downward induction on $d'$, with the $d'=0$ case establishing the full theorem.

Thus, assume that $0 \leq d' \leq d$ and that the claim has already been proven for larger values of $d'$ (this hypothesis is vacuous for $d'=d$).
We will show that for any given (and sufficiently small) $\eps > 0$, and for sufficiently large $(Q_1,\dots,Q_d)$ (in the product directed set ${\mathcal F}[H_1] \times \dots \times {\mathcal F}[H_d]$), the quantity $\langle \vec f \rangle_{\Box^d_{Q_1,\dots,Q_d}(\X)}$ only increases by at most $\eps$ if one increases any of the $Q_i$, that is to say
\begin{equation}\label{fow}
\langle \vec f \rangle_{\Box^d_{Q_1,\dots,Q_{i-1},Q_i + R, Q_{i+1},\dots,Q_d}(\X)}
\leq \langle \vec f \rangle_{\Box^d_{Q_1,\dots,Q_d}(\X)} + \eps
\end{equation}
for any $i=1,\dots,d$ and any finite non-empty multiset $R$.  Applying this once for each $i$, we see that the limit superior of the $\langle \vec f \rangle_{\Box^d_{Q_1,\dots,Q_d}(\X)}$ does not exceed the limit inferior by more than $d\eps$, and sending $\eps \to 0$ (and using the boundedness of the $\langle \vec f \rangle_{\Box^d_{Q_1,\dots,Q_d}(\X)}$, from H\"older's inequality) we obtain the claim.

It remains to establish \eqref{fow}.  There are two cases, depending on whether $i \leq d'$ or $i > d'$.  First suppose that $i \leq d'$; by relabeling we may take $i=1$.  Using \eqref{gip} and the $d'$-symmetry (and hence $1$-symmetry) of $\vec f$, we may rewrite $\langle \vec f \rangle_{\Box^d_{Q_1,\dots,Q_d}(\X)}$ as
$$ \E_{h^0_i,h^1_i \in Q_i \forall i=2,\dots,d} \| \E_{h_1 \in Q_1} T^{h_1} f_{h^0_2,h^1_2,\dots,h^0_d,h^1_d} \|_{L^2(\X)}^2$$
where
\begin{align*}
 f_{h^0_2,h^1_2,\dots,h^0_d,h^1_d} &\coloneqq  \prod_{\omega_2,\dots,\omega_d \in \{0,1\}} T^{h^{\omega_2}_2 + \dots + h^{\omega_d}_d} f_{(0,\omega_2,\dots,\omega_d)}\\
& \quad \prod_{\omega_2,\dots,\omega_d \in \{0,1\}} T^{h^{\omega_2}_2 + \dots + h^{\omega_d}_d} f_{(1,\omega_2,\dots,\omega_d)}.
\end{align*}
From the unitarity of shift operators and the triangle inequality, we have
$$
\| \E_{h_1 \in Q_1+R} T^{h_1} f_{h^0_2,h^1_2,\dots,h^0_d,h^1_d} \|_{L^2(\X)} \leq \| \E_{h_1 \in Q_1} T^{h_1} f_{h^0_2,h^1_2,\dots,h^0_d,h^1_d} \|_{L^2(\X)} $$
for any finite non-empty multiset $R$ in $H_1$.  This gives \eqref{fow} (without the epsilon loss!) in the case $i \leq d'$.

Now suppose $i>d'$.  By relabeling we may take $i=d$.  In this case, we rewrite $\langle \vec f \rangle_{\Box^d_{Q_1,\dots,Q_d}(\X)}$ as
\begin{equation}\label{shirt}
 \E_{h^0_i,h^1_i \in Q_i \forall i=1,\dots,d-1} \left\langle \E_{h_d \in Q_d} T^{h_d} f^0_{h^0_1,h^1_1,\dots,h^0_{d-1},h^1_{d-1}}, \E_{h_d \in Q_d} T^{h_d} f^1_{h^0_1,h^1_1,\dots,h^0_{d-1},h^1_{d-1}} \right\rangle_{L^2(\X)}
\end{equation}
where
$$ f^{\omega_d}_{h^0_1,h^1_1,\dots,h^0_{d-1},h^1_{d-1}} \coloneqq  \prod_{\omega_1,\dots,\omega_{d-1} \in \{0,1\}} T^{h^{\omega_1}_1 + \dots + h^{\omega_{d-1}}_{d-1}} f_{(\omega_1,\dots,\omega_d)}.$$
Note from H\"older's inequality that the $f^{\omega_d}_{h^0_1,h^1_1,\dots,h^0_{d-1},h^1_{d-1}}$ all lie in the closed unit ball of $L^2(\X)$.

A similar rewriting shows that the quantity
$$ \E_{h^0_i,h^1_i \in Q_i \forall i=1,\dots,d-1} \| \E_{h_d \in Q_d} T^{h_d}  f^0_{h^0_1,h^1_1,\dots,h^0_{d-1},h^1_{d-1}} \|_{L^2(\X)}^2 $$
is the Gowers product $\langle (f_{(\omega_1,\dots,\omega_{d-1},0)})_{\omega \in\{0,1\}^d} \rangle_{\Box^d_{Q_1,\dots,Q_d}(\X)}$.  This tuple of functions is $d'+1$-symmetric after rearrangement, and so by induction hypothesis this expression converges to a limit as $(Q_1,\dots,Q_d) \to (H_1,\dots,H_d)$.  In particular, for $(Q_1,\dots,Q_d)$ sufficiently large, we have
\begin{align*}
& \E_{h^0_i,h^1_i \in Q_i \forall i=1,\dots,d-1} \| \E_{h_d \in Q_d + \{0,n\}} T^{h_d} f^0_{h^0_1,h^1_1,\dots,h^0_{d-1},h^1_{d-1}} \|_{L^2(\X)}^2 \\
&\quad \geq \E_{h^0_i,h^1_i \in Q_i \forall i=1,\dots,d-1} \| \E_{h_d \in Q_d} T^{h_d} f^0_{h^0_1,h^1_1,\dots,h^0_{d-1},h^1_{d-1}} \|_{L^2(\X)}^2 - \frac{\eps}{10}
\end{align*}
for any $n \in H_d$.  By the parallelogram law, this implies that
$$ \E_{h^0_i,h^1_i \in Q_i \forall i=1,\dots,d-1} \left\| \E_{h_d \in Q_d + \{n\}} T^{h_d} f^0_{h^0_1,h^1_1,\dots,h^0_{d-1},h^1_{d-1}} 
- \E_{h_d \in Q_d} T^{h_d} f^0_{h^0_1,h^1_1,\dots,h^0_{d-1},h^1_{d-1}} \right \|_{L^2(\X)}^2 
\leq \frac{4\eps}{10}
$$
for all $n \in H_d$, which by the triangle inequality implies that
$$ \E_{h^0_i,h^1_i \in Q_i \forall i=1,\dots,d-1} \left \| \E_{h_d \in Q_d + R} T^{h_d} f^0_{h^0_1,h^1_1,\dots,h^0_{d-1},h^1_{d-1}} 
- \E_{h_d \in Q_d} T^{h_d} f^0_{h^0_1,h^1_1,\dots,h^0_{d-1},h^1_{d-1}} \right \|_{L^2(\X)}^2 
\leq \frac{4\eps}{10}
$$
for any finite non-empty multiset $R$ in $H_d$.  By Cauchy-Schwarz and \eqref{shirt}, this implies (for $\eps$ small enough) that
$$ \left| \langle \vec{f} \rangle_{\Box^d_{Q_1,\dots,Q_{d-1},Q_d+R}(\X)} - \langle \vec{f} \rangle_{\Box^d_{Q_1,\dots,Q_d}(\X)} \right| \leq \eps
$$
which gives \eqref{fow}.
\end{proof}

We observe that the above argument (specialised to the $d$-symmetric case) also gives that
\begin{equation}\label{dod}
 \| f \|_{\Box^d_{H_1,\dots,H_d}(\X)} \leq  \|f\|_{\Box^d_{Q_1,\dots,Q_d}(\X)}
\end{equation}
whenever $Q_i$ are finite non-empty multisets in $H_i$ for $i=1,\dots,d$.  Taking limits in \eqref{csg}, we also have
\begin{equation}\label{csg-2}
|\langle (f_\omega)_{\omega \in\{0,1\}^d} \rangle_{\Box^d_{H_1,\dots,H_d}(\X)}|
\leq \prod_{\omega \in \{0,1\}^d} \|f\|_{\Box^d_{H_1,\dots,H_d}(\X)}^{2^d}; 
\end{equation}
in particular, writing $\langle \rangle_{U^d_H(\X)} \coloneqq \langle \rangle_{\Box^d_{H,\dots,H}(\X)}$, we have
\begin{equation}\label{csg-3}
|\langle (f_\omega)_{\omega \in\{0,1\}^d} \rangle_{U^d_H(\X)}|
\leq \prod_{\omega \in \{0,1\}^d} \|f\|_{U^d_{H}(\X)}^{2^d}
\end{equation}
for any subgroup $H$ of $G$.

Let $\X$ be a $G$-system, and let $H_1,\dots,H_d$ be subgroups of $G$.  From H\"older's inequality we see that
$$
\left|\langle (f_\omega)_{\omega \in\{0,1\}^d} \rangle_{\Box^d_{H_1,\dots,H_d}(\X)}\right| \leq
\prod_{\omega \in \{0,1\}^d} \|f_\omega\|_{L^{2^d}(\X)}$$
for all $f_\omega \in L^{2^d}(\X)$.  By duality, we can thus define a bounded multilinear \emph{dual operator}
$$ {\mathcal D}^d_{H_1,\dots,H_d}\colon L^{2^d}(\X)^{\{0,1\}^d \backslash \{0\}^d} \to L^{2^d/(2^d-1)}(\X) $$
by requiring that
\begin{equation}\label{boxh}
\left\langle (f_\omega)_{\omega \in\{0,1\}^d} \right\rangle_{\Box^d_{H_1,\dots,H_d}(\X)} = \left\langle f_{0^d}, {\mathcal D}^d_{H_1,\dots,H_d}( (f_\omega)_{\omega \in \{0,1\}^d \backslash \{0\}^d} ) \right\rangle_{L^2(\X)}.
\end{equation}
In the degenerate case $d=0$ we adopt the convention ${\mathcal D}^0 () = 1$.  We also abbreviate ${\mathcal D}^d_{H,\dots,H}$ as ${\mathcal D}^d_H$.

We can similarly define the local dual operators
$$ {\mathcal D}^d_{Q_1,\dots,Q_d}\colon L^{2^d}(\X)^{\{0,1\}^d \backslash \{0\}^d} \to L^{2^d/(2^d-1)}(\X) $$
for finite non-empty multisets $Q_i$ in $H_i$.  Theorem \ref{gawd} asserts that ${\mathcal D}^d_{Q_1,\dots,Q_d}$ converges to ${\mathcal D}^d_{H_1,\dots,H_d}$ in the weak operator topology.  We can upgrade this convergence to the strong operator topology:

\begin{proposition}\label{strong}  Let $\X$ be a $G$-system, let $H_1,\dots,H_d$ be subgroups of $G$, and let $f_\omega \in L^{2^d}(\X)$ for all $\omega \in \{0,1\}^d \backslash \{0\}^d$.  Then ${\mathcal D}^d_{Q_1,\dots,Q_d}( (f_\omega)_{\omega \in \{0,1\}^d \backslash \{0\}^d} )$ converges strongly in $L^{2^d/(2^d-1)}(\X)$ to ${\mathcal D}^d_{H_1,\dots,H_d}( (f_\omega)_{\omega \in \{0,1\}^d \backslash \{0\}^d} )$ as $(Q_1,\dots,Q_d) \to (H_1,\dots,H_d)$.
\end{proposition}

\begin{proof}  The claim is trivially true for $d=0$, so assume $d \geq 1$.  By multilinearity we may assume the $f_\omega$ are real.  By a limiting argument using H\"older's inequality, we may assume without loss of generality that $f_\omega$ all lie in $L^\infty(\X)$ and not just in $L^{2^{d}}(\X)$, in which case the ${\mathcal D}^d_{Q_1,\dots,Q_d}( (f_\omega)_{\omega \in \{0,1\}^d \backslash \{0\}^d} )$ are uniformly bounded in $L^\infty(\X)$.  Theorem \ref{gawd} already gives weak convergence, so it suffices to show that the limit
$$ \lim_{(Q_1,\dots,Q_d) \to (H_1,\dots,H_d)} {\mathcal D}^d_{Q_1,\dots,Q_d}( (f_\omega)_{\omega \in \{0,1\}^d \backslash \{0\}^d} ) $$
exists in the strong $L^2(\X)$ topology.

By the cosine rule (and the completeness of $L^2(\X)$), it suffices to show that the joint limit
$$ \lim_{(Q_1,\dots,Q_d,Q'_1,\dots,Q'_d) \to (H_1,\dots,H_d,H_1,\dots,H_d)} 
\left\langle {\mathcal D}^d_{Q_1,\dots,Q_d}( (f_\omega)_{\omega \in \{0,1\}^d \backslash \{0\}^d} ), {\mathcal D}^d_{Q'_1,\dots,Q'_d}( (f_\omega)_{\omega \in \{0,1\}^d \backslash \{0\}^d} ) \right\rangle_{L^2(\X)}$$
exists.  But the expression inside the limit can be written as an inner product
$$ \langle (\tilde f_{\omega'})_{\omega' \in \{0,1\}^{2d}} \rangle_{\Box^{2d}_{Q_1,\dots,Q_d,Q_1,\dots,Q_d}}$$
where $\tilde f_{(0,\omega)} = \tilde f_{(\omega,0)} \coloneqq  f_\omega$ for $\omega \coloneqq  \{0,1\}^d \backslash \{0\}^d$, and $\tilde f_{\omega'} \coloneqq  1$ for all other $\omega'$, and the claim then follows from Theorem \ref{gawd} (with $d$ replaced by $2d$).
\end{proof}

In the $d=1$ case, we have
$$ {\mathcal D}^1_H f = \lim_{Q \to H} \E_{a \in Q} T^a f $$
in the strong topology of $L^2(\X)$ for $f \in L^2(\X)$.  From this definition it is easy to see that ${\mathcal D}^1_H f$ is $H$-invariant, and that $f - {\mathcal D}^1_H f$ is orthogonal to all $H$-invariant functions; thus we obtain the mean ergodic theorem
\begin{equation}\label{met}
 {\mathcal D}^1_H f = \E( f | \X^H )
\end{equation}
In particular, we have
$$ \| f \|_{\Box^1_H(\X)} = \| \E( f | \X^H )\|_{L^2(\X)}$$
for all $f \in L^2(\X)$.  For $d > 1$, we have the easily verified identity
$$ \| f \|_{\Box^d_{Q_1,\dots,Q_d}(\X)}^{2^d} = \E_{h \in Q_d - Q_d} \| f \overline{T^h f} \|_{\Box^{d-1}_{Q_1,\dots,Q_{d-1}}(\X)}^{2^{d-1}} $$
which on taking limits using Theorem \ref{gawd} and dominated convergence implies that
$$ \| f \|_{\Box^d_{H_1,\dots,H_d}(\X)}^{2^d} = \lim_{Q_d \to H_d} \E_{h \in Q_d - Q_d} \| f \overline{T^h f} \|_{\Box^{d-1}_{H_1,\dots,H_{d-1}}(\X)}^{2^{d-1}} $$
for all $f \in L^{2^d}(\X)$. In particular
$$ \| f \|_{U^d_H(\X)}^{2^d} = \lim_{Q \to \infty} \E_{h \in Q - Q} \| f \overline{T^h f} \|_{U^{d-1}_H(\X)}^{2^{d-1}}.$$
From this, we see that the seminorm $U^d_H(\X)$ defined here agrees with the Gowers-Host-Kra seminorm from \cite{hk}; see \cite[Appendix A]{btz} for details.

A key property concerning dual functions is that they are closed under multiplication after taking convex closures:

\begin{proposition}\label{convex}  Let $\X$ be a $G$-system, and let $H_1,\dots,H_d$ be subgroups of $G$.  Let $B$ be the closed convex hull (in $L^2(\X)$) of all functions of the form ${\mathcal D}^d_{H_1,\dots,H_d}( (f_\omega)_{\omega \in \{0,1\}^d \backslash \{0\}^d} )$, where the $f_\omega$ all lie in the closed unit ball of $L^\infty(\X)$.  Then $B$ is closed under multiplication: if $F,F' \in B$, then $FF' \in B$.
\end{proposition}

\begin{proof}  We may assume $d \geq 1$, as the $d=0$ case is trivial.  By convexity and a density argument, we may assume that $F,F'$ are themselves dual functions, thus
$$
F = {\mathcal D}^d_{H_1,\dots,H_d}( (f_\omega)_{\omega \in \{0,1\}^d \backslash \{0\}^d} )$$
and
$$
F' = {\mathcal D}^d_{H_1,\dots,H_d}( (f'_\omega)_{\omega \in \{0,1\}^d \backslash \{0\}^d} )$$
for some $f_\omega, f'_\omega$ in the closed unit ball of $L^\infty(\X)$.

By Proposition \ref{strong}, we can write $FF'$ as 
\begin{align*}
&\lim_{(Q_1,\dots,Q_d) \to (H_1,\dots,H_d)} \lim_{(Q'_1,\dots,Q'_d) \to (H_1,\dots,H_d)}
\E_{h_i \in Q_i-Q_i \forall i=1,\dots,d} 
\E_{k_i \in Q'_i-Q'_i \forall i=1,\dots,d}  \\
&\quad \prod_{\omega \in \{0,1\}^d \backslash \{0\}^d} {\mathcal C}^{|\omega|-1} 
(T^{\omega_1 h_1 + \dots + \omega_d h_d} f_\omega T^{\omega_1 k_1 + \dots + \omega_d k_d} f'_\omega)
\end{align*}
where the limits are in the $L^2(\X)$ topology.
For any given $h \in G$, averaging a bounded sequence over $Q'_i$ and averaging over $Q'_i+h$ are approximately the same if $Q'_i$ is sufficiently large (e.g. if $Q'_i$ is larger than the progression $\{ 0, h, \dots, Nh \}$ for some large $h$).  Because of this, we can shift the $k_i$ variable by $h_i$ in the above expression without affecting the limit.  In other words, $FF'$ is equal to
\begin{align*}
&\lim_{(Q_1,\dots,Q_d) \to (H_1,\dots,H_d)} \lim_{(Q'_1,\dots,Q'_d) \to (H_1,\dots,H_d)}
\E_{h_i \in Q_i-Q_i \forall i=1,\dots,d} 
\E_{k_i \in Q'_i-Q'_i \forall i=1,\dots,d} \\
&\quad \prod_{\omega \in \{0,1\}^d \backslash \{0\}^d} {\mathcal C}^{|\omega|-1} 
(T^{\omega_1 h_1 + \dots + \omega_d h_d} f_\omega T^{\omega_1 (h_1+k_1) + \dots + \omega_d (h_d+k_d)} f'_\omega).
\end{align*}
By Proposition \ref{strong} (with $d$ replaced by $2d$), the above expression remains convergent if we work with the joint limit
$\lim_{(Q_1,\dots,Q_d,Q'_1,\dots,Q'_d) \to (H_1,\dots,H_d,H_1,\dots,H_d)}$.  In particular, we may interchange limits and write the above expression as
\begin{align*}
&\lim_{(Q'_1,\dots,Q'_d) \to (H_1,\dots,H_d)}
\lim_{(Q_1,\dots,Q_d) \to (H_1,\dots,H_d)} 
\E_{h_i \in Q_i-Q_i \forall i=1,\dots,d} 
\E_{k_i \in Q'_i-Q'_i \forall i=1,\dots,d} \\
&\quad \prod_{\omega \in \{0,1\}^d \backslash \{0\}^d} {\mathcal C}^{|\omega|-1} 
(T^{\omega_1 h_1 + \dots + \omega_d h_d} f_\omega T^{\omega_1 (h_1+k_1) + \dots + \omega_d (h_d+k_d)} f'_\omega).
\end{align*}
Computing the inner limit, this simplifies to
$$
\lim_{(Q'_1,\dots,Q'_d) \to (H_1,\dots,H_d)}
\E_{k \in Q'_i-Q'_i \forall i=1,\dots,d} 
{\mathcal D}^d_{H_1,\dots,H_d}\left( ( f_\omega T^{\omega_1 k_1 + \dots + \omega_d k_d} f'_\omega )_{\omega \in \{0,1\}^d \backslash \{0\}^d} \right).$$
This is the strong limit of convex averages of elements of $B$, and thus lies in $B$ as required.
\end{proof}

We can now construct the characteristic factor (cf. \cite[Lemma 4.3]{hk}):

\begin{theorem}[Existence of characteristic factor]\label{charfac}  Let $\X$ be a $G$-system, and let $H_1,\dots,H_d$ be subgroups of $G$.  
Then there is a unique $G$-system $\ZZ^{{<}d}_{H_1,\dots,H_d}(\X) = (X, {\mathcal Z}^{{<}d}_{H_1,\dots,H_d}(\X), \mu)$ with the property that
$$ \|f\|_{\Box^d_{H_1,\dots,H_d}(\X)} = 0 \iff \E( f | \ZZ^{{<}d}_{H_1,\dots,H_d}(\X) ) = 0$$
for all $f \in L^\infty(\X)$.  

Furthermore, a function $f \in L^\infty(\X)$ lies in $L^\infty( \ZZ^{{<}d}_{H_1,\dots,H_d}(\X) )$ if and only if it is the limit (in $L^2(\X)$) of a sequence of linear combinations of dual functions ${\mathcal D}^d_{H_1,\dots,H_d}( (f_\omega)_{\omega \in \{0,1\}^d \backslash \{0\}^d} )$ with $f_\omega \in L^\infty(\X)$, with the sequence uniformly bounded in $L^\infty(\X)$.
\end{theorem}

\begin{proof} The uniqueness is clear.  To prove existence, let $B \subset L^\infty(\X)$ be the set in Proposition \ref{convex}.
Define ${\mathcal Z}^{{<}d}_{H_1,\dots,H_d}(\X)$ to be the set of measurable sets $E$ such that $1_E$ is expressible as the limit (in $L^2(\X)$) of a uniformly bounded sequence in the set $\R B \coloneqq  \bigcup_{\lambda > 0} \lambda \cdot B$.  It is easy to see that ${\mathcal Z}^{{<}d}_{H_1,\dots,H_d}(\X)$ is a $G$-invariant $\sigma$-algebra, so that $\ZZ^{{<}d}_{H_1,\dots,H_d}(\X) \coloneqq  (X, {\mathcal Z}^{{<}d}_{H_1,\dots,H_d}(\X), \mu)$ is a $G$-system.  If $f \in \R B$, then from Proposition \ref{convex} any polynomial combination of $f$ lies in $\R B$ also; from the Weierstrass approximation theorem we conclude that the level sets $\{ f \geq \lambda \}$ are measurable in ${\mathcal Z}^{{<}d}_{H_1,\dots,H_d}(\X)$ for almost every $\lambda$.  In particular, $f$ itself is measurable in this factor, and hence $\R B \subset L^\infty( \ZZ^{{<}d}_{H_1,\dots,H_d}(\X) )$.

Let $f \in L^\infty(\X)$.  If $\|f\|_{\Box^d_{H_1,\dots,H_d}(\X)} > 0$, then by \eqref{boxh} $f$ is not orthogonal to the dual function ${\mathcal D}^d_{H_1,\dots,H_d}( (f)_{\omega \in \{0,1\}^d \backslash \{0\}^d} )$, which lies in $\R B$ and is hence in $L^\infty( \ZZ^{{<}d}_{H_1,\dots,H_d}(\X) )$.  This implies that $\E( f | \ZZ^{{<}d}_{H_1,\dots,H_d}(\X) ) \neq 0$.

Conversely, suppose that $\|f\|_{\Box^d_{H_1,\dots,H_d}(\X)} = 0$.  From the Cauchy-Schwarz-Gowers inequality \eqref{csg-3} we thus have
$$
\langle (f_\omega)_{\omega \in\{0,1\}^d} \rangle_{\Box^d_{H_1,\dots,H_d}(\X)} = 0$$
whenever $f_\omega \in L^\infty(\X)$ is such that $f_{0^d} = f$.  From \eqref{boxh} we conclude that $f$ is orthogonal to $B$, and hence to
$L^\infty( \ZZ^{{<}d}_{H_1,\dots,H_d}(\X) )$.  Thus $\E( f | \ZZ^{{<}d}_{H_1,\dots,H_d}(\X) ) = 0$.
\end{proof}

We have a basic corollary of Theorem \ref{charfac} (cf. \cite[Proposition 4.6]{hk}):

\begin{corollary}\label{locx}  Let $\X$ be a $G$-system, let $\Y$ be a factor of $\X$, and let $H_1,\dots,H_d$ be subgroups of $G$.  Then
$$
L^\infty( \ZZ^{{<}d}_{H_1,\dots,H_d}(\Y) ) = L^\infty(\Y) \cap L^\infty( \ZZ^{{<}d}_{H_1,\dots,H_d}(\X) )$$
or equivalently
$$ \ZZ^{{<}d}_{H_1,\dots,H_d}(\Y)  = \Y \wedge \ZZ^{{<}d}_{H_1,\dots,H_d}(\X).$$
\end{corollary}

\begin{proof}  If $f \in L^\infty( \ZZ^{{<}d}_{H_1,\dots,H_d}(\Y) )$, then $f \in L^\infty(\Y)$ and (by Theorem \ref{charfac}) $f$ can be expressed as the limit of dual functions of functions in $L^\infty(\Y)$, and hence in $L^\infty(\X)$, and so the inclusion
$$
L^\infty( \ZZ^{{<}d}_{H_1,\dots,H_d}(\Y) ) \subset L^\infty(\Y) \cap L^\infty( \ZZ^{{<}d}_{H_1,\dots,H_d}(\X) )$$
then follows from another application of Theorem \ref{charfac}.  Conversely, if $f \in L^\infty(\Y) \cap L^\infty( \ZZ^{{<}d}_{H_1,\dots,H_d}(\X) )$, then by Theorem \ref{charfac}, $f$ is orthogonal to all functions in $L^\infty(\X)$ of vanishing $\Box^d_{H_1,\dots,H_d}$ norm, and hence also to all functions in $L^\infty(\Y)$ of vanishing $\Box^d_{H_1,\dots,H_d}$ norm.  The claim then follows from another application of Theorem \ref{charfac}.
\end{proof}

We now can deduce Theorem \ref{concat-uap-erg} and Theorem \ref{concat-box-erg} from Theorem \ref{concat-uap} and Theorem \ref{concat-box} respectively.  At this point we will begin to need the hypothesis that $G$ is the sum of a finitely generated group and a profinite group.  We just give the argument for Theorem \ref{concat-uap-erg}; the argument for Theorem \ref{concat-box-erg} is completely analogous and is left to the reader.  Let $(\X,T)$, $G$, $H_1, H_2, d_1, d_2$ be as in Theorem \ref{concat-uap-erg}.  By Corollary \ref{locx}, we may assume without loss of generality that
\begin{equation}\label{xop}
 \X \equiv \ZZ^{{<}d_1}_{H_1}(\X) \equiv \ZZ^{{<}d_2}_{H_2}(\X).
\end{equation}
Indeed, if we set 
$$ \X' := \ZZ^{{<}d_1}_{H_1}(\X) \wedge \ZZ^{{<}d_2}_{H_2}(\X),$$
we see from Corollary \ref{locx} that $\X'$ obeys the condition \eqref{xop}, and that $\ZZ^{{<}d_1+d_2-1}_{H_1+H_2}(\X')$ is a factor of $\ZZ^{{<}d_1+d_2-1}_{H_1+H_2}(\X)$, so the general case of Theorem \ref{concat-box-erg} can be derived from this special case.
By Theorem \ref{charfac}, it suffices to show that if $f,g \in L^\infty(\X)$ and $\|g\|_{U^{d_1+d_2-1}_{H_1+H_2}(\X)}=0$, then $f,g$ are orthogonal.

For $i=1,2$, let $B_i$ be the closed convex hull (in $L^2(\X)$) of the dual functions
${\mathcal D}^{d_i}_{H_i}( (f_\omega)_{\omega \in \{0,1\}^{d_i} \backslash \{0\}^{d_i}} )$ with $f_\omega$ in the closed unit ball of $L^\infty(\X)$.  By Theorem \ref{charfac}, we see that for every $\delta>0$, there exists a real number $F(\delta)$ such that $f$ lies within $\delta$ in $L^2(\X)$ norm of both $F(\delta) \cdot B_1$ and $F(\delta) \cdot B_2$.  On the other hand, from the Cauchy-Schwarz-Gowers inequality \eqref{csg} one has
$$ |\langle f_i, f' \rangle_{L^2(\X)}| \leq \|f'\|_{U^{d_i}_{H_i}(\X)}$$
whenever $i=1,2$ and $f_i \in B_i$, and $f'$ is in the closed unit ball of $L^\infty(\X)$.  We conclude that
$$ |\langle f, f' \rangle_{L^2(\X)}| \leq \delta + F(\delta) \eps$$
whenever $f'$ is in the closed unit ball of $L^\infty(\X)$ with $\|G\|_{U^{d_i}_{H_i}(\X)} \leq \eps$ for some $i=1,2$.  By \eqref{dod} and \eqref{upqd} we conclude that
$$
\| f\|_{U^{d_i}_{Q_i}(\X)^*,\eps} \leq \inf_{\delta>0} (\delta + F(\delta) \eps) $$
for any coset progression $Q_i$ in $H_i$.  The right-hand side goes to zero as $\eps \to 0$. 

Since $G$ is the sum of a finitely generated group and a profinite group, the subgroups $H_1,H_2$ are also.  In particular, for each $i=1,2$, we may obtain a F{\o}lner sequence $Q_{i,n}$ for $H_i$ of coset progressions of bounded rank (thus for any $g \in H_i$, $Q_{i,n}$ and $Q_{i,n}+h$ differ (as multisets) by $o( |Q_{i,n}|)$ elements as $n \to \infty$).  (Indeed, if $H_i$ is finitely generated, one can use ordinary progressions as the F{\o}lner sequence, whereas if $H_i$ is at most countable and bounded torsion, one can use subgroups for the F{\o}lner sequence, and the general case follows by addition.)  Applying Theorem \ref{concat-uap}, we conclude that
$$
\| f\|_{U^{d_1+d_2-1}_{Q_{1,n}+Q_{2,n}}(\X)^*,\eps} \leq c(\eps)$$
for some $c(\eps)$ independent of $n$ that goes to zero as $\eps \to 0$.  Since the $Q_{i,n}$ are F{\o}lner sequences for $H_i$, $Q_{1,n}+Q_{2,n}$ is a F{\o}lner sequence for $H_1+H_2$.  In particular, by \eqref{fol-joint} one has
$$
\lim_{n \to \infty} \| g \|_{U^{d_1+d_2-1}_{Q_{1,n}+Q_{2,n}}(\X)} = \lim_{Q \to H_1+H_2} \| g \|_{U^{d_1+d_2-1}_Q(\X)}
= \|g\|_{U^{d_1+d_2-1}_{H_1+H_2}(\X)} = 0$$
and thus
$$ |\langle f,g \rangle_{L^2(\X)}| \leq c(\eps)$$
for every $\eps>0$.  Sending $\eps \to 0$, we obtain the claim.

\section{An ergodic theory argument}\label{erg2-sec}

We now give an ergodic theory argument that establishes Theorem \ref{concat-uap-erg-weak}.  The arguments here rely heavily on those in \cite{hk}, but are not needed elsewhere in this paper. 

For this section it will be convenient to restrict attention to $G$-systems $(\X,T)$ in which $\X$ is a compact metric space with the Borel $\sigma$-algebra, in order to access tools such as disintegration of measure.  The requirement of being a compact metric space is stronger than our current hypothesis that $\X$ is countably generated modulo null sets; however, it is known (see \cite[Proposition 5.3]{furst} that every $G$-system that is countably generated modulo null sets is equivalent (modulo null sets) to another $G$-system $(\X',T')$ in which $\X'$ is a compact metric space with the Borel $\sigma$-algebra.  The corresponding characteristic factors such as $\ZZ^{{<}d_1}_{H_1}(\X)$ are also equivalent up to null sets (basically because the Gowers-Host-Kra seminorms are equivalent).  Because of this, we see that to prove Theorem \ref{concat-uap-erg-weak} it suffices to do so when $\X$ is a compact metric space.

We now recall the construction of cubic measures from \cite{hk}, which in \cite{host} was generalised\footnote{Strictly speaking, the constructions in \cite{hk}, \cite{host} focus on the special case when $H_1,\dots,H_d$ are cyclic groups, but they extend without much difficulty to the more general setting described here.} to our current setting of arbitrary actions of multiple subgroups of an at most countable additive group.  

\begin{definition}[Cubic measures]  Let $(\X,T)$ be a $G$-system with $\X = (X, {\mathcal B}, \mu)$, and let $H_1,\dots,H_d$ be subgroups of $G$.  We define the $G$-system $(\X^{[d]}_{H_1,\dots,H_d}, T)$ by setting $\X^{[d]}_{H_1,\dots,H_d} \coloneqq  (X^{[d]}, {\mathcal B}^{[d]}, \mu^{[d]}_{H_1,\dots,H_d})$, where $X^{[d]} \coloneqq  X^{\{0,1\}^d}$ is the set of tuples $(x_\omega)_{\omega \in \{0,1\}^d}$ with $x_\omega\in X$ for all $\omega \in \{0,1\}^d$, ${\mathcal B}^{[d]}$ is the product measure, and $\mu^{[d]}_{H_1,\dots,H_d}$ is the unique probability measure such that
\begin{equation}\label{foa}
 \int_{X^{[d]}} \bigotimes_{\omega \in \{0,1\}^d} {\mathcal C}^{|\omega|} f_\omega\ d\mu^{[d]}_{H_1,\dots,H_d} = \langle (f_\omega)_{\omega \in \{0,1\}^d} \rangle_{\Box^d_{H_1,\dots,H_d}(\X)}
\end{equation}
for all $f_\omega \in L^\infty(\X)$, where the tensor product $\bigotimes_{\omega \in \{0,1\}^d} {\mathcal C}^{|\omega|} f_\omega$ is defined as
$$ \bigotimes_{\omega \in \{0,1\}^d} {\mathcal C}^{|\omega|} f_\omega( (x_\omega)_{\omega \in \{0,1\}^d} ) \coloneqq  
\prod_{\omega \in \{0,1\}^d} {\mathcal C}^{|\omega|} f_\omega(x_\omega).$$
Finally, the shift $T$ on $\X^{[d]}_{H_1,\dots,H_d}$ is defined via the diagonal action:
$$ T^g( x_\omega)_{\omega \in \{0,1\}^d} \coloneqq  (T^g x_\omega)_{\omega \in \{0,1\}^d}.$$
If $H_1=\dots=H_d=H$, we abbreviate $\X^{[d]}_{H_1,\dots,H_d}$ as $\X^{[d]}_H$, and $\mu^{[d]}_{H_1,\dots,H_d}$ as $\mu^{[d]}_H$.
\end{definition}

The uniqueness of the measure $\mu^{[d]}_{H_1,\dots,H_d}$ is clear.  To construct the measure $\mu^{[d]}_{H_1,\dots,H_d}$, one can proceed inductively as in \cite{hk}, setting $\mu^{[0]} = \mu$ and then defining $\mu^{[d]}_{H_1,\dots,H_d}$ for $d \geq 1$ to be the relative product of two copies of $\mu^{[d-1]}_{H_1,\dots,H_{d-1}}$ over the $H_d$-invariant factor of ${\mathcal B}^{[d-1]}$; this relative product can be constructed if $\X$ is a compact metric space.  Alternatively, one can verify that the formula \eqref{foa} defines a probability premeasure on the Boolean subalgebra of ${\mathcal B}^{[d]}$ generated by the product sets $\prod_{\omega \in \{0,1\}^d} E_\omega$ with $E_\omega \in {\mathcal B}$, extending this premeasure to a probability measure $\mu^{[d]}_{H_1,\dots,H_d}$ on ${\mathcal B}^{[d]}$ by the Carath\'eodory extension theorem, and then verifying that the resulting measure continues to obey \eqref{foa}.  We leave the details of these arguments to the interested reader.  Once this measure is constructed, it is easy to see that the diagonal action of $T$ preserves the measure $\mu^{[d]}_{H_1,\dots,H_d}$, and so $(\X^{[d]}_{H_1,\dots,H_d},T)$ is indeed a $G$-system.  Observe also that every $k$-dimensional face of $\{0,1\}^d$ with $0 \leq k \leq d$ induces an obvious factor map from $\X^{[d]}_{H_1,\dots,H_d}$ to $\X^{[k]}_{H_{i_1},\dots,H_{i_k}}$, where $1 \leq i_1 < \dots < i_k \leq d$ are the indices of the coordinate vectors parallel to this face.  Also, from Theorem \ref{gawd} we see that permutation of the groups $H_1,\dots,H_d$ leaves the systems $(\X^{[d]}_{H_1,\dots,H_d},T)$ unchanged up to the obvious relabeling isomorphism, thus for instance $\X^{[2]}_{H_1,H_2} \equiv \X^{[2]}_{H_2,H_1}$, with the isomorphism given by the map $(x_{00}, x_{10}, x_{01}, x_{11}) \mapsto (x_{00}, x_{01}, x_{10}, x_{11})$.

One can informally view the probability space $\X^{[d]}_{H_1,\dots,H_d}$ as describing the distribution of certain $d$-dimensional ``parallelopipeds'' in $X$, where the $d$ ``directions'' of the parallelopiped are ``parallel'' to $H_1,\dots,H_d$.  We will also need the following variant of these spaces, which informally describes the distribution of ``$L$-shaped'' objects in $X$.

\begin{definition}[$L$-shaped measures]\label{L-def}  Let $(\X,T)$ be a $G$-system with $\X = (X,{\mathcal B},\mu)$, and let $H,K$ be subgroups of $G$.  We define the system $(\X^L_{H,K},T)$ by setting $\X^{L}_{H,K} \coloneqq  (X^{L}, {\mathcal B}^{L}, \mu^{L}_{H,K})$, where $X^{L} \coloneqq  X^{\{00,01,10\}}$ is the set of tuples $(x_{00}, x_{10}, x_{01})$ with $x_{00}, x_{10}, x_{01} \in X$, ${\mathcal B}^{L}$ is the product measure, and $\mu^{L}_{H,K}$ is the unique probability measure such that
\begin{equation}\label{foa-l}
 \int_{X^{L}} f_{00} \otimes \overline{f_{10}} \otimes \overline{f_{01}}\ d\mu^{L}_{H,K} = \langle f_{00}, f_{10}, f_{01},1 \rangle_{\Box^2_{H,K}(\X)}
\end{equation}
for all $f_{00}, f_{01}, f_{10} \in L^\infty(\X)$.  The shift $T$ on $\X^L_{H,K}$ is defined by the diagonal action.
\end{definition}

The system $\X^L_{H,K}$ is clearly a factor of $\X^{[2]}_{H,K}$, and also has factor maps to $\X^{[1]}_H$ and $\X^{[1]}_K$ given by $(x_{00},x_{10}, x_{01}) \mapsto (x_{00}, x_{10})$ and $(x_{00}, x_{10}, x_{01}) \mapsto (x_{00}, x_{01})$ respectively.  A crucial fact for the purposes of establishing concatenation is that $\X^L_{H,K}$ additionally has a \emph{third} factor map to the space $\X^{[1]}_{H+K}$:

\begin{lemma}[Third factor of $L$-shapes]\label{L-proj}  Let $(\X,T)$ be a $G$-system, and let $H,K$ be subgroups of $G$.  Then the map $(x_{00}, x_{10}, x_{01}) \mapsto (x_{10}, x_{01})$ is a factor map from $(\X^L_{H,K},T)$ to $(\X^{[1]}_{H+K},T)$.
\end{lemma}

Informally, this lemma reflects the obvious fact that if $x_{00}$ and $x_{10}$ are connected to each other by an element of the $H$ action, and $x_{00}$ and $x_{01}$ are connected to each other by an element of the $K$ action, then $x_{10}$ and $x_{01}$ are connected to each other by an element of the $H+K$ action.

\begin{proof}  If $f_{00}, f_{10}, f_{01} \in L^\infty(\X)$ are real-valued, then by Theorem \ref{gawd}, Definition \ref{gaw}, and Definition \ref{L-def} we have
\begin{align*}
 \int_{X^L} f_{00} \otimes f_{10} \otimes f_{01}\ d\mu^L_{H,K} &= \lim_{(Q,R) \to (H,K)} \E_{h,h' \in Q} \E_{k,k' \in R} \int_X T^{h+k} f_{00} T^{h'+k} f_{10} T^{h+k'} f_{01}\ d\mu \\
&=
\lim_{(Q,R) \to (H,K)} \E_{a \in Q-Q} \E_{b \in R-R} \int_X f_{00} T^a f_{10} T^b f_{01}\ d\mu \\
&=
\lim_{(Q,R) \to (H,K)} \int_X f_{00} (\E_{a \in Q-Q} T^a f_{10}) (\E_{b \in R-R} T^b f_{01})\ d\mu\\
&=
\lim_{(Q,R) \to (H,K)} \int_X f_{00} {\mathcal D}^{[1]}_Q f_{10} {\mathcal D}^{[1]}_R f_{01}\ d\mu \\
&= \int_X f_{00} {\mathcal D}^{[1]}_H f_{10} {\mathcal D}^{[1]}_K f_{01}\ d\mu\\
&= 
\int_X f_{00} \E(f_{10}|\X^H) \E(f_{01}|\X^K)\ d\mu
\end{align*}
where we use Proposition \ref{strong} and \eqref{met} in the last two lines.  Specialising to the case $f_{00}=1$, we conclude in particular that
\begin{align*}
 \int_{X^L} 1 \otimes f_{10} \otimes f_{01}\ d\mu^L_{H,K} &= \int_X \E(f_{10}|\X^H) \E(f_{01}|\X^K)\ d\mu  \\
&= \int_X \E(f_{10}|\X^H) \E( \E(f_{01}|\X^{K}) | \X^H )\ d\mu \\
&= \int_X \E(f_{10}|\X^H) \E(f_{01}|\X^{H+K})\ d\mu \\
&= \int_X \E(\E(f_{10}|\X^H)|\X^{H+K}) \E(f_{01}|\X^{H+K})\ d\mu \\
&= \int_X \E(f_{10}|\X^{H+K}) \E(f_{01}|\X^{H+K})\ d\mu\\
& = \int_X f_{10} {\mathcal D}^{[1]}_{H+K} f_{01}\ d\mu \\
&= \langle f_{10}, f_{01} \rangle_{\Box^{[1]}_{H+K}(\X)} \\
&= \int_{X^{[1]}} f_{10} \otimes f_{01}\ d\mu^{[1]}_{H+K}(\X),
\end{align*}
where the key identity $\E( \E(f_{01}|\X^{K}) | \X^H ) = \E(f_{01}|\X^{H+K})$ can be established for instance from \eqref{met}, and \eqref{met} is also used in the third to last line.  The claim follows.
\end{proof}

We now begin the proof of Theorem \ref{concat-uap-erg-weak}.
Let $(\X,T)$, $G$, $H_1,H_2$, $d_1,d_2$ be as in Theorem \ref{concat-uap-erg}.  We begin with a few reductions.  By induction we may assume that the claim is already proven for smaller values of $d_1+d_2$.
By shrinking $G$ if necessary, we may assume that $G = H_1 + H_2$ (note that replacing $G$ with $H_1+H_2$ does not affect factors such as $\ZZ^{{<}d_1}_{H_1}(\X)$). 

Next, we observe that we may reduce without loss of generality to the case where the action of $G$ on $(\X,T)$ is ergodic.  To see this, we argue as follows.  As $\X$ was assumed to be a compact metric space, we have an ergodic decomposition $\mu = \int_Y \mu_y\ d\nu(y)$ for some probability space $(Y,\nu)$ (the invariant factor $\X^G$ of $X$), and some probability measures $\mu_y$ on $\X$ depending measurably on $y$, and ergodic in $G$ for almost every $y$; see \cite[Theorem 5.8]{furst} or \cite[Theorem 6.2]{eins}.  Let $\X_y = (X, {\mathcal B}, \mu_y)$ denote the components of this decomposition.  From \eqref{ghk-def}, \eqref{boxnorm-def} we have the identity
$$ \| f \|_{U^{d_1}_{H_1}(\X)}^{2^{d_1}} = \int_Y \| f \|_{U^{d_1}_{H_1}(\X_y)}^{2^{d_1}}\ d\nu(y)$$
for any $f \in L^\infty(\X)$.  We conclude that a bounded measurable function $f \in L^\infty(\X)$ vanishes in $U^{d_1}_{H_1}(\X)$ if and only if it vanishes in $U^{d_1}_{H_1}(\X_y)$ for almost every $y$.  By \eqref{fuz}, this implies that $f$ is measurable (modulo null sets) with respect to $\ZZ^{{<}d_1}_{H_1}(\X)$ if and only if it is measurable (modulo null sets) with respect to $\ZZ^{{<}d_1}_{H_1}(\X_y)$ for almost every $y$.  Similarly for $\ZZ^{{<}d_2}_{H_2}(\X)$ and $\ZZ^{{<}d}_{H_1+H_2}(\X)$.  From this it is easy to see that Theorem \ref{concat-uap-erg} for $\X$ will follow from Theorem \ref{concat-uap-erg} for almost every $\X_y$.  Thus we may assume without loss of generality that the system $(\X,T)$ is $G$-ergodic.

Next, if we set $\X' \coloneqq  \ZZ^{{<}d_1}_{H_1}(\X) \wedge \ZZ^{{<}d_2}_{H_2}(\X)$, we see from Corollary \ref{locx} (as in the proof of Theorem \ref{concat-uap-erg}) that
$$ \ZZ^{{<}d_1+d_2-1}_{H_1+H_2}(\X') = \X' \wedge \ZZ^{{<}d_1+d_2-1}_{H_1+H_2}(\X).$$
Thus we may replace $\X$ by $\X'$ and assume without loss of generality that
$$ \X \equiv \ZZ^{{<}d_1}_{H_1}(\X) \equiv \ZZ^{{<}d_2}_{H_2}(\X).$$
Following \cite{hk} (somewhat loosely\footnote{In particular, what we call order $<d$ here would be called order $d-1$ in \cite{hk}.}), let us say that a $G$-system $\X$ is of \emph{$H$-order ${<} d$} if $\X \equiv \ZZ^{{<}d}(\X)$; thus we are assuming that $\X$ has $H_1$-order ${<}d_1$ and $H_2$-order ${<}d_2$.   Our task is now to show that $\X$ has $G$-order ${<}d_1+d_2-1$.  For future reference we note from Corollary \ref{locx} that any factor of a $G$-system with $H$-order ${<}d$ also has $H$-order ${<}d$.

For future reference, we observe that the property of having $G$-order ${<}d$ is also preserved under taking Host-Kra powers:

\begin{lemma}\label{glob}  Let $H,H'$ be subgroups of $G$. Let $\Y$ be a $G$-system with $H$-order ${<}d$ for some $d \geq 1$.  Then $\Y^{[1]}_{H'}$ is also of $H$-order ${<}d$.
\end{lemma}

\begin{proof}  The space $\Y^{[1]}_{H'}$ contains two copies of $\Y$ as factors, which we will call $\Y_1$ and $\Y_2$.  By Corollary \ref{locx}, we have $\Y_i = \ZZ^{{<}d}_H(\Y_i) \leq \ZZ^{{<}d}_H(\Y)$ for $i=1,2$, so that $\Y_1 \vee \Y_2 \leq \ZZ^{{<}d}_H(\Y)$. But $\Y_1 \vee \Y_2 = \Y$, and the claim follows.
\end{proof}

If $d_1=1$, then every function in $L^\infty(\X)$ is $H_1$-invariant, and it is then easy to see that the $U^{d_2}_{H_2}(\X)$ and $U^{d_2}_{H_1+H_2}(\X)$ seminorms agree.  By \eqref{fuz}, we conclude that $\ZZ^{{<}d_2}_{H_2}(\X)$ is equivalent to $\ZZ^{{<}d_2}_{H_1+H_2}(\X)$, and the claim follows.  Similarly if $d_2=1$.  Thus we may assume that $d_1,d_2 > 1$.  

We now set
\begin{equation}\label{ydef}
 \Y \coloneqq  \ZZ^{{<}d_1+d_2-2}_{H_1+H_2}(\X); 
\end{equation}
From the induction hypothesis we have
\begin{equation}\label{yo}
\ZZ^{{<}d_1-1}_{H_1}(\X), \ZZ^{{<}d_2-1}_{H_2}(\X) \leq \Y \leq \X.
\end{equation}

We now analyse $\X$ as an extension over $\Y$, following the standard path in \cite{fw}, \cite{hk}.  Given a subgroup $H$ of $G$, we say (as in \cite{fk}) that $\X$ is a \emph{compact extension} of $\Y$ with respect to the $H$ action if any function in $L^\infty(\X)$ can be approximated to arbitrary accuracy (in $L^2(\X)$) by an $H$-invariant finite rank module over $L^\infty(\Y)$.  We have

\begin{proposition}[Compact extension]  Let $i=1,2$.  Then $\X$ is a compact extension of $\Y$ with respect to the $H_i$ action.
 \end{proposition}

\begin{proof}  Since $\X$ is of $H_i$-order ${<}d_i$ and $\Y \geq \ZZ^{{<}d_i-1}_{H_i}(\X)$ as a factor, the claim follows from \cite[Lemma 6.2]{hk}.  (The argument there is stated in the case $H_i=G=\Z$, but it extends without any modification of the argument to the case when $H_i$ is a subgroup of an arbitrary countable abelian group $G$.)
\end{proof}

We now invoke \cite[Proposition 2.3]{fk}, which asserts that if one system $\X$ is a compact extension of another $\Y$ for two commuting group actions $H, K$, then it is also a compact extension for the combined action of $H+K$.  Since we are assuming $G = H_1+H_2$, we conclude that $\X$ is a compact extension of $\Y$ as a $G$-system.

Since $\X$ is also assumed to be $G$-ergodic, we may now use the Mackey theory of isometric extensions from \cite[\S 5]{fw}, and conclude that $\X$ is an \emph{isometric extension} of $\Y$ in the sense that $\X$ is equivalent to a system of the form $\Y \times_\rho K/L$ where $K = (K,\cdot)$ is a compact group, $L$ is a closed subgroup of $K$, $\rho\colon G \times \Y \to K$ is a measurable function obeying the \emph{cocycle equation}
\begin{equation}\label{rho}
 \rho(g + h, x) = \rho(g,T^h x) \rho(h, x)
\end{equation}
and the $\Y \times_\rho K/L$ is the product of the probability spaces $\Y$ and $K/L$ (the latter being given the Haar measure) with action given by the formula
$$ T^g (y, t) \coloneqq  (T^g y, \rho(g,y) t)$$
for all $y \in Y$ and $t \in K/L$.

We now give the standard abelianisation argument, originating from \cite{fw} and used also in \cite{hk}, that allows us to reduce to the case of abelian extensions.

\begin{proposition}[Abelian extension]  With $\rho, K, L$ as above, $L$ contains the commutator group $[K,K]$.  In particular, after quotienting out by $L$ we may assume without loss of generality that $K$ is abelian and $L$ is trivial.
\end{proposition}

\begin{proof}  This is essentially \cite[Proposition 6.3]{hk}, but for the convenience of the reader we provide an arrangement (essentially due to Szegedy \cite{szegedy-old}) of the argument here, which uses the action of $H_1$ but does not presume $H_1$-ergodicity.  

We identify $\X$ with $\Y \times_\rho K/L$. 
For any $k \in K$, we define the rotation actions $\tau_k$ on $L^\infty(\X)$ by
$$ \tau_k f( y, t ) \coloneqq  f( y, k^{-1} t ).$$
At present we do \emph{not} know that these actions commute with the shift $T^g$; however they are certainly measure-preserving, so in particular $\E( f - \tau_k f | \Y ) = 0$ for all $f \in L^\infty(\X)$.  Since $\Y \geq \ZZ^{{<}d_1-1}_{H_1}(\X)$, this implies from \eqref{fuz} that $\|f-\tau_k f\|_{U^{d_1-1}_{H_1}(\X)}=0$ for all $f \in L^\infty(\X)$ and $k \in K$.  From the Cauchy-Schwarz-Gowers inequality \eqref{csg-2} we conclude that the inner product $\langle (f_\omega)_{\omega \in \{0,1\}^{d_1-1}} \rangle_{U^{d_1-1}_{H_1}(\X)}$ vanishes whenever one of the functions $f_\omega \in L^\infty(\X)$ is of the form $f - \tau_k f$ for some $f \in L^\infty(\X)$ and $k \in K$.  By linearity, this implies that the inner product $\langle (f_\omega)_{\omega \in \{0,1\}^{d_1-1}}\rangle_{U^{d_1-1}_{H_1}(\X)}$ is unchanged if $\tau_k$ is applied to one of the functions $f_\omega \in L^\infty(\X)$ for some $k \in K$.  Using the recursive relations between the Gowers inner products, this implies that $\langle (f_\omega)_{\omega \in \{0,1\}^{d_1}} \rangle_{U^{d_1}_{H_1}(\X)}$ is unchanged if $\tau_k$ is applied to two adjacent functions $f_\omega \in L^\infty(\X)$ connected by an edge, for any $k \in K$.  Taking the commutator of this fact using two intersecting edges on $\{0,1\}^{d_1}$ (recalling that $d_1>1$), we conclude that $\langle (f_\omega)_{\omega \in \{0,1\}^{d_1}} \rangle_{U^{d_1}_{H_1}(\X)}$ is unchanged if a single $f_\omega$ is shifted by $\tau_k$ for some $k$ in the commutator group $[K,K]$.  Equivalently, for $f \in L^\infty(\X)$ and $k \in [K,K]$, $f - \tau_k f$ is orthogonal to all dual functions for $U^{d_1}_{H_1}(\X)$; since $\X \equiv \ZZ^{{<}d_1}_{H_1}(\X)$, this implies that $f-\tau_k f$ is trivial.  Thus the action of $[K,K]$ on $\Y \times_\rho K/L$ is trivial, and so $L$ lies in $[K,K]$ as required.
\end{proof}

With this proposition, we may thus write $\X = \Y \times_\rho K$ for some compact \emph{abelian} group $K = (K,\cdot)$; our task is now to show that $\Y \times_\rho K$ has $G$-order ${<}d_1+d_2-1$.

Define an \emph{$S^1$-cocycle} (or \emph{cocycle}, for short) of a $G$-system $(\Y,T)$ to be a map $\eta\colon G \times \Y \to S^1$ taking values in the unit circle $S^1 \coloneqq  \{ z \in \C: |z|=1\}$ obeying the cocycle equation \eqref{rho}; this is clearly an abelian group.  Observe that for any character $\chi\colon K \to S^1$ in the Pontryagin dual of $K$, $\chi \circ \rho$ is a cocycle.  We say that a cocycle is an \emph{$H$-coboundary} if there is a measurable $F\colon X \to S^1$ such that $\eta(h,x) = F(T^h x) / F(x)$ for all $h \in H$ and almost every $x \in X$; this is a subgroup of the space of all cocycles for each $H$.  Given a cocycle $\eta\colon G \times \Y \to S^1$ on $\Y$ and a subgroup $H$ of $G$, we define the cocycle $d^{[1]}_H \eta\colon G \times \Y^{[1]}_H \to S^1$ on the Host-Kra space $\Y^{[1]}_H$ by the formula
$$ d^{[1]}_{H} \eta(g,x,x') \coloneqq  \eta(g,x) \overline{\eta(g,x')}$$
for all $g \in G$ and $x,x' \in X$; it is easy to see that $d^{[1]}_H$ is a homomorphism from cocycles on $\Y$ to cocycles on $\Y^{[1]}_H$, which maps $H'$-coboundaries to $H'$-coboundaries for any subgroup $H'$ of $G$.
We may iterate this construction to define a homomorphism $d^{[k]}_{H_1,\dots,H_k}$ from cocycles on $\Y$ to cocycles on $\Y^{[k]}_{H_1,\dots,H_k}$ by setting
$$ d^{[k]}_{H_1,\dots,H_k} \coloneqq  d^{[1]}_{H_1} \dots d^{[1]}_{H_k}.$$
Note that $d^{[1]}_H$ and $d^{[1]}_K$ commute for any $H,K$ (after identifying $\Y^{[2]}_{H,K}$ with $\Y^{[2]}_{K,H}$ in the obvious fashion), and so the operator $d^{[k]}_{H_1,\dots,H_k}$ is effectively unchanged with respect to permutation by $H_1,\dots,H_k$.  If $H_1=\dots=H_k=H$, we abbreviate $d^{[k]}_{H_1,\dots,H_k}$ as $d^{[k]}_H$.

We say that a cocycle $\sigma\colon G \times \Y \to S^1$ is of \emph{$H$-type $d$} if $d^{[d]}_H \sigma$ is a $H$-coboundary.  Because the operator $d^{[1]}_{H'}$ maps $H$-coboundaries to $H$-coboundaries for any $H' \leq G$, we see that $d^{[1]}_{H'}$ maps cocycles of $H$-type $d$ to cocycles of $H$-type $d$, and that any cocycle of $H$-type $d$ is also of $H$-type $d'$ for any $d'>d$.

We have a fundamental connection between type and order from \cite{hk}:

\begin{proposition}[Type equation]\label{type}  Let $d \geq 1$ be an integer, let $H$ be a subgroup of an at most countable additive group $G$, and let $\Y$ be an $G$-system of $H$-order ${<}d$ that is $G$-ergodic. Let $K$ be a compact abelian group, and let $\rho\colon H \times Y \to K$ be a cocycle.
\begin{itemize}
\item[(i)]  If the system $\Y \times_\rho K$ has $H$-order ${<}d$, then $\chi \circ \rho$ has $H$-type $d-1$ for all characters $\chi\colon K \to S^1$.
\item[(ii)]  Conversely, if $\chi \circ \rho$ has $H$-type $d-1$ for all characters $\chi\colon K \to S^1$, then $\Y \times_\rho K$ has $H$-order ${<}d$.
\end{itemize}
\end{proposition}

One can use a result of Moore and Schmidt \cite{moore-schmidt} to conclude that $\rho$ itself is of $H$-type $d-1$ in conclusion (i), but we will not need to do so here.  A key technical point to note is that no $H$-ergodicity hypothesis is imposed.

\begin{proof}  For part (i), see\footnote{Again, the argument in \cite{hk} is stated only for $G=\Z$, but extends to other at most countable additive groups $G$ without any modification of the argument.  In any event, the version of the argument in \cite{btz} is explicitly stated for all such groups $G$.}  \cite[Proposition 6.4]{hk} or \cite[Proposition 4.4]{btz}. For part (ii), we argue as follows\footnote{One can also establish (ii) by using \cite[Proposition 7.6]{hk} to handle the case when $\Y \times_\rho K$ is ergodic, and Mackey theory and the ergodic decomposition to extend to the non-ergodic case; we leave the details to the interested reader.}. Suppose for contradiction that $\Y \times_\rho K$ did not have $H$-order ${<}d$, then  by \eqref{fuz} there is a non-zero function $f \in L^\infty(\Y \times_\rho K)$ whose $U^d_H(\Y \times_\rho K)$ seminorm vanished.  The $U^d_H(\Y \times_\rho K)$ norm is invariant with respect to the $K$-action, since this action commutes with the $H$-action.  Using Fourier inversion in the $K$ direction (i.e. decomposition of $L^2(Y \times_\rho K)$ into $K$-isotypic components of the $K$-action) and the triangle inequality, we may assume that $f$ takes the form $f(y,k) = F(y) \chi(k)$ for some $F \in L^\infty(\Y)$ and some character $\chi\colon K \to S^1$, with $F$ not identically zero.  By \eqref{fuz} and the hypothesis that $\Y$ has $H$-order ${<}d$, the property of having  vanishing $U^d_H$ norm is unaffected by multiplication by functions in $L^\infty(\Y)$, as well as shifts by $G$, thus $(y,k) \mapsto (T_g |F(y)|) \chi(k)$ also has vanishing $U^d_H$ norm for $g \in G$.  Averaging and using the ergodic theorem, we conclude that the function $u\colon (y,k) \mapsto \chi(k)$ has vanishing $U^d_H(\Y \times_\rho K)$ seminorm, and hence so does $\tilde F u$ for any $\tilde F \in L^\infty(\Y)$.  By the Gowers-Cauchy-Schwarz inequality, this implies that
$$ \int_{(\Y \times_\rho K)^{[d]}_H} \bigotimes_{\omega \in \{0,1\}^d} {\mathcal C}^{|\omega|}(\tilde F_\omega u) = 0 $$
for all $\tilde F_\omega \in L^\infty(\Y)$, and thus
\begin{equation}\label{yak}
\int_{(\Y \times_\rho K)^{[d]}_H} \tilde F d^{[d]}_H u = 0 
\end{equation}
for all $\tilde F \in L^\infty(\Y^{[d]}_H)$, where the integrals are understood to be with respect to the cubic measure $\mu^{[d]}_H$ of $\Y \times_\rho K$.  On the other hand, by hypothesis we have 
$$d^{[d-1]}_H(\chi \circ \rho)(h,x) = B(T^h y) \overline{B(y)}$$ 
for some measurable $B\colon \Y^{[d-1]}_H \to S^1$ and all $h \in H$ and almost every $y \in \Y^{[d-1]}_H$, which implies that
$$ (\overline{B} \otimes B) d^{[d]}_H u( T^h (y,k), (y,k) ) = 1$$
for all $h \in H$ and almost every $(y,k) \in (\Y \times_\rho K)^{[d-1]}_H$, where by abuse of notation we write $\overline{B} \otimes B$ for the function $((y',k'),(y,k)) \mapsto \overline{B}(y') B(y)$.  By construction of the cubic measure on $(\Y \times_\rho K)^{[d]}_H$ and the ergodic theorem, this implies that
$$ \int_{(\Y \times_\rho K)^{[d]}_H} (\overline{B} \otimes B) d^{[d]}_H u = 1, $$
which contradicts \eqref{yak} since $\overline{B} \otimes B \in L^\infty(\Y^{[d]}_H)$.
\end{proof}

In our current context, Proposition \ref{type}(i) shows that $\chi \circ \rho$ is of $H_i$-type $d_i-1$ for all $i=1,2$ and all characters $\chi\colon K \to S^1$.

We will shortly establish the following proposition:

\begin{proposition}[Concatenation of type]\label{type-concat}  Let $\Y$ be a $G$-system of $G$-order ${<}d_1+d_2-2$, of $H_1$-order ${<}d_1$, and $H_2$-order ${<}d_2$.  Let $\sigma\colon G \times \Y \to S^1$ be a cocycle which is of $H_1$-type $d_1-1$ and $H_2$-type $d_2-1$.  Then $\Y \times_\sigma S^1$ is of $G$-order ${<}d_1+d_2-1$.  In particular, by Proposition \ref{type}(i), $\sigma$ is of $G$-type $d_1+d_2-2$.
\end{proposition}

\begin{remark}  This result may be compared with Proposition \ref{concat}.  A result close to the $d_1=d_2=2$ case of Proposition \ref{type-concat} was previously established in \cite[Proposition 2.1]{austin-alt}.
\end{remark}

\begin{example}  Let $d_1=d_2=2$, $G=\Z^2$, $H_1 = \Z \times \{0\}$, and $H_2 = \{0\} \times \Z$.  Let $\Y$ be the $G$-system consisting of the $2$-torus $(\R/\Z)^2$ with Lebesgue measure and the shift
$$ T^{(n,m)}(x,y) \coloneqq (x+n\alpha, y+m\alpha)$$
for some fixed irrational $\alpha$.  One easily verifies that $\Y$ os of $G$-order, $H_1$-order, and $H_2$-order $<2$.  If we define $\sigma\colon G \times \Y \to S^1$ by the formula
$$ \sigma( (n,m), (x,y) ) := e( ny + mx + nm\alpha )$$
where $e(x) := e^{2\pi i x}$, then $\sigma$ is a cocycle.  One can view $\Y^{[1]}_{H_1}$ as the space of pairs $((x,y),(x',y))$ with $x,y,x' \in \R/\Z$, with the product shift map.  We have
$$ d^{[1]}_{H_1} \sigma( (n,0), ((x,y), (x',y)) ) = 1$$
and so $\sigma$ is certainly of $H_1$-type $1$; it is similarly of $H_2$-type $1$.  The system $\Y \times_\sigma S^1$ is the system $\X$ in Example \ref{stex}, and is thus of $G$-order $<3$.  One can also check that $d^{[2]}_G \sigma$ is identically $1$, basically because the phase $ny + mx + nm \alpha$ is linear in $x,y$.
\end{example}

Assuming Proposition \ref{type-concat} for the moment, we combine it with Proposition \ref{type}(i) to conclude that $\chi \circ \rho$ is of $G$-type $d_1+d_2-2$ for all characters $\chi\colon K \to S^1$.
Applying Proposition \ref{type}(ii), we conclude that $Y \times_\rho K$ is of $G$-order ${<}d_1+d_2-1$, as required.

It remains to establish Proposition \ref{type-concat}.  We first need a technical extension of a result of Host and Kra:

\begin{proposition}  Let $\Y$ be an $G$-system that is $G$-ergodic, and let $\rho\colon G \times Y \to S^1$ be a cocycle which is of $H_1$-type $d_1-2$.  Then $\rho$ differs by a $G$-coboundary from a cocycle which is measurable with respect to $\ZZ^{{<}d_1-1}_{H_1}(\Y)$.
\end{proposition}

\begin{proof}  If $\Y$ was $H_1$-ergodic then this would be immediate from \cite[Corollary 7.9]{hk} (the argument there is stated for $H_1=\Z$, but extends to more general at most countable additive groups).  To extend this result to the $G$-ergodic case, we will give an alternate arrangement\footnote{One could also derive the non-ergodic case from the ergodic case treated in \cite[Corollary 7.9]{hk} by using the ergodic decomposition in $H_1$, Mackey theory, and some measurable selection lemmas, using the $G$-ergodicity to keep the Mackey range groups uniform over the ergodic components; we omit the details.} of the arguments in \cite{hk}, which does not rely on $H_1$-ergodicity.

Let $\X$ denote the $G$-system $\X \coloneqq  \Y \times_\rho S^1$, then $S^1$ acts on $\X$ by translation, with each element $\zeta$ of $S^1$ transforming a function $f\colon (y,z) \mapsto f(y,z)$ in $L^\infty(\X)$ to the translated function $\tau_\zeta f\colon (y,z) \mapsto f(y, \zeta^{-1} z)$.  As $\X$ is an abelian extension of $\Y$, the $S^1$-action commutes with the $G$-action and in particular with the $H_1$-action.  This implies that the factor $\ZZ^{{<}d_1-1}_{H_1}(\X)$ of $\X$ inherits an $S^1$-action (which by abuse of notation we will also call $\tau_\zeta$) which commutes with the $G$-action.

Let us say that a function $f \in L^\infty(\X)$ has \emph{$S^1$-frequency one} if one has $\tau_\zeta f = \zeta^{-1} f$ for all $\zeta \in S^1$, or equivalently if $f$ has the form $f(y,\zeta) = \tilde f(y) \zeta$ for some $\tilde f \in L^\infty(\Y)$.  We claim\footnote{In the language of \cite[Proposition 7.6]{hk}, this is the analogue of the group $W$ being trivial in the non-ergodic setting.} that there is a function $f$ of $S^1$-frequency one with non-vanishing $U^{d_1-1}_{H_1}(\X)$ seminorm.  Suppose for the moment that this were the case, then by \eqref{fuz} there is a function $f' \in L^\infty(\ZZ_{H_1}^{{<}d_1-1}(\X))$ which has a non-zero inner product with a function of $S^1$-frequency one.  Decomposing $f'$ into Fourier components with respect to the $S^1$ action, and recalling that this action preserves $L^\infty(\ZZ_{H_1}^{{<}d_1-1}(\X))$, we conclude that $L^\infty(\ZZ_{H_1}^{{<}d_1-1}(\X))$ contains a function $F$ of $S^1$-frequency one.  The absolute value $|F|$ of this function is $S^1$-invariant and lies in $L^\infty(\ZZ_{H_1}^{{<}d_1-1}(\X))$, hence by Corollary \ref{locx}, it lies in $L^\infty(\ZZ_{H_1}^{{<}d_1-1}(\Y))$.  The support of $|F|$ may not be all of $\Y$, but from $G$-ergodicity we can cover $\Y$ (up to null sets) by the support of countably many translates $|T^g F|$ of $|F|$.  By gluing these translates $T^g F$ together and then normalizing, we may thus find a function $u$ in $L^\infty(\ZZ_{H_1}^{{<}d_1-1}(\X))$ of $S^1$-frequency one which has magnitude one, that is to say it takes values in $S^1$ almost everywhere.  One can then check that the function $\tilde \rho\colon G \times Y \to S^1$ defined (almost everywhere) by $\tilde \rho(g,\cdot) \coloneqq  (T^g u) \overline{u}$ (where we embed $L^\infty(\ZZ_{H_1}^{{<}d_1-1}(\Y))$ in $L^\infty(\Y)$ as usual) is a cocycle that differs from $\rho$ by a $G$-coboundary, giving the claim.

It remains to prove the claim.  We use an argument similar to the one used to prove Proposition \ref{type}(ii).  Suppose for contradiction that all functions of $S^1$-frequency one had vanishing $U^{d_1-1}_{H_1}(\X)$ norm.  By the Cauchy-Schwarz-Gowers inequality \eqref{csg}, we then have
$$
\int_{\X^{[d_1-1]}_{H_1}} \prod_{\omega \in \{0,1\}^{d_1-1}} {\mathcal C}^{|\omega|}(f_\omega(y_\omega) z_\omega)\ d\mu^{[d_1-1]}_{H_1} = 0$$
for all $f_\omega \in L^\infty(\Y)$, where we parameterise $\X^{[d_1-1]}_{H_1}$ by $((y_\omega,z_\omega))_{\omega \in \{0,1\}^{d_1-1}}$ with $y_\omega \in Y$ and $z_\omega \in S^1$.  Approximating functions in $L^\infty(\Y^{[d_1-1]}_{H_1})$ by tensor products, we conclude that
\begin{equation}\label{nab}
\int_{\X^{[d_1-1]}_{H_1}} f( (y_\omega)_{\omega \in \{0,1\}^{d_1-1}}) \prod_{\omega \in \{0,1\}^{d_1-1}} {\mathcal C}^{|\omega|} z_\omega\ d\mu^{[d_1-1]}_{H_1} = 0
\end{equation}
for any $f \in L^\infty(\Y^{[d_1-1]}_{H_1})$.

Now recall that $\rho$ is of $H_1$-type $d_1-2$, so that one has an identity of the form
$$ d^{[d_1-2]}_{H_1} \rho(g, y) = F( T_g y ) \overline{F(y)}$$
for $g \in H_1$ and almost every $y$ in $\Y^{[d_1-2]}_{H_1}$, and some $F \in L^\infty(\Y^{[d_1-2]}_{H_1})$ taking values in $S^1$.  Since $T^g( y, z ) = (T^g y, \rho(g,y) z)$, this implies that
$$ \prod_{\omega \in \{0,1\}^{d_1-2}} {\mathcal C}^{|\omega|} (z'_\omega \overline{z_\omega}) = F(y') \overline{F(y)} $$
for almost every $(y,z)$ in $\X^{[d_1-2]}_{H_1}$ and $g \in H_1$, where $(y',z') \coloneqq  T^g (y,z)$ and $z = (z_\omega)_{\omega \in \{0,1\}^{d_1-2}}$, $z' = (z'_\omega)_{\omega \in \{0,1\}^{d_1-2}}$.  By \eqref{foa}, this implies that
$$ \prod_{\omega \in \{0,1\}^{d_1}} {\mathcal C}^{|\omega|} z_\omega = F(y') \overline{F(y)} $$
for almost every $((y',y),z)$ in $\X^{[d_1-1]}_{H_1}$, with $z = (z_\omega)_{\omega \in \{0,1\}^{d_1-1}}$.  But this contradicts \eqref{nab} with $f\colon (y',y) \mapsto F(y) \overline{F(y')}$.
\end{proof}

We use this result to obtain a variant of Proposition \ref{type-concat}:

\begin{proposition}[Concatenation of type, variant]\label{type-concat-variant}  Let $\Y$ be a $G$-system of $G$-order ${<}d_1+d_2-2$, of $H_1$-order ${<}d_1$, and $H_2$-order ${<}d_2$.  Let $\sigma\colon G \times Y \to S^1$ be a cocycle which is of $H_1$-type $d_1-2$ and $H_2$-type $d_2-1$.  Then $\sigma$ is of $G$-type $d_1+d_2-3$.  Similarly if one assumes instead that $\sigma$ has $H_1$-type $d_1-1$ and $H_2$-type $d_2-2$.
\end{proposition}

\begin{proof} We just prove the first claim, as the second is similar.  Applying the preceding proposition to the restriction of $\sigma$ to $H \times Y$, we see that $\sigma$ differs by a $G$-coboundary from a cocycle $\sigma'\colon G \times Y \to S^1$ which is measurable with respect to $\ZZ^{{<}d_1-1}(\Y)$ when restricted to $H_1 \times Y$.  By Proposition \ref{type}(ii), we now conclude that the system $\ZZ^{{<}d_1-1}_{H_1}(\Y) \times_{\sigma'} S^1$ is of $H_1$-order ${<}d_1-1$.  If we let $\X \coloneqq  \Y \times_{\sigma'} S^1$, and introduce the function $f \in L^\infty(\X)$ by $f(y,u) \coloneqq  u$, then $f$ is measurable with respect to the factor $\ZZ^{{<}d_1-1}_{H_1}(\Y) \times_{\sigma'} S^1$, and thus lies in $L^\infty(\ZZ^{{<}d_1-1}_{H_1}(\X))$.

On the other hand, since $\sigma$ is of $H_2$-type $d_2-1$, $\sigma'$ is also.  Since $\Y$ is of $H_2$-order ${<}d_2$, we may apply Proposition \ref{type}(ii) to conclude that $\X$ is of $H_2$-order ${<}d_2$.  In particular $f$ also lies in $\ZZ^{{<}d_2-1}_{H_2}(\X)$.  Applying the induction hypothesis for Theorem \ref{concat-uap-erg-weak}, we conclude that $f$ lies in $L^\infty(\ZZ^{{<}d_1+d_2-2}_G(\X))$.  Since $\Y$ was already of $G$-order ${<}d_1+d_2-2$, $L^\infty(\Y)$ also lies in $L^\infty(\ZZ^{{<}d_1+d_2-2}_G(\X))$.  By Fourier analysis, any element of $L^\infty(\X)$ can be approximated in $L^2$ to arbitrary accuracy by polynomial combinations of $f$ and elements of $L^\infty(\Y)$, and hence  $L^\infty(\X)$ is contained in $L^\infty(\ZZ^{{<}d_1+d_2-2}_G(\X))$; that is to say, $\X$ has $G$-order ${<}d_1+d_2-2$.  By Proposition \ref{type}(i), this implies that  $\sigma'$ is of $G$-type $d_1+d_2-3$ on $\Y$.  Since $\sigma$ differs from $\sigma'$ by a $G$-coboundary, we conclude that $\sigma$ is of $G$-type $d_1+d_2-3$ also, as required.
\end{proof}

We now prove Proposition \ref{type-concat}, using an argument that is a heavily disguised analogoue of the argument used to prove Proposition \ref{concat}.  Let $\Y, \sigma$ be as in Proposition \ref{type-concat}.  Then by Lemma \ref{glob}, $\Y^{[1]}_{H_1}$ is of $G$-order ${<}d_1+d_2-2$, of $H_1$-order ${<}d_1$, and $H_2$-order ${<}d_2$.  The cocycle $d^{[1]}_{H_1} \sigma$ on $\Y^{[1]}_{H_1}$ is of $H_1$-type $d_1-2$ and $H_2$-type $d_2-1$.  Applying Proposition \ref{type-concat-variant}, we conclude that $d^{[1]}_{H_1} \sigma$ is of $G$-type $d_1+d_2-3$.  Similarly $d^{[1]}_{H_2} \sigma$ is of $G$-type $d_1+d_2-3$.

We now lift $\Y^{[1]}_{H_1}$ and $\Y^{[1]}_{H_2}$ up to the space $\Y^L_{H_1,H_2}$ to conclude that the cocycles $(g,x_{00}, x_{10}, x_{01}) \mapsto  \sigma(g,x_{00}) \overline{\sigma(g,x_{10})}$ and  $(g,x_{00}, x_{10}, x_{01}) \mapsto  \sigma(g,x_{00}) \overline{\sigma(g,x_{01})}$ are of $G$-type $d_1+d_2-3$ on $\Y^L_{H_1,H_2}$.  Dividing one function by the other, we conclude that $\tilde \sigma\colon (g,x_{00}, x_{10}, x_{01}) \mapsto  \sigma(g,x_{10}) \overline{\sigma(g,x_{01})}$ is also of $G$-type $d_1+d_2-3$ on $\Y^L_{H_1,H_2}$.  Since $\Y$ has $G$-order ${<}d_1+d_2-2$, we see from Lemma \ref{glob} that $\Y^L_{H_1,H_2}$ also has $G$-order ${<}d_1+d_2-2$, so by\footnote{The result in \cite{hk} is stated for the case $G=\Z$, but exactly the same argument works for more general countabe abelian $G$.  Similarly for \cite[Proposition 6.4]{hk}.} \cite[Proposition 7.6]{hk} we conclude that $\Y^L_{H_1,H_2} \times_{\tilde \sigma} S^1$ has $G$-order ${<}d_1+d_2-2$.  Now, by Lemma \ref{L-proj}, $\Y^L_{H_1,H_2}$ has $\Y^{[1]}_{H_1+H_2}$ as a factor, and $\tilde \sigma$ is a pullback of $d^{[1]}_{H_1+H_2} \sigma$.  Thus $\Y^{[1]}_{H_1+H_2} \times_{d^{[1]}_{H_1+H_2} \sigma} S^1$ is a factor of $\Y^L_{H_1,H_2} \times_{\tilde \sigma} S^1$ and thus has $G$-order ${<}d_1+d_2-2$.  By \cite[Proposition 6.4]{hk}, we conclude that $d^{[1]}_G \sigma = d^{[1]}_{H_1+H_2} \sigma$ has $G$-type $d_1+d_2-3$, and so $\sigma$ has $G$-type $d_1+d_2-2$, as required.

\section{Sketch of combinatorial concatenation argument}\label{sketch-sec}

In this section we give an informal sketch of how Theorem \ref{concat-uap} is proven, glossing over several technical issues that the nonstandard analysis formalism is used to handle.

We assume inductively that $d_1,d_2 > 1$, and that the theorem has already been proven for smaller values of $d_1+d_2$.
Let us informally call a function $f$ \emph{structured of order $<d_1$ along $H_1$} if it obeys bounds similar to (i), and similarly define the notion of \emph{structured of order $<d_2$ along $H_2$}; these notions can be made rigorous once one sets up the nonstandard analysis formalism.  Roughly speaking, Theorem \ref{concat-uap} then asserts that functions that are structured of order $<d_1$ along $H_1$ and structured of order $<d_2$ along $H_2$ are also structured of order $<d_1+d_2-1$ along $H_1+H_2$.  A key point (established using the machinery of dual functions) is that the class of functions that have a certain structure (e.g. being structured of order $<d_1$ along $H_1$) form a shift-invariant algebra, in that they are closed under addition, scalar multiplication, pointwise multiplication, and translation.  

By further use of the machinery of dual functions, one can show that if $f$ is structured of order $<d_1$ along $H_1$, then the shifts $T^{n} f$ of $f$ with $n \in H_1$ admit a representation roughly of the form
\begin{equation}\label{gah}
T^{n} f \approx \E_h c_{n,h} g_h
\end{equation}
where $\E_h$ represents some averaging operation\footnote{In the rigorous version of this argument, $\E_h$ will be an internal average over some internally finite multiset.  Crucially, the functions $c_{n,h}$ and $g_h$ will depend \emph{internally} on the parameters $n,h$, which when combined with the overspill and underspill properties of nonstandard analysis, will give us some important uniform control on error terms.  Such uniformity can also be obtained in the standard analysis setting, of course, but requires a surprisingly complicated amount of notation in order to manage all the epsilon-type parameters which appear.} with respect to some parameter $h$, the $g_h$ are bounded functions, and the $c_{n,h}$ are functions that are structured of order $<d_1-1$ along $H_1$; this type of ``higher order uniformly almost periodic'' representation of shifts of structured functions generalizes \eqref{pxh}, and first appeared in \cite{tao:ergodic}.  In particular, if $n' \in H_2$, we have
$$ T^{n+n'} f \approx \E_h T^{n'} c_{n,h} T^{n'} g_h.$$
A crucial point now (arising from the shift-invariant algebra property mentioned earlier) is that the structure one has on the original function $f$ is inherited by the functions $c_{n,h}$ and $g_h$.  Specifically, since $f$ is structured of order $<d_1$ along $H_1$ and of order $<d_2$ along $H_2$, the functions $g_h$ appearing in \eqref{gah} should also be structued in this fashion.  This implies that
\begin{equation}\label{tang}
 T^{n'} g_h = \E_{h'} c'_{n',h,h'} g_{h,h'}
\end{equation}
where $c'_{n',h,h'}$ is structured of order $<d_1$ along $H_1$ and of order $<d_2-1$ along $H_2$, and $g_{h,h'}$ is some bounded function.  This leads to a representation of the form
\begin{equation}\label{tnf}
 T^{n+n'} f \approx \E_{h,h'} c''_{n,n',h,h'} g_{h,h'}
\end{equation}
where $c''_{n,n',h,h'} = (T^{n'} c_{n,h}) c'_{n',h,h'}$.  But by the induction hypothesis as before one can show that $c'_{n',h,h'}$ is structured of order $<d_1+d_2-2$ along $H_1+H_2$, and then from the shift-invariant algebra property mentioned earlier, we see that $c''_{n,n',h,h'}$ is also structured of order $<d_1+d_2-2$ along $H_1+H_2$.  The representation \eqref{tnf} can then be used (basically by a Cauchy-Schwarz argument, similar to one used in \cite{tao:ergodic}) to establish that $f$ is structured of order $<d_1+d_2-1$ along $H_1+H_2$.

\begin{remark}  We were not able to directly adapt this argument to give a purely ergodic theory proof of Theorem \ref{concat-uap-erg} or Theorem \ref{concat-box-erg}, mainly due to technical problems defining the notion of ``uniform almost periodicity'' in the ergodic context, and in ensuring that this almost periodicity was uniformly controlled with respect to parameters such as $n,n',h,h'$.  Instead, the natural ergodic analogue of this argument appears to be the variant inclusion
$$ \mathbf{F}^{<{d_1}}_{H_1}(\X) \wedge \mathbf{F}^{<{d_2}}_{H_2}(\X) \leq \mathbf{F}^{<{d_1+d_2-1}}_{H_1+H_2}(\X)$$
under the hypotheses of Theorem \ref{concat-uap-erg-weak}, where the \emph{Furstenberg factors} ${\mathbf F}^{<d}_H(\X)$ \cite{furst-s} are defined recursively
 by setting $\mathbf{F}^{<1}_H(\X) = \X^H$ to be the invariant factor and ${\mathbf F}^{<d+1}_H(\X)$ to be the maximal compact extension of ${\mathbf F}^{<d}_H(\X)$.   This inclusion can be deduced from \cite[Proposition 2.3]{fk} (which was already used in Section \ref{erg2-sec}) and an induction on $d_1+d_2$; we leave the details to the interested reader.  
 We remark that the proof of \cite[Proposition 2.3]{fk} can be viewed as a variant of the arguments sketched in this section.  The Furstenberg factors ${\mathbf F}^{<d}_H(\X)$ are, in general, larger than the Host-Kra factors $\ZZ^{<d}_H(\X)$, because the cocycles in the latter must obey the type condition in Proposition \ref{type}, whereas the former factors have no such constraint.
\end{remark}

\section{Taking ultraproducts}\label{ultra-sec}

To prove our main combinatorial theorems rigorously, it is convenient to use the device of ultraproducts to pass to a nonstandard analysis formulation, in order to hide most of the ``epsilon management'', as well as to exploit infinitary tools such as countable saturation, Loeb measure and conditional expectation.  The use of nonstandard analysis to analyze Gowers uniformity norms was first introduced by Szegedy \cite{szegedy-old}, \cite{szegedy}, and also used by Green and the authors in \cite{gtz-uk}.  

We quickly set up the necessary formalism.  (See for instance \cite{goldblatt} for an introduction to the foundations of nonstandard analysis that is used here.)  We will need to fix a non-principal ultrafilter $\alpha \in \beta \N \backslash \N$ on the natural numbers, thus $\alpha$ is a collection of subsets of natural numbers such that the function $A \mapsto 1_{A \in \alpha}$ forms a finitely additive $\{0,1\}$-valued probability measure on $\N$, which assigns zero measure to every finite set.  The existence of such a non-principal ultrafilter is guaranteed by the axiom of choice.  We refer to elements of $\alpha$ as \emph{$\alpha$-large sets}.

We assume the existence of a \emph{standard universe} ${\mathfrak U}$ - a set that contains all the mathematical objects of interest to us, in particular containing all the objects mentioned in the theorems in the introduction.  Objects in this universe will be referred to as \emph{standard} objects.  A \emph{standard set} is a set consisting entirely of standard objects, and a \emph{standard function} is a function whose domain and range are standard sets.  The standard universe will not need to obey all of the usual ZFC set theory axioms (though one can certainly assume this if desired, given a suitable large cardinal axiom); however we will at least need this universe to be closed under the ordered pair construction $x,y \mapsto (x,y)$, so that the Cartesian product of finitely many standard sets is again a standard set.  

A \emph{nonstandard object} is an equivalence class of tuples $(x_n)_{n \in A}$ of standard objects indexed by an $\alpha$-large set $A$, with two tuples $(x_n)_{n \in A}$, $(y_n)_{n \in B}$ equivalent if they agree on an $\alpha$-large set.  We write $\lim_{n \to \alpha} x_n$ for the equivalence class associated with a tuple $(x_n)_{n \in A}$, and refer to this nonstandard object as the \emph{ultralimit} of the $x_n$.  Thus for instance a \emph{nonstandard natural number} is an ultralimit of standard natural numbers, a \emph{nonstandard real number} is an ultralimit of standard real numbers, and so forth.  If $(X_n)_{n \in A}$ is a sequence of standard sets indexed by an $\alpha$-large set $A$, we define the \emph{ultraproduct} $\prod_{n \to \alpha} X_n$ to be the collection of all ultralimits $\lim_{n \to \alpha} x_n$, where $x_n \in X_n$ for an $\alpha$-large set of $n$.  An \emph{internal set} is a set which is an ultraproduct of standard sets.  We use the term \emph{external set} to denote a set of nonstandard objects that is not necessarily internal.  Note that every standard set $X$ embeds into the nonstandard set ${}^* X \coloneqq  \prod_{n \to \alpha} X$, which we call the \emph{ultrapower} of $X$, by identifying every standard object $x$ with its nonstandard counterpart $\lim_{n \to \alpha} x$.  In particular, the standard universe ${\mathfrak U}$ embeds into the \emph{nonstandard universe} ${}^* {\mathfrak U}$ of all nonstandard objects.

If $X = \prod_{n \to \alpha} X_n$ and $Y = \prod_{n \to \alpha} Y_n$ are internal sets, we have a canonical isomorphism
$$ X \times Y \equiv \prod_{n \to \alpha} (X_n \times Y_n)$$
which (by abuse of notation) allows us to identify the Cartesian product of two internal sets as another internal set.  Similarly for Cartesian products of any (standard) finite number of internal sets.  We will implicitly use such identifications in the sequel without further comment.

An \emph{internally finite set} is an ultraproduct of finite sets (such sets are also known as \emph{hyperfinite} sets in the literature).  Similarly with ``set'' replaced by ``multiset''.  (The multiplicity of an element of an internally finite multiset will of course be a nonstandard natural number in general, rather than a standard natural number.)

Given a sequence $(f_n)_{n \in A}$ of standard functions $f_n\colon X_n \to Y_n$ indexed by an $\alpha$-large set $A$, we define the \emph{ultralimit} $\lim_{n \to \alpha} f_n\colon \prod_{n \to \alpha} X_n \to \prod_{n \to \alpha} Y_n$ to be the function
$$(\lim_{n \to \alpha} f_n)( \lim_{n \to \alpha} x_n) \coloneqq  \lim_{n \to \alpha} f_n(x_n).$$
This is easily seen to be a well-defined function.  Functions that are ultralimits of standard functions will be called \emph{internal functions}.  We use the term \emph{external function} to denote a function between external sets that is not necessarily an internal function.  We will use boldface symbols such as $\mathbf{f}$ to refer to internal functions, distinguishing them in particular from functions $f$ that take values in the standard complex numbers $\C$ rather than the nonstandard complex numbers ${}^* \C$.
  
Using the ultralimit construction, any ultraproduct $X = \prod_{n \to \alpha} X_n$ of structures $X_n$ for some first-order language ${\mathcal L}$, will remain a structure of that language ${\mathcal L}$; furthermore, thanks to the well-known theorem of {\L}os, any first-order sentence will hold in $X$ if and only if it holds in $X_n$ for an $\alpha$-large set of $n$.  For instance, if $G_n$ is an additive group for an $\alpha$-large set of $n$, then $\prod_{n \to \alpha} G_n$ will also be an additive group.

A crucial property of internal sets for us will be the following compactness-like property of internal sets.

\begin{theorem}[Countable saturation]\label{countable}\   
\begin{itemize}
\item[(i)]  Let $(X^{(i)})_{i \in \N}$ be a countable sequence of internal sets.  If the finite intersections $\bigcap_{i=1}^N X^{(i)}$ are non-empty for every (standard) natural number $N$, then $\bigcap_{i=1}^\infty X^{(i)}$ is also non-empty.
\item[(ii)]  Let $X$ be an internal set.  Then any countable cover of $X$ by internal sets has a finite subcover.
\end{itemize}
\end{theorem}

\begin{proof}  It suffices to prove (i), as the claim (ii) follows from taking complements and contrapositives.  Write $X^{(i)} = \prod_{n \to \alpha} X^{(i)}_n$.  Since $\bigcap_{i=1}^N X^{(i)}$ is non-empty, there is an $\alpha$-large set $A_N$ such that $\bigcap_{i=1}^N X_n^{(i)}$ is non-empty for all $n \in A_N$.  By shrinking the $A_N$ as necessary, we may assume that the $A_N$ are non-increasing in $N$.  If we then choose $x_n$ for any $n$ in  $A_N \backslash A_{N+1}$ to lie in $\bigcap_{i=1}^N X_n^{(i)}$ for all $N$, the ultralimit $\lim_{n \to \alpha} x_n$ lies in $\bigcap_{i=1}^\infty X^{(i)}$, giving the claim.
\end{proof}

A nonstandard complex number $z \in {}^* \C$ is said to be \emph{bounded} if one has $|z| \leq C$ for some standard $C$, and \emph{infinitesimal} if $|z| \leq \eps$ for all standard $\eps > 0$.  We write $z=O(1)$ when $z$ is bounded and $z=o(1)$ when $z$ is infinitesimal.  By modifying the proof of the Bolzano-Weierstrass theorem, we see that every bounded $z$ can be uniquely written as $z = \st z + o(1)$ for some standard complex number $\st z$, known as the \emph{standard part} of $z$.  

Countable saturation has the following important consequence:

\begin{corollary}[Overspill/underspill]\label{over}  Let $A$ be an internal subset of ${}^* \C$.
\begin{itemize}
\item[(i)]  If $A$ contains all standard natural numbers, then $A$ also contains an unbounded natural number.
\item[(ii)]  If all elements of $A$ are bounded, then $A$ is contained in $\{ z \in {}^* \C: |z| \leq C\}$ for some standard $C > 0$.
\item[(iii)]  If all elements of $A$ are infinitesimal, then $A$ is contained in $\{ z \in {}^* \C: |z| \leq \eps\}$ for some infinitesimal $\eps>0$.
\item[(iv)]  If $A$ contains all positive standard reals, then $A$ also contains a positive infinitesimal real.
\end{itemize}
\end{corollary}

\begin{proof} If (i) failed, then we would have $\N = A \cap {}^* \N$, and hence $\N$ would be internal, which contradicts Theorem \ref{countable}(ii).  The claim (ii) follows from the contrapositive of (i) applied to the internal set $\{ n \in {}^* \N: |z| \geq n \hbox{ for some } z \in A \}$.  The claim (iii) similarly follows from (i) applied to the internal set $\{ n \in {}^* \N: |z| \leq 1/n \hbox{ for all } z \in A \}$.  Finally, the claim (iv) follows from (i) applied to the internal set $\{ n \in {}^*\N: 1/n \in A \}$. 
\end{proof}

An internal function $\mathbf{f}\colon X \to {}^* \C$ is said to be \emph{bounded} if it is bounded at every point, or equivalently (thanks to overspill or countable saturation) if there is a standard $C$ such that $|f(x)| \leq C$ for all $x \in X$, and we denote this assertion by $\mathbf{f}= O(1)$.  Similarly, an internal function $\mathbf{f}\colon X \to {}^* \C$ is said to be \emph{infinitesimal} if it is infinitesimal at every point, or equivalently (thanks to underspill or countable saturation) if there is an infinitesimal $\eps>0$ such that $|f(x)| \leq \eps$ for all $x \in X$, and we denote this assertion by $\mathbf{f}=o(1)$.

Let $\X_n = (X_n, {\mathcal B}_n, \mu_n)$ be a sequence of standard probability spaces, indexed by an $\alpha$-large set $A$.  In the ultraproduct $X \coloneqq  \prod_{n \to \alpha} X_n$, we have a Boolean algebra ${\mathcal B}$ of \emph{internally measurable sets} - that is to say, internal sets of the form $E = \prod_{n \to \alpha} E_n$, where $E_n \in {\mathcal B}_n$ for an $\alpha$-large set of $n$.  Similarly, we have a complex *-algebra ${\mathcal A}[\X]$ of bounded internally measurable functions - functions $\mathbf{f}\colon X \to {}^* \C$ that are ultralimits of measurable functions $f_n\colon X_n \to \C$, and which are bounded.  We have a finitely additive nonstandard probability measure ${}^* \mu \coloneqq  \lim_{n \to \alpha} \mu_n\colon {\mathcal B} \to {}^* [0,1]$, and a finitely additive nonstandard integral
$$ \int_X \mathbf{f}\ d{}^* \mu \coloneqq  \lim_{n \to \alpha} \int_{X_n} f_n\ d\mu_n $$
defined for bounded internally measurable functions $\mathbf{f} = \lim_{n \to \alpha} f_n \in {\mathcal A}[\X]$.  The standard part $\mu \coloneqq  \st {}^* \mu$ of the nonstandard measure ${}^* \mu$ is then an (external) finitely additive probability measure on ${\mathcal B}$.  From Theorem \ref{countable}(ii), this finitely additive measure is automatically a premeasure, and so by the Carath\'eodory extension theorem it may be extended to a countably additive measure $\mu\colon {\mathcal L} \to [0,1]$ to the $\sigma$-algebra ${\mathcal L}$ generated by ${\mathcal B}$.  The space $\X \coloneqq  (X, {\mathcal L}, \mu)$ is then known as the \emph{Loeb probability space} associated to the standard probability spaces $\X_n$. By construction, every Loeb-measurable set $E \in {\mathcal L}$ can be approximated up to arbitrarily small error in $\mu$ by an internally measurable set.  As a corollary, for any (standard) $1 \leq p < \infty$, then any function in $L^p(\X)$ can be approximated to arbitrary accuracy in $L^p(\X)$ norm by the standard part of a bounded internally measurable function, that is to say $\st {\mathcal A}[\X]$ is dense in $L^p(\X)$.  Indeed, we can say a little more:

\begin{lemma}\label{int}  If $f \in L^\infty(\X)$, then there exists $\mathbf{f} \in {\mathcal A}[\X]$ such that $f = \st \mathbf{f}$ almost everywhere.
\end{lemma}

\begin{proof}  We may normalise $f$ to lie in the closed unit ball of $L^\infty(\X)$.  By density, we may find a sequence $\mathbf{f}_n \in {\mathcal A}[\X]$ for $n \in \N$ bounded in magnitude by $1$, such that $\|f - \st \mathbf{f}_n \|_{L^1(\X)} < 1/n$ for all $n$.  In particular, $\| \mathbf{f}_N - \mathbf{f}_n \|_{{}^* L^1(\X)} \leq 2/n$ for all $n \leq N$.  By countable saturation, there thus exists an internally measurable function $\mathbf{f}$ bounded in magnitude by $1$ such that
$\| \mathbf{f} - \mathbf{f}_n \|_{{}^* L^1(\X)} \leq 2/n$ for all $n$.  Taking limits we see that $\|f-\st \mathbf{f} \|_{L^1(\X)}=0$, and the claim follows.
\end{proof}

By working first with simple functions and then taking limits, we easily establish the identity
$$ \int_X \st \mathbf{f}\ d\mu = \st \int_X \mathbf{f}\ d{}^* \mu $$
for any $\mathbf{f} \in {\mathcal A}[\X]$.  In particular we have
$$ \| \st \mathbf{f} \|_{L^p(\X)} = \st \|\mathbf{f}\|_{{}^* L^p(\X)}$$
for any $\mathbf{f} \in {\mathcal A}[\X]$ and $1 \leq p \leq \infty$ and
$$ \langle \st \mathbf{f}, \st \mathbf{g} \rangle_{L^2(\X)} = \st \langle \mathbf{f}, \mathbf{g} \rangle_{{}^* L^2(\X)}$$
for any $\mathbf{f}, \mathbf{g} \in {\mathcal A}[\X]$, where the internal inner product $\langle \mathbf{f}, \mathbf{g} \rangle_{{}^* L^2(\X)}$ is defined as
$$ \langle \mathbf{f}, \mathbf{g} \rangle_{{}^* L^2(\X)} \coloneqq  \int_X \mathbf{f} \overline{\mathbf{g}}\ d{}^* \mu$$
and the internal $L^p$ norm is defined as
$$ \|\mathbf{f}\|_{{}^* L^p(\X)} \coloneqq  \left(\int_X |\mathbf{f}|^p\ d{}^*\mu\right)^{1/p}$$
for finite $p$ and
$$ \|\mathbf{f}\|_{{}^* L^\infty(\X)} \coloneqq  {}^* \mu\!-\!\operatorname{ess} \sup_{x \in X} |\mathbf{f}(x)|$$
for $p=\infty$.

If $H = \prod_{n \to \alpha} H_n$ is an internally finite non-empty multiset, and $h \mapsto \mathbf{f}_h$ is an internal map from $H$ to ${\mathcal A}[\X]$,
then the internal function $h \mapsto \| \mathbf{f}_h \|_{{}^* L^\infty(\X)}$ is bounded, and hence (by Corollary \ref{over}) its image is bounded above by some standard $C$, so that $\{ \st \mathbf{f}_h: h \in H \}$ is bounded in $L^\infty(\X)$.  We can also define the internal average $\E_{h \in H} \mathbf{f}_h$ as
$$ \E_{h \in H} \mathbf{f}_h \coloneqq  \lim_{n \to \infty} \E_{h_n \in H_n} f_{h_n,n}$$
where $h \mapsto \mathbf{f}_h$ is the ultralimit of the maps $h_n \mapsto f_{h_n,n}$.  We have a basic fact about the location of this average:

\begin{lemma}\label{conv}  Let $H$ be an internally finite non-empty multiset, and let $h \mapsto \mathbf{f}_h$ be an internal map from $H$ to ${\mathcal A}[\X]$.  Then the average $\st \E_{h \in H} \mathbf{f}_h$ lies in the closed convex hull (in $L^2(\X)$) of $\{ \st \mathbf{f}_h: h \in H \}$.
\end{lemma}

\begin{proof}  If this were not the case, then by the Hahn-Banach theorem there would exist $\eps>0$ and $g \in L^2(\X)$ such that
$$ \operatorname{Re} \langle \st \E_{h \in H} \mathbf{f}_h, g \rangle_{L^2(\X)} > \operatorname{Re} \langle \st \mathbf{f}_h, g \rangle_{L^2(\X)} + \eps$$
for all $h \in H$.  By truncation (and the boundedness of $\{ \st \mathbf{f}_h: h \in H \}$) we may assume (after shrinking $\eps$ slightly) that $g \in L^\infty(\X)$, and then by Lemma \ref{int} we may write $g = \st \mathbf{g}$ for some $\mathbf{g} \in {\mathcal A}[\X]$.  But then we have
$$ \operatorname{Re} \langle \E_{h \in H} \mathbf{f}_h, \mathbf{g} \rangle_{{}^* L^2(\X)} > \operatorname{Re} \langle \mathbf{f}_h, \mathbf{g} \rangle_{{}^* L^2(\X)} + \eps/2,$$
which on taking internal averages implies
$$ \operatorname{Re} \langle \E_{h \in H} \mathbf{f}_h, \mathbf{g} \rangle_{{}^* L^2(\X)} > \operatorname{Re} \langle \E_{h \in H} \mathbf{f}_h, \mathbf{g} \rangle_{{}^* L^2(\X)} + \eps/2,$$
which is absurd.
\end{proof}

\subsection{Translation to nonstandard setting}\label{nonst-sec}

We now translate Theorems \ref{concat-uap}, \ref{concat-box} to a nonstandard setting.

Let $G = \prod_{n \to \alpha} G_n$ be a nonstandard additive group, that is to say an ultraproduct of standard additive groups $G_n$.
Define a \emph{internal coset progression} $Q$ in $G$ to be an ultraproduct $Q = \prod_{n \to \alpha} Q_n$ of standard coset progressions $Q_n$ in $G_n$.  We will be interested in internal coset progressions of bounded rank, which is equivalent to $Q_n$ having bounded rank on an $\alpha$-large set of $n$.  Given a internal coset progression $Q$, we define the (external) set $o(Q)$ as $o(Q) \coloneqq  \bigcap_{\eps \in \R^+} (\eps Q)$, that is to say the set of all $x \in G$ such that $x \in \eps Q$ for all standard $\eps > 0$.  Here we are interpreting $\eps Q$ and $o(Q)$ as sets rather than multisets.

Given a sequence $(\X_n,T_n)$ of standard $G_n$-sytems, we can form a $G$-system $(\X,T)$ by setting $\X$ to be the Loeb probability space associated to the $\X_n$, and $T^g$ to be the ultralimit of the $T_n^{g_n}$ for any $g = \lim_{n\to \alpha} g_n$.  It is easy to verify that this is a $G$-system; we will refer to such systems as \emph{Loeb $G$-systems}.

Let $d$ be a standard positive integer, and let $Q = \prod_{n \to \alpha} Q_n$ be an internal coset progression.  If $\mathbf{f} = \lim_{n \to \alpha} f_n \in {\mathcal A}[\X]$ is a bounded internally measurable function, we can define the internal Gowers norm
$$ \|\mathbf{f}\|_{{}^* U^d_Q(\X)} \coloneqq  \lim_{n \to \alpha} \|f_n\|_{U^d_{Q_n}(\X_n)}.$$
From the H\"older and triangle inequalities one has
$$ \|f_n\|_{U^d_{Q_n}(\X_n)} \leq \|f_n\|_{L^{2^d}(\X_n)}$$
for each $n$, and hence on taking ultralimits
$$ \|\mathbf{f}\|_{{}^* U^d_{Q}(\X)} \leq \|\mathbf{f}\|_{{}^* L^{2^d}(\X)}$$
and then on taking standard parts
$$ \st \|\mathbf{f}\|_{{}^* U^d_{Q}(\X)} \leq \|\st \mathbf{f}\|_{L^{2^d}(\X)}.$$
We can thus uniquely define the \emph{external Gowers seminorm} $U^d_Q(\X)$ on $L^{2^d}(\X)$ such that
\begin{equation}\label{staf}
 \|\st \mathbf{f} \|_{U^d_Q(\X)} = \st \|\mathbf{f}\|_{{}^* U^d_Q(\X)}
\end{equation}
for all $\mathbf{f} \in {\mathcal A}[\X]$, and
$$ \| f \|_{U^d_Q(\X)} \leq \|f\|_{L^{2^d}(\X)}$$
for all $f \in L^{2^d}(\X)$.

In a similar fashion, if $Q^{(i)} = \prod_{n \to \alpha} Q^{(i)}_n$ are internal coset progressions, we can define the internal box norms
$$ \|\mathbf{f}\|_{{}^* \Box^d_{Q^{(1)},\dots,Q^{(d)}}(\X)} \coloneqq  \lim_{n \to \alpha} \|f_n\|_{\Box^d_{Q^{(1)}_n,\dots,Q^{(d)}_n}(\X_n)}$$
for $\mathbf{f} = \lim_{n \to \alpha} f_n \in {\mathcal A}[\X]$, and the external box seminorm $\Box^d_{Q^{(1)},\dots,Q^{(d)}}(\X)$ on $L^{2^d}(\X)$ such that
$$ \|\st \mathbf{f} \|_{\Box^d_{Q^{(1)},\dots,Q^{(d)}}(\X)} = \st \|\mathbf{f}\|_{{}^* \Box^d_{Q^{(1)},\dots,Q^{(d)}}(\X)}$$
for all $\mathbf{f} \in {\mathcal A}[\X]$, and
$$ \| f \|_{\Box^d_{Q^{(1)},\dots,Q^{(d)}}(\X)} \leq \|f\|_{L^{2^d}(\X)}$$
for all $f \in L^{2^d}(\X)$.

\begin{lemma}\label{lard}  Let $Q$ be a internal coset progression of bounded rank in a nonstandard additive group $G$, and let $(\X,T)$ be a Loeb $G$-system.  Then for any standard $d \geq 1$ and any $f \in L^{2^d}(\X)$, 
the limit
$$ \| f\|_{U^d_{o(Q)}(\X)} \coloneqq  \lim_{\eps \to 0} \|f\|_{U^d_{\eps Q}(\X)}$$
exists.  Furthermore one has
\begin{align*}
\|f\|_{U^d_{o(Q)}(\X)} &= \sup_{\eps \in \R: \eps > 0} \| f\|_{U^d_{\eps Q}(\X)} \\
&=\inf_{\eps = o(1): \eps > 0} \| f\|_{U^d_{\eps Q}(\X)}. 
\end{align*}

More generally, given internal coset progressions $Q^{(1)},\dots,Q^{(d)}$ of bounded rank in $G$, the limit
$$ \| f\|_{\Box^d_{o(Q^{(1)}),\dots,o(Q^{(d)}}(\X)} \coloneqq  \lim_{\eps \to 0} \|f\|_{\Box^d_{\eps Q^{(1)},\dots,\eps Q^{(d)}}(\X)}$$
exists, with
\begin{align*}
\| f\|_{\Box^d_{o(Q^{(1)}),\dots,o(Q^{(d)})}(\X)} &= \sup_{\eps \in \R: \eps > 0} \|f\|_{\Box^d_{\eps Q^{(1)},\dots, \eps Q^{(d)}}(\X)}  \\
&= \inf_{\eps = o(1): \eps > 0} \|f\|_{\Box^d_{\eps Q^{(1)},\dots, \eps Q^{(d)}}(\X)}.
\end{align*}
\end{lemma}

\begin{proof} It suffices to prove the second claim.  By limiting arguments we may assume that $f$ lies in $L^\infty(\X)$; by Lemma \ref{int} we may then assume $f = \st \mathbf{f}$ for some $\mathbf{f} \in {\mathcal A}[\X]$, which we may normalize to be bounded in magnitude by $1$.

Suppose that we can show that
\begin{equation}\label{suppose}
\|\mathbf{f}\|_{{}^* \Box^d_{\eps Q^{(1)},\dots, \eps Q^{(d)}}(\X)}^{2^d} \leq \|\mathbf{f}\|_{{}^* \Box^d_{\eps' Q^{(1)},\dots, \eps' Q^{(d)}}(\X)}^{2^d} + o(1)
\end{equation}
whenever $\eps>0$ is standard and $\eps'>0$ is infinitesimal.  This clearly implies 
$$
\sup_{\eps \in \R: \eps>0} \|\mathbf{f}\|_{{}^* \Box^d_{\eps Q^{(1)},\dots, \eps Q^{(d)}}(\X)} \leq \inf_{\eps' = o(1): \eps'>0} \|\mathbf{f}\|_{{}^* \Box^d_{\eps' Q^{(1)},\dots, \eps' Q^{(d)}}(\X)} + \delta$$
for any standard $\delta>0$; on the other hand from underspill we cannot have
$$
\sup_{\eps \in \R: \eps>0} \|\mathbf{f}\|_{{}^* \Box^d_{\eps Q^{(1)},\dots, \eps Q^{(d)}}(\X)} > c > \inf_{\eps' = o(1): \eps'>0} \|\mathbf{f}\|_{{}^* \Box^d_{\eps' Q^{(1)},\dots, \eps' Q^{(d)}}(\X)}$$
for any standard $c > 0$.  Taking standard parts, we conclude that
$$
\sup_{\eps \in \R: \eps>0} \|f\|_{\Box^d_{\eps Q^{(1)},\dots, \eps Q^{(d)}}(\X)} = \inf_{\eps' = o(1): \eps'>0} \|f\|_{\Box^d_{\eps' Q^{(1)},\dots, \eps' Q^{(d)}}(\X)}.$$
Denoting this quantity by $\| f\|_{\Box^d_{o(Q^{(1)}),\dots,o(Q^{(d)})}(\X)}$, we see that for every standard $\delta > 0$, one has
$$ \|\mathbf{f}\|_{{}^* \Box^d_{\eps Q^{(1)},\dots, \eps Q^{(d)}}(\X)} \leq \| f\|_{\Box^d_{o(Q^{(1)}),\dots,o(Q^{(d)})}(\X)}+\delta$$
for all infinitesimal $\eps > 0$, and hence (by underspill) for all sufficiently small standard $\eps>0$ as well.  The claim then follows.

It remains to prove \eqref{suppose}. By \eqref{boxnorm-def}, the left-hand side expands as
$$ \E_{h_i \in \eps Q^{(i)}-\eps Q^{(i)} \forall i =1,\dots,d} \int_X \Delta_{0,h_1} \dots \Delta_{0,h_d} \mathbf{f}\ d{}^*\mu,$$
where the nonlinear operators $\Delta_{h,h'}$ were defined in \eqref{deltah-def}.
We can shift each $\eps Q^{(i)}-\eps Q^{(i)}$ by an element of $\eps' Q^{(i)} - \eps Q^{(i)}$ while only affecting this expression by $o(1)$.  Averaging, we can thus write this expression as
$$ \E_{h'_i \in \eps' Q^{(i)}-\eps' Q^{(i)} \forall i=1,\dots, d} \E_{h_i \in \eps Q^{(i)}-\eps Q^{(i)} + h'_i \forall i =1,\dots,d} \int_X \Delta_{0,h_1} \dots \Delta_{1,h_d} \mathbf{f}\ d{}^*\mu + o(1).$$
Writing $h_i = h'_i + k_i$, this becomes
$$ \E_{k_i \in \eps Q^{(i)}-\eps Q^{(i)} \forall i =1,\dots,d} 
\E_{h'_i \in \eps' Q^{(i)}-\eps' Q^{(i)} \forall i=1,\dots, d} \int_X \Delta_{0,h'_1+k_1} \dots \Delta_{1,h'_d+k_d} \mathbf{f}\ d{}^*\mu + o(1).$$
This can be rewritten in turn as
$$ \E_{k_i \in \eps Q^{(i)}-\eps Q^{(i)} \forall i =1,\dots,d} \langle ( T^{\omega_1 k_1 + \dots + \omega_d k_d} \mathbf{f} )_{\omega \in \{0,1\}^d} \rangle_{{}^* \Box^d_{\eps' Q^{(1)},\dots,\eps' Q^{(d)}}} + o(1).$$
The claim \eqref{suppose} now follows from the Cauchy-Schwarz-Gowers inequality \eqref{csg}.
\end{proof}

From this lemma we obtain seminorms $\| \|_{U^d_{o(Q)}(\X)}$ and $\| \|_{\Box^d_{o(Q^{(1)}),\dots,o(Q^{(d)})}(\X)}$ on $L^{2^d}(\X)$, and hence on $L^\infty(\X)$.

We can now state the nonstandard theorem that Theorem \ref{concat-uap} will be derived from.

\begin{theorem}[Concatenation theorem for anti-uniformity norms, nonstandard version]\label{concat-uap-nonst}  Let $Q_1, Q_2$ be internal coset progressions of bounded rank in a nonstandard additive group $G$, let $(\X,T)$ be a Loeb $G$-system, let $d_1,d_2$ be standard positive integers, and let $f$ lie in the closed unit ball of $L^\infty(\X)$.  We make the following hypotheses:
\begin{itemize}
\item[(i)] $f$ is orthogonal (in $L^2(\X)$) to $\{ g \in L^\infty(\X): \|g\|_{U^{d_1}_{o(Q_1)}(\X)} = 0 \}$.
\item[(ii)] $f$ is orthogonal to $\{ g \in L^\infty(\X): \|g\|_{U^{d_2}_{o(Q_2)}(\X)} = 0 \}$.
\end{itemize}
Then $f$ is orthogonal to $\{ g \in L^\infty(\X): \|g\|_{U^{d_1+d_2-1}_{o(Q_1+Q_2)}(\X)} = 0 \}$.
\end{theorem}

Let us assume this theorem for the moment, and show how it implies Theorem \ref{concat-uap}.  It suffices to show that for any $d_1,d_2,r_1,r_2,c_1,c_2$
as in Theorem \ref{concat-uap}, and any $\delta > 0$, there exists an $\eps > 0$, such that
$$ \| f \|_{U^{d_1+d_2-1}_{Q_1+Q_2}(\X)^*, \eps} \leq \delta$$
whenever $Q_1, Q_2, G, (\X,T), f$ are standard objects obeying the hypotheses of Theorem \ref{concat-uap} with the given values of $d_1,d_2,r_1,r_2,c_1,c_2$.

Suppose this were not the case.  Then there exists $d_1,d_2,r_1,r_2,c_1,c_2,\delta$ as above, together with a sequence $G_n$ of standard additive groups, sequences $Q_{1,n}, Q_{2,n}$ of standard coset progressions in $G_n$ of ranks $r_1,r_2$ respectively, a sequence $(\X_n,T_n)$ of $G_n$-systems, and a sequence $f_n$ of functions in the closed unit ball of $L^\infty(\X)$ such that
\begin{equation}\label{fand}
\| f_n \|_{U^{d_1}_{Q_{1,n}}(\X_n)^*, \eps} \leq c_1(\eps)
\end{equation}
and $\| f_n \|_{U^{d_2}_{Q_{2,n}}(\X_n)^*, \eps} \leq c_2(\eps)$ 
for all (standard) $\eps > 0$, but such that
\begin{equation}\label{fln}
 \| f_n \|_{U^{d_1+d_2-1}_{\eps_n Q_{1,n}+ \eps_n Q_{2,n}}(\X_n)^*, \eps_n} > \delta
\end{equation}
for some $\eps_n>0$ that goes to zero as $n \to \infty$.
Now we take ultraproducts, obtaining a nonstandard additive group $G$, internal coset progressions $Q_1,Q_2$ of bounded rank in $G$, a Loeb $G$-system $(\X,T)$, and a bounded internally measurable function $\mathbf{f} \coloneqq  \lim_{n \to \alpha} f_n$.  Since the $f_n$ lie in the closed unit ball of $L^\infty(\X_n)$, we see that $\st \mathbf{f}$ lies in the closed unit ball of $L^\infty(\X)$.

Now suppose that $g \in L^\infty(\X)$ is such that $\|g\|_{U^{d_1}_{o(Q_1)}(\X)} = 0$; we claim that $\st \mathbf{f}$ is orthogonal to $g$.  We may normalise $g$ to lie in the closed unit ball of $L^\infty(\X)$.  By Lemma \ref{lard} one has
$$\|g\|_{U^{d_1}_{\eps Q_1}(\X)} = 0$$
for all standard $\eps > 0$.  We can write $g$ as the limit in $L^{2^{d_1}}(\X)$ of $\st \mathbf{g}^{(i)}$ as $i \to \infty$, for some $\mathbf{g}^{(i)} \in {\mathcal A}[\X]$ that are bounded in magnitude by $1$.  Since the $L^{2^{d_1}}(\X)$ norm controls the $U^{d_1}_{\eps Q_1}(\X)$ semi-norm, we have for any standard $\delta > 0$ that
$$ \| \st \mathbf{g}^{(i)} \|_{U^{d_1}_{\eps Q_1}(\X)} \leq \delta $$
for sufficiently large $i$ and all $\eps > 0$.  In particular, writing $\mathbf{g}^{(i)} = \lim_{n \to \alpha} g^{(i)}_n$ and setting $\eps=2\delta$, we have for sufficiently large $i$ that
$$ \| g^{(i),n} \|_{U^{d_1}_{2 \delta Q_{1,n}}(\X_n)} \leq 2 \delta $$
for an $\alpha$-large set of $n$.   By \eqref{fand} we conclude that
$$ |\langle f_n, g^{(i)}_n \rangle_{L^2(\X_n)}| \leq c_1(2\delta) $$
and thus on taking ultralimits and then standard parts
$$ |\langle \st \mathbf{f}, \st \mathbf{g}^{(i)} \rangle_{L^2(\X)}| \leq c_1(2\delta) $$
and hence on sending $i$ to infinity
$$ |\langle \st \mathbf{f}, g \rangle_{L^2(\X)}| \leq c_1(2\delta).$$
Sending $\delta$ to zero, we obtain the claim.

For similar reasons, $\st \mathbf{f}$ is orthogonal to any $g \in L^\infty(\X)$ with $\|g\|_{U^{d_2}_{o(Q_2)}(\X)} = 0$.  Applying Theorem \ref{concat-uap-nonst}, we conclude that $\st \mathbf{f}$ is orthogonal to any $g \in L^\infty(\X)$ with $\|g\|_{U^{d_1+d_2-1}_{o(Q_1+Q_2)}(\X)} = 0$.  On the other hand, from \eqref{fln} and \eqref{upqd}, we can find a $g_n$ in the closed unit ball of $L^\infty(\X_n)$ for each $n$ such that
$$ \|g_n\|_{U^{d_1+d_2-1}_{\eps_n Q_{1,n}+\eps_n Q_{2,n}}(\X_n)} \leq \eps_n $$
and
$$ |\langle f_n, g_n \rangle_{L^2(\X_n)}| \geq \delta.$$
Setting $g \coloneqq  \st \lim_{n \to \alpha} g_n$, we conclude on taking ultralimits that $g$ is in the closed unit ball of $L^\infty(\X)$ with
$$ \|g\|_{U^{d_1+d_2-1}_{\eps Q_{1}+\eps Q_{2}}(\X)} = 0$$
for some infinitesimal $\eps>0$ and
$$ |\langle \st \mathbf{f}, g \rangle_{L^2(\X)}| \geq \delta.$$
But from Lemma \ref{lard} we have
$$ \|g\|_{U^{d_1+d_2-1}_{o(Q_{1}+Q_{2})}(\X)} = 0$$
and we contradict the previously established orthogonality properties of $\st \mathbf{f}$. This concludes the derivation of Theorem \ref{concat-uap} from Theorem \ref{concat-uap-nonst}.

An identical argument shows that Theorem \ref{concat-box} is a consequence of the following nonstandard version.

\begin{theorem}[Concatenation theorem for anti-box norms, nonstandard version]\label{concat-box-nonst}  Let $d_1,d_2$ be standard positive integers, let $Q_{1,1},\dots,Q_{1,d_1}, Q_{2,1},\dots,Q_{2,d_2}$ be internal coset progressions of bounded rank in a nonstandard additive group $G$, let $(\X,T)$ be a Loeb $G$-system, and let $f$ lie in the closed unit ball of $L^\infty(\X)$.  We make the following hypotheses:
\begin{itemize}
\item[(i)] $f$ is orthogonal (in $L^2(\X)$) to $\{ g \in L^\infty(\X): \|g\|_{\Box^{d_1}_{o(Q_{1,1}),\dots,o(Q_{1,d_1})}(\X)} = 0 \}$.
\item[(ii)] $f$ is orthogonal to $\{ g \in L^\infty(\X): \|g\|_{\Box^{d_2}_{o(Q_{2,1}),\dots,o(Q_{2,d_2})}(\X)} = 0 \}$.
\end{itemize}
Then $f$ is orthogonal to $\{ g \in L^\infty(\X): \|g\|_{\Box^{d_1 d_2}_{(o(Q_{1,i}+Q_{2,j}))_{1 \leq i \leq d_1; 1 \leq j \leq d_2}}(\X)} = 0 \}$.
\end{theorem}

It remains to establish Theorem \ref{concat-uap-nonst} and Theorem \ref{concat-box-nonst}.

\section{Nonstandard dual functions}\label{dual-sec}

We now develop some nonstandard analogues of the machinery in Section \ref{erg-sec}, and use this machinery to prove Theorem \ref{concat-uap-nonst} and Theorem \ref{concat-box-nonst} (and hence Theorem \ref{concat-uap} and Theorem \ref{concat-box} respectively).

Let $(\X,T)$ be a Loeb $G$-system for some nonstandard abelian group $G$, and let $Q_1,\dots,Q_d$ be internal coset progressions of bounded rank in $G$ for some standard $d \geq 0$.  Given bounded internally measurable functions $\mathbf{f}_\omega \in {\mathcal A}[\X]$ for $\omega \in \{0,1\}^d \backslash \{0\}^d$, we define the \emph{internal dual function} ${}^* {\mathcal D}^d_{Q_1,\dots,Q_d}( \mathbf{f}_\omega )_{\omega \in \{0,1\}^d \backslash \{0\}^d}$ to be the bounded internally measurable function
$$ 
{}^* {\mathcal D}^d_{Q_1,\dots,Q_d}( \mathbf{f}_\omega )_{\omega \in \{0,1\}^d \backslash \{0\}^d}
= \E_{h_1 \in Q_1-Q_1,\dots,h_d \in Q_d-Q_d} \prod_{\omega \in \{0,1\}^d \backslash \{0\}^d} {\mathcal C}^{|\omega|-1} T^{\omega_1 h_1 + \dots + \omega_d h_d} \mathbf{f}_\omega $$
where $\omega = (\omega_1,\dots,\omega_d)$, $|\omega| \coloneqq  \omega_1+\dots+\omega_d$, and ${\mathcal C}\colon f \mapsto \overline{f}$ is the complex conjugation operator.  (When $d=0$, we adopt the convention that ${}^* {\mathcal D}^0() = 1$.) This is a multilinear map from ${\mathcal A}[\X]^{2^d-1}$ to ${\mathcal A}[\X]$, and from the definition of the internal box norm ${}^* \Box^d_{Q_1,\dots,Q_d}(\X)$ we have the identity
\begin{equation}\label{dual-ident}
 \|\mathbf{f}\|_{{}^* \Box^d_{Q_1,\dots,Q_d}(\X)}^{2^d} = 
\langle \mathbf{f}, {}^* {\mathcal D}^d_{Q_1,\dots,Q_d}( \mathbf{f} )_{\omega \in \{0,1\}^d \backslash \{0\}^d} \rangle_{{}^* L^2(\X)}
\end{equation}
for any $\mathbf{f} \in {\mathcal A}[\X]$.

From H\"older's inequality and the triangle inequality we see that
$$
\| {}^* {\mathcal D}^d_{Q_1,\dots,Q_d}( \mathbf{f}_\omega )_{\omega \in \{0,1\}^d \backslash \{0\}^d} \|_{{}^* L^{2^d/(2^d-1)}(\X)}
\leq \prod_{\omega \in \{0,1\}^d \backslash \{0\}^d} \|\mathbf{f}_\omega \|_{{}^* L^{2^d}(\X)}.$$
By a limiting argument and multilinearity, we may thus uniquely define a bounded multilinear \emph{dual operator} ${\mathcal D}^d_{Q_1,\dots,Q_d}\colon L^{2^d}(\X)^{2^d-1} \to L^{2^d/(2^d-1)}(\X)$ such that
$$
\| {\mathcal D}^d_{Q_1,\dots,Q_d}( f_\omega )_{\omega \in \{0,1\}^d \backslash \{0\}^d} \|_{L^{2^d/(2^d-1)}(\X)}
\leq \prod_{\omega \in \{0,1\}^d \backslash \{0\}^d} \|f_\omega \|_{L^{2^d}(\X)}$$
for all $f_\omega \in L^{2^d}(\X)$, and such that
\begin{equation}\label{dstar}
 {\mathcal D}^d_{Q_1,\dots,Q_d}( \st \mathbf{f}_\omega )_{\omega \in \{0,1\}^d \backslash \{0\}^d}
= \st  {}^* {\mathcal D}^d_{Q_1,\dots,Q_d}( \mathbf{f}_\omega )_{\omega \in \{0,1\}^d \backslash \{0\}^d} 
\end{equation}
whenever $\mathbf{f}_\omega \in {\mathcal A}[\X]$ for $\omega \in \{0,1\}^d \backslash \{0\}^d$.

An important fact about dual functions, analogous to Theorem \ref{charfac}, is that the dual operator maps factors to characteristic subfactors for the associated Gowers norm:

\begin{theorem}[Dual functions and characteristic factors]\label{charf}  Let $(\Y,T)$ be a factor of $(\X,T)$, and let $Q_1,\dots,Q_d$ be internal coset progressions of bounded rank.  Let $V$ be the linear span of the space
$$ \{ {\mathcal D}^d_{\eps Q_1,\dots,\eps Q_d}( f_\omega )_{\omega \in \{0,1\}^d \backslash \{0\}^d}: f_\omega \in L^\infty(\Y); \eps \in \R; \eps > 0 \},$$
thus $V$ is the space of finite linear combinations of dual functions of functions in $L^\infty(\Y)$ (for various dilates $\eps Q_1,\dots,\eps Q_d$ of $Q_1,\dots,Q_d$).  Then there exists a unique subfactor $(\ZZ^{{<}d}_{o(Q_1),\dots,o(Q_d)}(\Y), T)$ of $(\Y,T)$ with the following properties:
\begin{itemize}
\item[(i)]  $V$ is contained in $L^\infty(\ZZ^{{<}d}_{o(Q_1),\dots,o(Q_d)}(\Y))$.
\item[(ii)] Conversely, every element of $L^\infty(\ZZ^{{<}d}_{o(Q_1),\dots,o(Q_d)}(\Y))$ is the limit (in $L^{2^d/(2^d-1)}$) of a sequence in $V$ that is uniformly bounded in $L^\infty$ norm.
\item[(iii)]  If $f \in L^\infty(\Y)$, then $\|f\|_{\Box^d_{o(Q_1),\dots,o(Q_d)}(\X)}=0$ if and only if $f$ is orthogonal (in the $L^2$ sense) to $L^\infty(\ZZ^{{<}d}_{o(Q_1),\dots,o(Q_d)}(\Y))$.
\end{itemize}
\end{theorem}

We now prove this theorem.  From (i) and (ii) it is clear that $\ZZ^{{<}d}_{o(Q_1),\dots,o(Q_d)}(\Y)$ is unique; the difficulty is to establish existence.
We first observe

\begin{lemma} We have $V \subset L^\infty(\Y)$.
\end{lemma}

\begin{proof}  It suffices by linearity to show that
$$
{\mathcal D}^d_{\eps Q_1,\dots,\eps Q_d}( f_\omega )_{\omega \in \{0,1\}^d \backslash \{0\}^d} \in L^\infty(\Y)$$
whenever $f_\omega \in L^\infty(\Y)$.  From Lemma \ref{int} we may write $f_\omega = \st F_\omega$ for some $F_\omega \in {\mathcal A}[\X]$.  From \eqref{dstar} and Lemma \ref{conv}, the dual function ${\mathcal D}^d_{\eps Q_1,\dots,\eps Q_d}( f_\omega )_{\omega \in \{0,1\}^d \backslash \{0\}^d}$ then lies in the closed convex hull of the functions
$$
\st \prod_{\omega \in \{0,1\}^d \backslash \{0\}^d} {\mathcal C}^{|\omega|-1} T^{\omega_1 h_1 + \dots + \omega_d h_d} F_\omega 
=
\prod_{\omega \in \{0,1\}^d \backslash \{0\}^d} {\mathcal C}^{|\omega|-1} T^{\omega_1 h_1 + \dots + \omega_d h_d} f_\omega $$
for various $h_1,\dots,h_d \in G$.  But such functions lie in a bounded subset of $L^\infty(\Y)$, and the claim follows.
\end{proof}

Next, we observe that $V$ is almost closed under multiplication (cf. Proposition \ref{convex}):

\begin{lemma}  If $f, f' \in V$, then $ff'$ is the limit (in $L^2(\X)$, or equivalently in $L^{2^d/(2^d-1)}(\X)$) of a sequence in $V$ that is uniformly bounded in $L^\infty$ norm.
\end{lemma}

\begin{proof}  By linearity and density, as well as Lemma \ref{int} and \eqref{dstar}, we may assume that
$$f = \st {}^* {\mathcal D}^d_{\eps Q_1,\dots,\eps Q_d}( \mathbf{f}_\omega )_{\omega \in \{0,1\}^d \backslash \{0\}^d}$$
and
$$f' = \st {}^* {\mathcal D}^d_{\eps' Q_1,\dots,\eps' Q_d}( \mathbf{f}'_\omega )_{\omega \in \{0,1\}^d \backslash \{0\}^d}$$
for some standard $\eps,\eps' > 0$ and some $\mathbf{f}_\omega,\mathbf{f}'_\omega \in {\mathcal A}[\X]$, which we may take to be real-valued.  
We can thus write
$$ ff' = \st \E_{h_i \in \eps Q_i - \eps Q_i, h'_i \in \eps' Q_i - \eps' Q_i \forall i=1,\dots,d} \prod_{\omega \in \{0,1\}^d \backslash \{0\}^d} T^{\omega_1 h_1 + \dots + \omega_d h_d} \mathbf{f}_\omega T^{\omega_1 h'_1 + \dots + \omega_d h'_d} \mathbf{f}'_\omega.$$
Let $\eps_* > 0$ be infinitesimal.  We can shift $\eps Q_i-\eps Q_i$ or $\eps' Q_i - \eps'Q_i$ by an element of $\eps_* Q_i - \eps_* Q_i$ while only affecting the above average by $o(1)$.  We conclude that
\begin{align*} ff' &= \st \E_{k_i \in \eps_* Q_i - \eps_* Q_i, h_i \in \eps Q_i - \eps Q_i, h'_i \in \eps' Q_i - \eps' Q_i} \\
&\quad \prod_{\omega \in \{0,1\}^d \backslash \{0\}^d} T^{\omega_1 (h_1+k_1) + \dots + \omega_d (h_d+k_d)} \mathbf{f}_\omega T^{\omega_1 (h'_1+k_1) + \dots + \omega_d (h'_d+k_d)} \mathbf{f}'_\omega.
\end{align*}
Performing the $k_1,\dots,k_d$ average first, we conclude that
$$ ff' = \st \E_{h_i \in \eps Q_i - \eps Q_i, h'_i \in \eps' Q_i - \eps' Q_i} {}^* {\mathcal D}^d_{\eps_* Q_1,\dots,\eps_* Q_d}( (T^{\omega_1 h_1 + \dots + \omega_d h_d} \mathbf{f}_\omega T^{\omega_1 h'_1 + \dots + \omega_d h'_d} \mathbf{f}'_\omega)_{\omega \in \{0,1\}^d \backslash \{0\}^d} ).$$
This bound holds for all infinitesimal $\eps_*>0$.  By overspill, we conclude that for any standard $\delta > 0$ we have
$$
\| ff' - \st \E_{h_i \in \eps Q_i - \eps Q_i, h'_i \in \eps' Q_i - \eps' Q_i} {}^* {\mathcal D}^d_{\eps_* Q_1,\dots,\eps_* Q_d}( (T^{\omega_1 h_1 + \dots + \omega_d h_d} \mathbf{f}_\omega T^{\omega_1 h'_1 + \dots + \omega_d h'_d} \mathbf{f}'_\omega)_{\omega \in \{0,1\}^d \backslash \{0\}^d} ) \|_{L^\infty(\X)} \leq \delta$$
for all sufficiently small standard $\eps_* > 0$.  By Lemma \ref{conv}, the second term inside the norm is the limit in $L^2$ of a bounded sequence in $V$, and the claim follows.
\end{proof}

Let $W$ denote the space of all functions in $L^\infty(\X)$ which are the limit (in $L^{2^d/(2^d-1)}$) of a sequence in $V$ that is uniformly bounded in $L^\infty$ norm.  From the previous two lemmas we see that $W$ is a subspace of $L^\infty(\Y)$ that is closed under multiplication, and which is also closed with respect to limits in $L^{2^d/(2^d-1)}$ of sequences uniformly bounded in $L^\infty$.  If we let ${\mathcal Z}$ denote the set of all measurable sets $E$ in ${\mathcal Y}$ such that $1_E$ lies in $W$, we then see that ${\mathcal Z}$ is a $\sigma$-algebra.  From the translation invariance identity
$$
T^n {\mathcal D}^d_{\eps Q_1,\dots,\eps Q_d}( f_\omega )_{\omega \in \{0,1\}^d \backslash \{0\}^d}
= {\mathcal D}^d_{\eps Q_1,\dots,\eps Q_d}( T^n f_\omega )_{\omega \in \{0,1\}^d \backslash \{0\}^d}$$
we see that $V$, and hence $W$ and ${\mathcal Z}$, are invariant with respect to shifts $T^n$, $n \in G$.  If we set
$\ZZ^{{<}d}_{o(Q_1),\dots,o(Q_d)}(\Y) \coloneqq  (X, {\mathcal Z}, \mu)$, then $(\ZZ^{{<}d}_{o(Q_1),\dots,o(Q_d)}(\Y),T)$ is a factor of $(\Y, T)$ obeying properties (i) and (ii).

It remains to establish (iii).  Firstly, suppose that $f \in L^\infty(\Y)$ is such that
$\|f\|_{\Box^d_{o(Q_1),\dots,o(Q_d)}(\X)}>0$.  By Lemma \ref{lard}, we thus have
$$\|f\|_{\Box^d_{\eps Q_1,\dots,\eps Q_d}(\X)} > 0$$
for some standard $\eps>0$.  From\eqref{dual-ident} (and a limiting argument) we thus have
$$
\langle f, {\mathcal D}^d_{Q_1,\dots,Q_d}( f )_{\omega \in \{0,1\}^d \backslash \{0\}^d} \rangle_{L^2(\X)} > 0.$$
In particular, $f$ is not orthogonal to $V$, and hence to $L^\infty(\ZZ^{{<}d}_{o(Q_1),\dots,o(Q_d)}(\Y))$.

Conversely, suppose that 
$$\|f\|_{\Box^d_{o(Q_1),\dots,o(Q_d)}(\X)} = 0,$$
then by Lemma \ref{lard} we have
$$\|f\|_{\Box^d_{\eps Q_1,\dots,\eps Q_d}(\X)} = 0$$
for any standard $\eps>0$.
By Lemma \ref{int}, we may write $f = \st \mathbf{f}$ for some $\mathbf{f} \in {\mathcal A}[\X]$, then
$$\|\mathbf{f}\|_{{}^* \Box^d_{\eps Q_1,\dots,\eps Q_d}(\X)} = o(1).$$
By the (internal) Cauchy-Schwarz-Gowers inequality \eqref{csg}, we thus have
$$ \langle \mathbf{f}, {}^* {\mathcal D}^d_{\eps Q_1,\dots,\eps Q_d}( \mathbf{f}_\omega )_{\omega \in \{0,1\}^d \backslash \{0\}^d} \rangle_{{}^* L^2(\X)} = o(1)$$
for all $\mathbf{f}_\omega \in {\mathcal A}[\X]$.  We conclude that $f$ is orthogonal to $V$, and hence to $L^\infty(\ZZ^{{<}d}_{o(Q_1),\dots,o(Q_d)}(\Y))$.
This concludes the proof of Theorem \ref{charf}.

We record a basic consequence of Theorem \ref{charf} (cf. Corollary \ref{locx} or \cite[Proposition 4.6]{hk}):

\begin{corollary}[Localisation]\label{local} Let $\X$ be a Loeb $G$-system for some nonstandard additive group $G$, 
let $\Y$ be a factor of $\X$, and let $Q_1,\dots,Q_d$ be internal coset progressions of bounded rank in $G$ for some standard positive integer $d$.  Then
$$\ZZ^{{<}d}_{o(Q_1),\dots,o(Q_d)}(\Y) \equiv \Y \wedge \ZZ^{{<}d}_{o(Q_1),\dots,o(Q_d)}(\X).$$
\end{corollary}

\begin{proof}  If $f \in L^\infty( \ZZ^{{<}d}_{o(Q_1),\dots,o(Q_d)}(\Y))$, then by Theorem \ref{charf} $f \in L^\infty(\Y)$, and $f$ can be approximated by finite linear combinations of dual functions, which then also lie in $L^\infty(\ZZ^{{<}d}_{o(Q_1),\dots,o(Q_d)}(\X))$, giving one inclusion.  To obtain the other inclusion, observe that if $f \in L^\infty(\Y) \cap L^\infty(\ZZ^{{<}d}_{o(Q_1),\dots,o(Q_d)}(\X))$, then by Theorem \ref{charf} $f$ is orthogonal to all functions in $L^\infty(\X)$ with vanishing $\Box^d_{o(Q_1),\dots,o(Q_d)}(\X)$ norm, and by second application of Theorem \ref{charf} we thus have
$f \in L^\infty( \ZZ^{{<}d}_{o(Q_1),\dots,o(Q_d)}(\Y))$.
\end{proof}

\begin{remark} Following the informal sketch in Section \ref{sketch-sec}, functions that are measurable with respect to $\ZZ^{{<}d}_{o(Q_1),\dots,o(Q_d)}(\X)$ should be viewed as being ``structured'' along the directions $Q_1,\dots,Q_d$.  The factor $\Y$ represents some additional structure, perhaps along some directions unrelated to $Q_1,\dots,Q_d$.  Corollary \ref{local} then guarantees that this additional structure is preserved when one performs such operations as orthogonal projection to $L^2( \ZZ^{{<}d}_{o(Q_1),\dots,o(Q_d)}(\X) )$.
\end{remark}

For the purpose of concatenation theorems, the following special case of Corollary \ref{local} is crucial (cf. Corollary \ref{concat-uap-split}):

\begin{corollary}\label{local-concat} Let $(\X,T)$ be a Loeb $G$-system for some nonstandard additive group $G$, and let
 $Q_{1,1},\dots,Q_{1,d_1}, Q_{2,1},\dots,Q_{2,d_2}$ be internal coset progressions of bounded rank.  Then
\begin{align*}
\ZZ^{{<}d_1}_{o(Q_{1,1}),\dots,o(Q_{1,d_1})}(\X) \wedge \ZZ^{{<}d_2}_{o(Q_{2,1}),\dots,o(Q_{2,d_2})}(\X) &\equiv 
 \ZZ^{{<}d_1}_{o(Q_{1,1}),\dots,o(Q_{1,d_1})}( \ZZ^{{<}d_2}_{o(Q_{2,1}),\dots,o(Q_{2,d_2})}(\X)  ) \\
&\equiv \ZZ^{{<}d_2}_{o(Q_{2,1}),\dots,o(Q_{2,d_2})}( \ZZ^{{<}d_1}_{o(Q_{1,1}),\dots,o(Q_{1,d_1})}(\X)  ).
\end{align*}
Furthermore, the spaces $L^\infty( \ZZ^{{<}d_1}_{o(Q_{1,1}),\dots,o(Q_{1,d_1})}(\X) )$ and $L^\infty( \ZZ^{{<}d_2}_{o(Q_{2,1}),\dots,o(Q_{2,d_2})}(\X) )$ are relatively orthogonal, in the sense that if $f_1 \in L^\infty( \ZZ^{{<}d_1}_{o(Q_{1,1}),\dots,o(Q_{1,d_1})}(\X) )$ and $f_2 \in L^\infty( \ZZ^{{<}d_2}_{o(Q_{2,1}),\dots,o(Q_{2,d_2})}(\X) )$ are both orthogonal to $L^\infty( \ZZ^{{<}d_1}_{o(Q_{1,1}),\dots,o(Q_{1,d_1})}(\X) ) \cap 
L^\infty( \ZZ^{{<}d_2}_{o(Q_{2,1}),\dots,o(Q_{2,d_2})}(\X) )$, then they are orthogonal to each other.
\end{corollary}

\begin{proof} The first claim is immediate from Corollary \ref{local}.  For the second claim, note from the first claim and Theorem \ref{charf} that $f_1$ has vanishing $\Box^{d_2}_{o(Q_{2,1}),\dots,o(Q_{2,d_2})}(\X)$ norm, and the claim follows from a second application of Theorem \ref{charf}.
\end{proof}

\section{Proof of Theorems \ref{concat-uap-nonst} and \ref{concat-box-nonst}}\label{boxproof}

We now prove Theorem \ref{concat-box-nonst}.  Let $d_1, d_2, G, Q_{1,1}, \dots, Q_{2,d_2}, \X, T$ be as in that theorem.  By Theorem \ref{charf}, it suffices to show that $f,g$ are orthogonal whenever
$$ f \in L^\infty( \ZZ^{{<}d_1}_{o(Q_{1,1}),\dots,o(Q_{1,d_1})}(\X) ) \cap L^\infty( \ZZ^{{<}d_2}_{o(Q_{2,1}),\dots,o(Q_{2,d_2})}(\X) ) $$
and
\begin{equation}\label{fun}
 \|g\|_{\Box^{d_1 d_2}_{(o(Q_{1,i}+Q_{2,j}))_{1 \leq i \leq d_1; 1 \leq j \leq d_2}}(\X)} = 0.
\end{equation}
We first deal with the base case $d_1=d_2=1$.  In this case we have
$$ f \in L^\infty( \ZZ^{{<}1}_{o(Q_{1,1})}(\X)) \cap L^\infty( \ZZ^{{<}1}_{o(Q_{2,1})}(\X))$$
and
\begin{equation}\label{go}
 \|g\|_{\Box^1_{o(Q_{1,1}+Q_{2,1})}(\X)} = 0.
\end{equation}
We may normalise $f,g$ to lie in the closed unit ball of $L^\infty(\X)$.  Observe that for any $\mathbf{h} \in {\mathcal A}[\X]$, the dual function
$$
{}^* {\mathcal D}^1_{\eps Q_{1,1}} \mathbf{h} = \E_{h \in \eps Q_{1,1} - \eps Q_{1,1}} T^h \mathbf{h} $$
obeys the translation invariance
$$ \st T^n {}^* {\mathcal D}^1_{\eps Q_{1,1}} \mathbf{h}  = \st {}^* {\mathcal D}^1_{\eps Q_{1,1}} \mathbf{h}$$
for all $n \in o(Q_{1,1})$.  By Theorem \ref{charf}, Lemma \ref{int}, and \eqref{dstar}, we conclude that $f$ is invariant with respect to shifts $T^n$ with $n \in o(Q_{1,1})$.  Similarly for $n \in o(Q_{2,1})$, and so $f$ is invariant with respect to shifts $T^n$ with $n \in o(Q_{1,1} + Q_{2,1})$.  In particular, for any infinitesimal $\eps>0$, and using Lemma \ref{int} to write $f = \st \mathbf{f}$ and $g = \st \mathbf{g}$, we have
\begin{align*}
|\langle f, g \rangle_{L^2(\X)}| &= \st |\langle \mathbf{f}, \mathbf{g} \rangle_{{}^* L^2(\X)}| \\
&= \st |\E_{n \in \eps Q_{1,1} + \eps Q_{2,1}} \langle T^n \mathbf{f}, T^n \mathbf{g} \rangle_{{}^* L^2(\X)}| \\
&= \st |\langle \mathbf{f}, \E_{n \in \eps Q_{1,1} + \eps Q_{2,1}} T^n \mathbf{g} \rangle_{{}^* L^2(\X)}| \\
&\leq \st \|\E_{n \in \eps Q_{1,1} + \eps Q_{2,1}} T^n \mathbf{g} \|_{{}^* L^2(\X)} \\
&= \st \| \mathbf{g} \|_{{}^* \Box^1_{\eps Q_{1,1} + \eps Q_{2,1}}(\X)}\\
&= \| g \|_{\Box^1_{\eps Q_{1,1} + \eps Q_{2,1}}(\X)}.
\end{align*}
Applying Lemma \ref{lard} and \eqref{go}, we obtain $\langle f, g \rangle_{L^2(\X)} = 0$ as required.

Now  assume inductively that $d_1+d_2 > 2$, and the claim has already been proven for smaller values of $d_1+d_2$.
Write
$$
L^\infty( \ZZ^{{<}d_1}_{o(Q_{1,1}),\dots,o(Q_{1,d_1})}(\X) ) \cap L^\infty( \ZZ^{{<}d_2}_{o(Q_{2,1}),\dots,o(Q_{2,d_2})}(\X) ) = L^\infty(\ZZ)$$
for some factor $(\ZZ,T)$.  By Corollary \ref{local}, one has
$$
L^\infty(\ZZ) = L^\infty( \ZZ^{{<}d_1}_{o(Q_{1,1}),\dots,o(Q_{1,d_1})}(\ZZ) ) =  L^\infty( \ZZ^{{<}d_2}_{o(Q_{2,1}),\dots,o(Q_{2,d_2})}(\ZZ) ).$$
By Theorem \ref{charf}, one can thus approximate $f$ (in $L^2$) by finite linear combinations of functions of the form
$$ {\mathcal D}^{d_1}_{\eps Q_{1,1},\dots,\eps Q_{1,d_1}}( f_\omega )_{\omega \in \{0,1\}^{d_1} \backslash \{0\}^{d_1}} $$
with $f_\omega \in L^\infty(\ZZ)$ and $\eps>0$ standard.  Without loss of generality, we may thus assume that
$$ f = {\mathcal D}^{d_1}_{\eps Q_{1,1},\dots,\eps Q_{1,d_1}}( f_\omega )_{\omega \in \{0,1\}^{d_1} \backslash \{0\}^{d_1}}.$$
A similar approximation argument (using the multilinearity and continuity properties of ${\mathcal D}^{d_1}_{\eps Q_{1,1},\dots,\eps Q_{1,d_1}}$) allows to assume without loss of generality that for each $\omega \in \{0,1\}^{d_1} \backslash \{0\}^{d_1}$, one has
$$ f_\omega = {\mathcal D}^{d_2}_{\eps_\omega Q_{2,1},\dots,\eps_\omega Q_{2,d_1}}( f_{\omega,\omega'} )_{\omega' \in \{0,1\}^{d_2} \backslash \{0\}^{d_2}}$$
for some $f_{\omega,\omega'} \in L^\infty(\ZZ)$ and $\eps_\omega > 0$ standard.  We may normalise so that all the functions $f, f_\omega, f_{\omega,\omega'}$ lie in the closed unit ball of $L^\infty$.

By Lemma \ref{int}, we may write $f_{\omega,\omega'} = \st \mathbf{f}_{\omega,\omega'}$ for some $\mathbf{f}_{\omega,\omega'} \in {\mathcal A}[\X]$, which implies that $f_\omega = \st \mathbf{f}_\omega$ and $f = \st \mathbf{f}$ where
$$ \mathbf{f}_\omega \coloneqq  {}^* {\mathcal D}^{d_2}_{\eps_\omega Q_{2,1},\dots,\eps_\omega Q_{2,d_1}}( \mathbf{f}_{\omega,\omega'} )_{\omega' \in \{0,1\}^{d_2} \backslash \{0\}^{d_2}}$$
and
$$ \mathbf{f} \coloneqq  {}^* {\mathcal D}^{d_1}_{\eps Q_{1,1},\dots,\eps Q_{1,d_1}}( \mathbf{f}_\omega )_{\omega \in \{0,1\}^{d_1} \backslash \{0\}^{d_1}}.$$
From the definition of dual function we have the recursive identity
$$
\mathbf{f} = \E_{h \in \eps Q_{1,1} - \eps Q_{1,1}} T^h \mathbf{f}_{(1,0^{d_1-1})} {}^* {\mathcal D}^{d_1-1}_{\eps Q_{1,2},\dots,\eps Q_{1,d_1}}( 
\mathbf{f}_{(0,\omega)} T^h \overline{\mathbf{f}_{(1,\omega)}} )_{\omega \in \{0,1\}^{d_1-1} \backslash \{0\}^{d_1-1}}$$
(recalling the convention that the dual function in this expression is $1$ if $d_1-1=0$).  If we shift this by any $n \in G$, we obtain
$$
T^n \mathbf{f} = \E_{h \in \eps Q_{1,1} - \eps Q_{1,1}} T^{h+n} \mathbf{f}_{(1,0^{d_1-1})} {}^* {\mathcal D}^{d_1-1}_{\eps Q_{1,2},\dots,\eps Q_{1,d_1}}( 
T^n \mathbf{f}_{(0,\omega)} T^{h+n} \overline{\mathbf{f}_{(1,\omega)}} )_{\omega \in \{0,1\}^{d_1-1} \backslash \{0\}^{d_1-1}}$$
or equivalently
$$
T^n \mathbf{f} = \E_{h \in \eps Q_{1,1} - \eps Q_{1,1} + n}  \mathbf{c}_{n,h} \mathbf{g}_h,$$
where
\begin{equation}\label{gah-def}
\mathbf{g}_h \coloneqq T^{h} \mathbf{f}_{(1,0^{d_1-1})}
\end{equation}
and
\begin{equation}\label{sas}
 \mathbf{c}_{n,h} \coloneqq  {}^* {\mathcal D}^{d_1-1}_{\eps Q_{1,2},\dots,\eps Q_{1,d_1}}( 
T^n \mathbf{f}_{(0,\omega)} T^{h} \overline{\mathbf{f}_{(1,\omega)}} )_{\omega \in \{0,1\}^{d_1-1} \backslash \{0\}^{d_1-1}}.
\end{equation}
If $n \in o( Q_{1,1} )$, then $\eps Q_{1,1}+n$ differs from $Q_{1,1}$ by $o( |Q_{1,1}|)$ elements (since $Q_{1,1}$ has bounded rank); since all functions here are uniformly bounded, we thus have the representation
$$
T^n \mathbf{f} = \E_{h \in \eps Q_{1,1} - \eps Q_{1,1}} \mathbf{c}_{n,h} \mathbf{g}_h + o(1)$$
for $n \in o(Q_{1,1})$ (cf. \eqref{gah}).

A similar argument gives
$$ T^{n'} \mathbf{f}_{(1,0^{d_1-1})} = 
\E_{h' \in \eps Q_{2,1} - \eps Q_{2,1}} \mathbf{c}'_{n',h'} \mathbf{g}'_{h'} + o(1)
$$
for $n' \in o(Q_{2,1})$, where
$$ \mathbf{g}'_{h'} \coloneqq T^{h'} \mathbf{f}_{(1,0^{d_1-1}),(1,0^{d_2-1})}.$$
and
\begin{equation}\label{sas-2}
 \mathbf{c}'_{n',h'} \coloneqq  
{}^* {\mathcal D}^{d_2-1}_{\eps Q_{2,2},\dots,\eps Q_{2,d_2}}( 
T^{n'} \mathbf{f}_{(1,0^{d_1-1}),(0,\omega)} T^{h'} \overline{\mathbf{f}_{(1,0^{d_1-1}),(1,\omega)}} )_{\omega \in \{0,1\}^{d_2-1} \backslash \{0\}^{d_2-1}}.
\end{equation}
In particular, from \eqref{gah-def} one has
$$ T^{n'} \mathbf{g}_h = \E_{h' \in \eps Q_{2,1} - \eps Q_{2,1}} \mathbf{c}'_{n',h,h'} \mathbf{g}'_{h,h'} + o(1)$$
where
$$ \mathbf{c}'_{n',h,h'} \coloneqq T^h \mathbf{c}'_{n',h'}$$
and
$$ \mathbf{g}'_{h,h'} := T^h \mathbf{g}'_{h'};$$
compare with \eqref{tang}.
Putting this together, we see that
\begin{equation}\label{stam}
T^{n+n'} \mathbf{f} = \E_{h \in \eps Q_{1,1} - \eps Q_{1,1}} \E_{h' \in \eps Q_{2,1} - \eps Q_{2,1}} 
\mathbf{c}''_{n,n',h,h'} \mathbf{g}'_{h,h'}  + o(1) 
\end{equation}
for $n \in o(Q_{1,1})$ and $n' \in o(Q_{2,1})$, where
\begin{equation}\label{tchn}
 \mathbf{c}''_{n,n',h,h'} \coloneqq  (T^{n'} \mathbf{c}_{n,h}) \mathbf{c}'_{n',h,h'};
\end{equation}
compare with \eqref{tnf}.

Let $\kappa,\kappa' > 0$ be standard quantities to be chosen later.  From \eqref{fun} and Lemma \ref{lard}, we may find an infinitesimal $\delta>0$ such that
$$
 \|g\|_{\Box^{d_1 d_2}_{(\delta Q_{1,i}+\delta Q_{2,j})_{1 \leq i \leq d_1; 1 \leq j \leq d_2}}(\X)} < \kappa'.
$$
By Lemma \ref{int}, we may write $g = \st \mathbf{g}$ for some $\mathbf{g} \in {\mathcal A}[\X]$, thus by \eqref{staf} we have
\begin{equation}\label{geo}
 \|\mathbf{g}\|_{{}^* \Box^{d_1 d_2}_{(\delta Q_{1,i}+\delta Q_{2,j})_{1 \leq i \leq d_1; 1 \leq j \leq d_2}}(\X)} < \kappa'.
\end{equation}
By shift invariance and \eqref{stam}, we have
\begin{align*}
|\langle f, g \rangle_{L^2(\X)}| &= \st |\langle \mathbf{f}, \mathbf{g} \rangle_{{}^* L^2(\X)}| \\
&= \st | \E_{n \in \delta Q_{1,1}, n' \in \delta Q_{2,1}} \langle T^{n+n'} \mathbf{f}, T^{n+n'} \mathbf{g} \rangle_{{}^* L^2(\X)} |\\
&= \st |
\E_{h \in \eps Q_{1,1} - \eps Q_{1,1}} \E_{h' \in \eps Q_{2,1} - \eps Q_{2,1}} 
\langle \mathbf{g}'_{h,h'},\E_{n \in \delta Q_{1,1}, n' \in \delta Q_{2,1}}  \overline{\mathbf{c}''_{n,n',h,h'}} T^{n+n'} \mathbf{g} \rangle_{{}^* L^2(\X)}|.
\end{align*}
We abbreviate $\E_{h \in \eps Q_{1,1} - \eps Q_{1,1}} \E_{h' \in \eps Q_{2,1} - \eps Q_{2,1}}$ as $\E_{h,h'}$, and $\E_{n \in \delta Q_{1,1}, n' \in \delta Q_{2,1}}$ as $\E_{n,n'}$.
The functions $\mathbf{g}'_{h,h'}$ are bounded (almost everywhere) by $1$ by construction, so by the Cauchy-Schwarz inequality we have
\begin{align*}
|\langle f, g \rangle_{L^2(\X)}|^2 
&\leq \st \E_{h,h'} \| \E_{n,n'} \overline{\mathbf{c}''_{n,n',h,h'}} T^{n+n'} \mathbf{g} \|_{{}^* L^2(\X)}^2 \\
&\leq \st \E_{h,h'} \E_{n_1,n'_1} \E_{n_2,n'_2} |\langle \mathbf{c}''_{n_1,n'_1,h,h'} \overline{\mathbf{c}''_{n_2,n'_2,h,h'}}, T^{n_1+n'_1} G T^{n_2+n'_2} \mathbf{g} \rangle_{{}^* L^2(\X)}|.
\end{align*}
where $n_1,n'_1$ and $n_2,n'_2$ are subject to the same averaging conventions as $n,n'$.
Now we analyse the factor $\mathbf{c}''_{n_1,n'_1,h,h'} \overline{\mathbf{c}''_{n_2,n'_2,h,h'}}$.  From \eqref{sas} we have
$$ \st \mathbf{c}_{n,h} =  {\mathcal D}^{d_1-1}_{\eps Q_{1,2},\dots,\eps Q_{1,d_1}}( 
T^n f_{(0,\omega)} T^{h} \overline{f_{(1,\omega)}} )_{\omega \in \{0,1\}^{d_1-1} \backslash \{0\}^{d_1-1}}$$
for any $h,n \in G$.
Let us temporarily assume that $d_1>1$.  Since $f_{(0,\omega)}, f_{(1,\omega)}$ are measurable in $\ZZ^{{<}d_2}_{o(Q_{2,1}),\dots,o(Q_{2,d_2})}(\X)$, we see from Theorem \ref{charf} that $\st \mathbf{c}_{n,h}$ is measurable with respect to
$$ \ZZ^{{<}d_1-1}_{o(Q_{1,2}),\dots,o(Q_{1,d_1})}( \ZZ^{{<}d_2}_{o(Q_{2,1}),\dots,o(Q_{2,d_2})}(\X) ).$$
Applying Corollary \ref{local-concat} and Theorem \ref{charf}, we conclude that $\st \mathbf{c}_{n,h}$ is orthogonal to any function in $L^\infty(\X)$ that vanishes in either $\Box^{{<}d_1-1}_{o(Q_{1,2}),\dots,o(Q_{1,d_1})}$ or $\Box^{{<}d_2}_{o(Q_{2,1}),\dots,o(Q_{2,d_2})}$ norms.  By induction hypothesis and Theorem \ref{charf}, we conclude that $\st \mathbf{c}_{n,h}$ is also orthogonal to any function in $L^\infty(\X)$ that vanishes in $\Box^{{<}(d_1-1)d_2}_{(o(Q_{1,i}+Q_{2,j}))_{2 \leq i \leq d_1, 1 \leq j \leq d_2}}$ norm.  By the monotonicity property \eqref{mono} of the Gowers box norms, $\st \mathbf{c}_{n,h}$ is therefore orthogonal to any function in $L^\infty(\X)$ that vanishes in $\Box^{{<}d_1 d_2-1}_{(o(Q_{1,i}+Q_{2,j}))_{1 \leq i \leq d_1, 1 \leq j \leq d_2; (i,j) \neq (1,1)}}$.  By Theorem \ref{charf}, we conclude that $\st \mathbf{c}_{n,h}$ is measurable with respect to $\ZZ^{{<}d_1 d_2-1}_{(o(Q_{1,i}+Q_{2,j}))_{1 \leq i \leq d_1, 1 \leq j \leq d_2; (i,j) \neq (1,1)}}(\X)$.  This argument was established for $d_1>1$, but when $d_1=1$, $\mathbf{c}_{n,h}$ is simply $1$, and the claim is still true.

A similar argument starting from \eqref{sas-2} shows that $\st \mathbf{c}'_{n',h'}$ is also measurable with respect to $\ZZ^{{<}d_1 d_2-1}_{(o(Q_{1,i}+Q_{2,j}))_{1 \leq i \leq d_1, 1 \leq j \leq d_2; (i,j) \neq (1,1)}}(\X)$ for any $h',n' \in G$.  From \eqref{tchn} we conclude that
$$ \st \mathbf{c}''_{n_1,n'_1,h,h'} \overline{\mathbf{c}''_{n_2,n'_2,h,h'}} \in L^\infty( \ZZ^{{<}d_1 d_2-1}_{(o(Q_{1,i}+Q_{2,j}))_{1 \leq i \leq d_1, 1 \leq j \leq d_2; (i,j) \neq (1,1)}}(\X) )$$
and hence by Theorem \ref{charf} and Lemma \ref{lard} we have, for any infinitesimal $\nu>0$, that
$$
|\langle \mathbf{c}''_{n_1,n'_1,h,h'} \overline{\mathbf{c}''_{n_2,n'_2,h,h'}}, \mathbf{h} \rangle_{{}^* L^2(\X)}| \leq \kappa$$
for any $h,h',n_1,n'_1,n_2,n'_2 \in G$, whenever $\mathbf{h}$ is an internally measurable function bounded in magnitude by $1$ with 
\begin{equation}\label{sa}
\| \mathbf{h} \|_{{}^* \Box^{d_1d_2-1}_{(\nu Q_{1,i}+\nu Q_{2,j}))_{1 \leq i \leq d_1, 1 \leq j \leq d_2; (i,j) \neq (1,1)}}(\X)} \leq \nu.
\end{equation}
By the underspill principle, this claim must also hold for some standard $\nu>0$, which depends on $\kappa$ but (crucially) does not depend on $\kappa'$ or $\delta$.  If \eqref{sa} fails, then we have instead the trivial bound
$$
|\langle \mathbf{c}''_{n_1,n'_1,h,h'} \overline{\mathbf{c}''_{n_2,n'_2,h,h'}}, \mathbf{h} \rangle_{{}^* L^2(\X)}| \leq 1.$$
Inserting these bounds into the previous estimation of $|\langle f, g \rangle_{L^2(\X)}|$, we conclude that
$$
|\langle f, g \rangle_{L^2(\X)}|^2 \leq \kappa + 
\st \E_{n_1,n'_1} \E_{n_2,n'_2} 1_{A(n_1,n'_1,n_2,n'_2)} $$
where $A(n_1,n'_1,n_2,n'_2)$ is the (internal) assertion that
$$
\| T^{n_1+n'_1} \mathbf{g} T^{n_2+n'_2} \mathbf{g} \|_{{}^* \Box^{d_1d_2-1}_{(\nu Q_{1,i}+\nu Q_{2,j}))_{1 \leq i \leq d_1, 1 \leq j \leq d_2; (i,j) \neq (1,1)}}(\X)} > \nu.
$$
However, we have the identity
$$
 \|\mathbf{g}\|_{{}^* \Box^{d_1 d_2}_{(\delta Q_{1,i}+\delta Q_{2,j})_{1 \leq i \leq d_1; 1 \leq j \leq d_2}}(\X)}^{2^{d_1 d_2}}
= 
\E_{n_1,n'_1} \E_{n_2,n'_2}
 \|T^{n_1+n'_1} \mathbf{g} T^{n_2+n'_2} \mathbf{g}\|_{{}^* \Box^{d_1 d_2}_{(\delta Q_{1,i}+\delta Q_{2,j})_{1 \leq i \leq d_1; 1 \leq j \leq d_2; (i,j) \neq (1,1)}}(\X)}^{2^{d_1 d_2-1}}$$
which by \eqref{geo} and Lemma \ref{lard} gives
$$
\E_{n_1,n'_1} \E_{n_2,n'_2}
 \|T^{n_1+n'_1} \mathbf{g} T^{n_2+n'_2} \mathbf{g}\|_{{}^* \Box^{d_1 d_2}_{(\nu Q_{1,i}+\nu Q_{2,j})_{1 \leq i \leq d_1; 1 \leq j \leq d_2; (i,j) \neq (1,1)}}(\X)}^{2^{d_1 d_2-1}} \leq (\kappa')^{2^{d_1d_2}} + o(1).$$
From Markov's inequality, we conclude that
$$
\E_{n_1,n'_1} \E_{n_2,n'_2} 1_{A(n_1,n'_1,n_2,n'_2)} \leq \nu^{-2^{d_1d_2-1}} (\kappa')^{2^{d_1d_2}} + o(1) $$
and thus
$$
|\langle f, g \rangle_{L^2(\X)}|^2 \leq \kappa + \nu^{-2^{d_1d_2-1}} (\kappa')^{2^{d_1d_2}}.$$
Sending $\kappa'$ to zero, and then sending $\kappa$ to zero, we obtain $\langle f, g \rangle_{L^2(\X)}=0$ as required.  This completes the proof of 
Theorem \ref{concat-box-nonst}.

We now briefly discuss the proof of Theorem \ref{concat-uap-nonst}, which is proven very similarly to Theorem \ref{concat-box-nonst}.  With $d_1,d_2,G,Q_1,Q_2,\X,T$ as in that theorem, our task is to show that $f,g$ are orthogonal whenever
$$ f \in L^\infty( \ZZ^{{<}d_1}_{o(Q_1),\dots,o(Q_1)}(\X) ) \cap L^\infty( \ZZ^{{<}d_2}_{o(Q_2),\dots,o(Q_2)}(\X) ) $$
and
$$
 \|g\|_{U^{d_1+d_2-1}_{o(Q_1+Q_2)}} = 0.$$
The base case $d_1=d_2=1$ is treated as before, so assume inductively that $d_1+d_2>2$ and the claim has already been proven for smaller values of $d_1+d_2$.  As before, we reduce to the case where $f$ is expressible in terms of $f_\omega, f_{\omega,\omega'}, \mathbf{f}, \mathbf{f}_\omega, \mathbf{f}_{\omega,\omega'}$ as in the preceding argument (with $Q_{1,i} = Q_1$ and $Q_{2,j} = Q_2$), and again arrive at the representation \eqref{stam}.  Repeating the previous analysis of $\st \mathbf{c}_{n,h}$, but now using the inductive hypothesis for Theorem \ref{concat-uap-nonst} rather than Theorem \ref{concat-box-nonst} (and not using the monotonicity of Gowers box norms), we see that $\st \mathbf{c}_{n,h}$ is measurable with respect to $\ZZ^{{<}d_1+d_2-2}_{(o(Q_1+Q_2), \dots, o(Q_1+Q_2)}$ for any $n,h \in G$; similarly for $\st \mathbf{c}'_{n',h'}$ and $\st \mathbf{c}''_{n_1,n'_1,h,h'} \overline{\mathbf{c}''_{n_2,n'_2,h,h'}}$.  Repeating the previous arguments (substituting Gowers box norms by Gowers uniformity norms as appropriate), we conclude Theorem \ref{concat-uap-nonst}.

\section{Proof of qualitative Bessel inequality}\label{bes-sec}

We now prove Theorem \ref{besu}.  The proof of Theorem \ref{besu-box} is very similar and is left to the interested reader.  Our arguments will be parallel to those used to prove Corollary \ref{bescor}, but using a nonstandard limit rather than an ergodic limit.

It will be convenient to reduce to a variant of this theorem in which the index set $I$ has bounded size.  More precisely, we will derive Theorem \ref{besu} from

\begin{theorem}\label{besu-2}  Let $M \geq 1$, and let $Q_1,\dots,Q_M$ be a finite sequence of coset progressions $Q_i$, all of rank at most $r$, in an additive group $G$. Let $\X$ be a $G$-system, and let $d$ be a positive integer.  Let $f$ lie in the unit ball of $L^\infty(\X)$, and suppose that
$$
\frac{1}{M^2} \sum_{1 \leq i < j \leq M} \|f\|_{U^{2d-1}_{\eps Q_i+\eps Q_j}(\X)} \leq \eps $$
for some $\eps > 0$.  If $\eps$ is sufficiently small depending on $M,r,d$, then
$$
\frac{1}{M} \sum_{i =1}^M \|f\|_{U^{d}_{Q_i}(\X)} \leq 2 M^{-1/2^d}.
$$
\end{theorem}

To see how Theorem \ref{besu-2} implies Theorem \ref{besu}, we use a random sampling argument.  Let $I, Q_i, r, G, \X, T, d, f, \eps$ be as in Theorem \ref{besu}.  Let $M$ be a positive integer to be chosen later, and select $M$ indices $a_1,\dots,a_M$ uniformly and independently from $I$ (with replacement).  Then from linearity of expectation we have
$$
\E \sum_{1 \leq i < j \leq M} \|f\|_{U^{2d-1}_{\eps Q_{a_i}+\eps Q_{a_j}}(\X)} = \frac{M(M-1)}{2}
\frac{1}{|I|^2} \sum_{i,j \in I} \|f\|_{U^{2d-1}_{\eps Q_{a_i}+\eps Q_{a_j}}(\X)} $$
and hence by \eqref{bes-1} and Markov's inequality we have
\begin{equation}\label{samp}
\frac{1}{M^2} \sum_{1 \leq i < j \leq M} \|f\|_{U^{2d-1}_{\eps Q_i+\eps Q_j}(\X)} \leq \eps 
\end{equation}
with probability greater than $1/2$.  On the other hand, since $\|f\|_{U^d_{Q_i}(\X)}$ lies between $0$ and $1$, the expression $\frac{1}{M} \sum_{i =1}^M \|f\|_{U^{d}_{Q_{a_i}}(\X)}$ has mean $\frac{1}{|I|} \sum_{i \in I} \|f\|_{U^{d}_{Q_i}(\X)}$ and variance at most $1/M$.  By Chebyshev's inequality, we conclude that
\begin{equation}\label{samp-2}
\left|\frac{1}{M} \sum_{i =1}^M \|f\|_{U^{d}_{Q_{a_i}}(\X)} - \frac{1}{|I|} \sum_{i \in I} \|f\|_{U^{d}_{Q_i}(\X)}\right| \leq \frac{\sqrt{2}}{\sqrt{M}} 
\end{equation}
with probability at most $1/2$.  Thus there exists a choice of $a_1,\dots,a_M$ obeying both \eqref{samp} and \eqref{samp-2}.  Applying Theorem \ref{besu-2}, we conclude that
$$ \frac{1}{|I|} \sum_{i \in I} \|f\|_{U^{d}_{Q_i}(\X)} \leq 2 M^{-1/2^d} + \frac{\sqrt{2}}{\sqrt{M}}$$
if $\eps$ is sufficiently small depending on $M,r,d$.  The right-hand side goes to zero as $M \to \infty$; optimizing $M$ in terms of $\eps$, we obtain Theorem \ref{besu} as required.

It remains to prove Theorem \ref{besu-2}.  As in the proof of Theorems \ref{concat-uap} and \ref{concat-box}, it will be convenient to pass to a nonstandard formulation, in order to access the tool of conditional expectation with respect to a characteristic factor.  Suppose for sake of contradiction that Theorem \ref{besu-2} failed.  Then there exists $M, r, d$, a sequence of real numbers $\eps_n > 0$ going to zero, a sequence $G_n$ of additive groups, a sequence $(\X_n,T_n)$ of $G_n$-systems, and a sequence $f_n$ of elements of $L^\infty(\X_n)$, and sequences $Q_{1,n},\dots,Q_{M,n}$ of coset progressions in $G_n$ of rank at most $r$, such that
$$
\frac{1}{M^2} \sum_{1 \leq i < j \leq M} \|f_n\|_{U^{2d-1}_{\eps_n Q_{i,n}+\eps_n Q_{j,n}}(\X_n)} \leq \eps_n $$
but
$$
\frac{1}{M} \sum_{i =1}^M \|f_n\|_{U^{d}_{Q_{i,n}}(\X_n)} > 2 M^{-1/2^d}. 
$$
We take ultraproducts as before, forming an internal additive group $G$, a Loeb system $(\X,T)$, a function $f \in L^\infty(\X_n)$, an infinitesimal $\eps>0$, and internal coset progressions $Q_1,\dots,Q_M$ of bounded rank such that
\begin{equation}\label{moo}
\frac{1}{M^2} \sum_{1 \leq i < j \leq M} \|f\|_{U^{2d-1}_{\eps Q_{i}+\eps Q_{j}}(\X)} = 0
\end{equation}
and
\begin{equation}\label{oom}
\frac{1}{M} \sum_{i =1}^M \|f\|_{U^{d}_{Q_{i}}(\X)} \geq 2 M^{-1/2^d}.
\end{equation}
From \eqref{moo}, Lemma \ref{lard}, and Theorem \ref{charfac}, we see that $f$ is orthogonal to $L^\infty( \ZZ^{{<}2d-1}_{Q_i+Q_j,\dots,Q_i+Q_j}(\X) )$ for any $1 \leq i < j \leq M$.  By Theorem \ref{concat-uap-nonst} and Theorem \ref{charfac}, this implies that $f$ is orthogonal to
$$
L^\infty( \ZZ^{{<}d}_{Q_i,\dots,Q_i}(\X) ) \cap L^\infty( \ZZ^{{<}d}_{Q_j,\dots,Q_j}(\X) ).$$
Thus, if we define
$$ f_i \coloneqq  \E( f | \ZZ^{{<}d}_{Q_i,\dots,Q_i}(\X)  )$$
then the $f_1,\dots,f_M$ are pairwise orthogonal.  We now compute
\begin{align*}
\sum_{i=1}^M \|f_i\|_{L^2(\X)}^2 &= \sum_{i=1}^M \langle f, f_i \rangle_{L^2(\X)} \leq \| \sum_{i=1}^M f_i \|_{L^2(\X)} = \left(\sum_{i=1}^M \|f_i\|_{L^2(\X)}^2\right)^{1/2}
\end{align*}
and hence
$$ \sum_{i=1}^M \|f_i\|_{L^2(\X)}^2 \leq 1;$$
since $\|f_i\|_{L^\infty(\X)} \leq 1$, we conclude from Lemma \ref{lard}, Theorem \ref{charfac} and H\"older's inequality that
\begin{align*}
\frac{1}{M} \sum_{i=1}^M \|f\|_{U^d_{Q_i}(\X)}  
&\leq \frac{1}{M} \sum_{i=1}^M \|f\|_{U^d_{o(Q_i)}(\X)} \\
&= \frac{1}{M} \sum_{i=1}^M \|f_i\|_{U^d_{o(Q_i)}(\X)}  \\
&\leq \frac{1}{M} \sum_{i=1}^M \|f_i\|_{L^{2^d}(\X)} \\
&\leq \frac{1}{M} \sum_{i=1}^M \|f_i\|_{L^2(\X)}^{1/2^{d-1}} \\
 &\leq \left(\frac{1}{M} \sum_{i=1}^M \|f_i\|_{L^2(\X)}^2\right)^{1/2^d} \\
&\leq M^{-1/2^d},
\end{align*}
contradicting \eqref{oom}.  This completes the proof of Theorem \ref{besu}.  The proof of Theorem \ref{besu-box} is analogous (using Theorem \ref{concat-box-nonst} in place of Theorem \ref{concat-uap-nonst}) and is left to the reader.

\section{Controlling a multilinear average}\label{cprop-sec}

In this section we establish Proposition \ref{cprop}.    Let $N, M, C, f_1,\dots,f_4$ be as in that proposition.  By linearity we may take $f_1,f_2,f_3,f_4$ to be real-valued. Suppose that
$$
|A_{N,M}(f_1,f_2,f_3,f_4)| \geq \delta$$
for some $\delta > 0$.  It will suffice to show that
$$
\| f_i \|_{U^3_{\Z/N\Z}(\Z/N\Z)} \geq \eps $$
for all $i=1,2,3,4$ and some $\eps>0$ depending only on $\delta,C$.  

Fix $i$. We have
$$ \left|\E_{n,m,k \in [M]} \int_X f_1 T^{-nm} f_2 T^{-nk} f_3 T^{-n(m+k)} f_4\ d\mu\right| \geq \delta.$$
where $X$ is $\Z/N\Z$ with uniform measure $\mu$ and the usual shift map $T$.  Let $Q$ be the coset progression $Q \coloneqq  \{ a: |a| \leq \kappa \sqrt{N}\}$ for some small $\kappa>0$ to be chosen shortly.  If $\kappa$ is small enough depending on $\delta,C$, we can shift $m$ and $k$ in the above average by any elements $a'-a, b'-b$ of $Q-Q$ without changing the average by more than $\delta/2$.  From this and some rearranging and averaging we see that
$$ |\E_{n,m,k \in [M]} \E_{a,b,a',b' \in Q} \int_X T^{na+nb} f_1 T^{-nm+na+nb'} f_2 T^{-nk+na'+nb} f_3 T^{-n(m+k)+na'+nb'} f_4\ d\mu| \geq \delta/2.$$
From the Cauchy-Schwarz-Gowers inequality \eqref{csg} we have
$$ |\E_{a,b,a',b' \in Q} \int_X T^{na+nb} f_1 T^{-nm+na+nb'} f_2 T^{-nk+na'+nb} f_3 T^{-n(m+k)+na'+nb'} f_4\ d\mu|
\leq \| f_i \|_{U^2_{n \cdot Q}(\X)} $$
for any $i=1,\dots,4$, where $n \cdot Q \coloneqq  \{ na: a \in Q\}$.  We thus have
$$ \E_{n \in [M]} \| f_i \|_{U^2_{n \cdot Q}(\X)} \geq \delta/2.$$
For the rest of this section, we adopt the asympotic notation $X \ll Y$ or $Y \gg X$ to denote the bound $X \leq C'(C,\delta) Y$ for some $C'(C,\delta)$ depending on $C,\delta$ (or on $\kappa$, since $\kappa$ can be chosen to depend only on $\delta$ and $C$).  Applying Theorem \ref{besu} in the contrapositive, we thus have
\begin{equation}\label{mam}
 \E_{n,m \in [M]} \| f_i \|_{U^3_{n \cdot Q + m \cdot Q}(\X)} \gg 1.
\end{equation}
If $A$ is a sufficiently large quantity depending on $\delta$, we may thus find $M/A \leq n,m \leq M$ with greatest common divisor $(n,m)$ at most $A$ such that
\begin{equation}\label{fao}
 \| f_i \|_{U^3_{n \cdot Q + m \cdot Q}(\X)}  \geq \frac{1}{A},
\end{equation}
since the contribution to \eqref{mam} of all other pairs $(n,m)$ can be shown to be $O( 1/A )$.

Fix this choice of $n,m$.  By hypothesis, the multiset $n \cdot Q + m \cdot Q$ has cardinality $\gg N$ (counting multiplicity), and has multiplicity $\ll 1$ at each point, and hence $(n \cdot Q + m \cdot Q) - (n \cdot Q + m \cdot Q)$ has cardinality $\gg N^2$ and multiplicity $\ll N$ at each point.  From \eqref{fao} we thus have
\begin{align*} \left |\E_{x,a,b,c \in \Z/N\Z} \left(\prod_{\omega_1,\omega_2,\omega_3 \in \{0,1\}} f_i(x + \omega_1 a + \omega_2 b + \omega_3 c)\right) \nu_2(a) \nu_2(b) \nu_2(c)\right|  \gg 1
\end{align*}
where  $\nu_2(a)$ is $\frac{1}{N}$ times the multiplicity of $(n \cdot Q + m \cdot Q) - (n \cdot Q + m \cdot Q)$ at $a$; equivalently, one has
$$ \nu_2(a) = \E_{b \in \Z/N\Z} \nu(b) \nu(a-b)$$
where $\nu(a)$ is the multiplicity of $n \cdot Q + m \cdot Q$ at $a$.  By the previous discussion, we have
$$ \E_{a \in \Z/N\Z} \nu(a)^2 \ll 1$$
which by Fourier expansion and the Plancherel identity allows us to write
$$ \nu_2(a) = \sum_{\xi \in \Z/N\Z} c_\xi e^{2\pi i a \xi / N}$$
for some coefficients $c_\xi$ with $\sum_{\xi \in \Z/N\Z} |c_\xi| \ll 1$.  By the pigeonhole principle, we may thus find $\xi_1,\xi_2,\xi_3 \in \Z/N\Z$ such that
$$ \left|\E_{x,a,b,c \in \Z/N\Z} \left(\prod_{\omega_1,\omega_2,\omega_3 \in \{0,1\}} f_i(x + \omega_1 a + \omega_2 b + \omega_3 c)\right) e^{2\pi i (\xi_1 a + \xi_2 b + \xi_3 c)/N}\right| \gg 1.$$
Writing $e^{2\pi i \xi_1 a/N} = e^{2\pi i \xi_1 (x+a)/N} e^{-2\pi i x/N}$, and similarly for the other two phases in the above expression, we thus have
$$ |\E_{x,a,b,c \in \Z/N\Z}  f_{i,0}(x) f_{i,1}(x+a) f_{i,2}(x+b) f_{i,3}(x+c) f_i(x+a+b) f_i(x+a+c) f_i(x+b+c) f_i(x+a+b+c)| \gg 1$$
for some functions $f_{i,0}, f_{i,1}, f_{i,2}, f_{i,3}\colon \Z/N\Z \to \C$ bounded in magnitude by $1$.  Applying the Cauchy-Schwarz-Gowers inequality \eqref{csg}, we conclude that
$$ \|f_i \|_{U^3_{\Z/N\Z}(\Z/N\Z)} \gg 1$$
and the claim follows.

\section{Reducing the complexity}\label{cprop2-sec}

In this section we sketch how the $U^3$ norm in Proposition \ref{cprop} can be lowered to $U^2$.  

The situation here is closely analogous to that of \cite[Theorem 7.1]{gt-reg}, and we will assume familiarity here with the notation from that paper.  That theorem was proven by combining an ``arithmetic regularity lemma'' that decomposes an arbitrary function into an ``irrational'' nilsequence, plus a uniform error, together with a ``counting lemma'' that controls the contribution of the irrational nilsequence.  The part of the proof of \cite[Theorem 7.1]{gt-reg} that involves the regularity lemma goes through essentially unchanged when establishing Proposition \ref{cprop}; the only difficulty is to establish a suitable counting lemma for the irrational nilsequence, which in this context can be taken to a nilsequence of degree $\leq 2$.  More precisely, we need to show

\begin{proposition}[Counting lemma]\label{stairs}  Let $(G/\Gamma,G_\bullet)$ be a filtered nilmanifold of degree $\leq 2$ and complexity $\leq M$ (as defined in \cite[Definitions 1.3,1.4]{gt-reg}.  Let $g\colon \Z \to G$ be an $(A,N)$-irrational polynomial sequence adapted to $G_\bullet$ (as defined in \cite[Definition A.6]{gt-reg}).  Let $F\colon (G/\Gamma)^4 \to \R$ be a function of Lipschitz norm at most $M$, and assume that $F$ has mean zero along $G_{(2)}^4$ in the sense that
\begin{equation}\label{fao2}
 \int_{g_4 G_{(2)}^4/\Gamma_{(2)}^4} F = 0
\end{equation}
for all $g_4 \in G^4$.  Suppose that $C^{-1} \sqrt{N} \leq M_1 \leq C \sqrt{N}$ for some $C>0$.  Then 
$$ \E_{x \in [N]} \E_{n,m,k \in [M_1]} F( (g(x), g(x+nm), g(x+nk), g(x+n(m+k))) \Gamma^4 ) = o_{A \to \infty; C,M}(1) + o_{N \to \infty; C,M}(1),$$
using the asymptotic notation conventions from \cite{gt-reg}.
\end{proposition}

By repeating the proof of \cite[Theorem 7.1]{gt-reg}, using the above proposition as a substitute for \cite[Theorem 1.11]{gt-reg}, one obtains Theorem \ref{cprop-2}.  In more detail: we first use the Cauchy-Schwarz-Gowers inequality to reduce to the case $f_1=f_2=f_3=f_4=f$.  Then we split $f$ using the regularity lemma as $f = f_{\operatorname{nil}} + f_{\operatorname{sml}} + f_{\operatorname{unf}}$, where $f_{\operatorname{nil}}$ is a nilsequence of degree $\leq 2$, $f_{\operatorname{small}}$ is small in $L^2$, and $f_{\operatorname{unf}}$ is very small in $U^3$ norm.  As in \cite[\S 7]{gt-reg}, the contributions of the latter two terms are easily seen to be small, and by partitioning $[N]$ into arithmetic progressions, one can essentially reduce to the case where the nilsequence $f_{\operatorname{nil}}$ is highly irrational.  Writing each $f_{\operatorname{nil}}(n) = F(g(n)\Gamma)$, the hypothesis that $f$ is small in $U^2$, together with \cite[Proposition 7.2]{gt-reg} show that $F$ has essentially mean zero along $G_{(2)}$, in the sense that $\int_{g_4 G_{(2)}/\Gamma_{(2)}} F$ is close to zero for most $g_4$.  By modifying $F$ slightly we may assume that $F$ has \emph{exactly} mean zero, in which case one can use Proposition \ref{stairs} to conclude the argument. 

It remains to establish Proposition \ref{stairs}.  Suppose that
$$ |\E_{x \in [N]} \E_{n,m,k \in [M_1]} F( (g(x), g(x+nm), g(x+nk), g(x+n(m+k))) \Gamma^4 )| \geq \eps$$
for some $\eps>0$; our objective is to obtain a contradiction when $A$ and $N$ are sufficiently large depending on $\eps,C,M$.  In this section we use $X \ll Y$, $Y \gg X$, or $X=O(Y)$ to denote the bound $|X| \leq C'(\eps,C,M) Y$ for some $C'(\eps,C,M)$ depending on $\eps,C,M$.  As the quantity $|\E_{x \in [N]} \E_{m,k \in [M_1]} F( (g(x), g(x+nm), g(x+nk), g(x+n(m+k))) \Gamma^4 )|$ is $O(1)$ for any $n$, we have
\begin{equation}\label{map}
|\E_{x \in [N]} \E_{m,k \in [M_1]} F( (g(x), g(x+nm), g(x+nk), g(x+n(m+k))) \Gamma^4 )| \gg 1
\end{equation}
for $\gg M_1$ choices of $n \in [M_1]$.  On the other hand, using the linear forms $\Psi\colon (x,m,k) \mapsto (x,x+m,x+k,x+m+k)$, one can compute (in the notation of \cite[Definition 1.10]{gt-reg})	that the Leibman group $G^\Psi$ (which is also the second Host-Kra group of $G$) contains $G_{(2)}^4$, basically because the quadratic polynomials $x^2$, $(x+m)^2$, $(x+k)^2$, $(x+m+k)^2$ are linearly independent, and so the linear space $\Psi^{[2]}$ appearing in \cite[Definition 1.10]{gt-reg} is all of $\R^4$.  Applying \eqref{fao2}, we conclude that
$$ \int_{g_4 G^\Psi/\Gamma^\Psi} F = 0$$
for any $g_4 \in G^4$.  From this and a change of variables, we see from \eqref{map} and \cite[Theorem 1.11]{gt-reg} that for $\gg M_1$ choices of $n \in [M_1]$, there exists an integer $a$ such that the map $g_{n,a}\colon m \mapsto g(nm+a)$ is \emph{not} $(A', M_1)$-irrational, for some $A' = O(1)$.  Unpacking the definition from \cite[Definition A.6]{gt-reg}, this means that there is an $i=1,2$ and an $i$-horizontal character $\xi_i\colon G_{(i)} \to \R$ of complexity $O(1)$ such that
$$ \| \xi_i( g_{n,a,i} )\|_{\R/\Z} \ll M_1^{-i} $$
for $\gg M_1$ choices of $n \in [M_1]$, 
where $g_{n,a,i}$ is the $i^{\operatorname{th}}$ Taylor coefficient of $g_{n,a}$.  But we have the identity
$$ \xi_i( g_{n,a,i} ) = n^i \xi_i( g_i ) $$
(see the proof of \cite[Lemma A.8]{gt-reg})
and hence
$$ \| n^i \xi_i( g_i )\|_{\R/\Z} \ll M_1^{-i} $$
for $\gg M_1$ choices of $n \in [M_1]$.  Applying the Vinogradov lemma (see e.g. \cite[Lemma 3.2]{gt-mobius}), we conclude that there exists a natural number $q=O(1)$ such that
$$ \| q \xi_i(g_i) \|_{\R/\Z} \ll M_i^{-2i} \ll N^{-i}.$$
But this contradicts the $(A,N)$-irrationality of $g$, and the claim follows.



\begin{dajauthors}
\begin{authorinfo}[tt]
  Terence Tao\\
  Department of Mathematics, UCLA\\
  405 Hilgard Ave\\
  tao\imageat{}math.ucla.edu\\
	\url{https://www.math.ucla.edu/~tao}
\end{authorinfo}

\begin{authorinfo}[tz]
  Tamar Ziegler\\
	Einstein Institute of Mathematics \\
Edmond J. Safra Campus, Givat Ram \\
The Hebrew University of Jerusalem \\
Jerusalem, 91904, Israel \\
  tamarz\imageat{}math.huji.ac.il\\
	\url{http://www.ma.huji.ac.il/~tamarz/}
\end{authorinfo}
\end{dajauthors}

\end{document}